\documentclass{amsart}
\usepackage{geometry}                
\usepackage{graphicx}
\usepackage{amssymb}
\usepackage{epstopdf}
\usepackage{color}
\usepackage{enumitem,mathrsfs}
\usepackage[normalem]{ulem}
\usepackage{mathtools}
\usepackage{hyperref}

\theoremstyle{plain}
\newtheorem{theorem}{Theorem}
\newtheorem*{theorem*}{Theorem}
\newtheorem{lemma}{Lemma}
\newtheorem{corollary}{Corollary}
\newtheorem{proposition}{Proposition}

\theoremstyle{definition}
\newtheorem{definition}{Definition}
\newtheorem{example}{Example}

\theoremstyle{remark}
\newtheorem{remark}{Remark}

\font\myfont=cmr10 at 14pt
\newcommand{\e}{{\text{\myfont e}}}

\newcommand{\MP}{\ensuremath{{\mathcal P}}}
\newcommand{\MZ}{\ensuremath{{\mathcal Z}}}
\newcommand{\ZZ}{\ensuremath{{\mathbb Z}}}
\newcommand{\CC}{\ensuremath{{\mathbb C}}}
\newcommand{\CW}{\ensuremath{{\widehat{\mathbb C}}}}
\newcommand{\RR}{\ensuremath{{\mathbb R}}}
\newcommand{\NN}{\ensuremath{{\mathbb N}}}

\newcommand{\HH}{\ensuremath{{\mathbb H}}}
\newcommand{\E}{\ensuremath{{\mathscr E}}}
\newcommand{\R}{\ensuremath{{\mathcal R}}}
\newcommand{\ES}{\ensuremath{{E}}}
\newcommand{\CS}{\ensuremath{{\mathcal C}}}
\newcommand{\abs}[1]{\left\lvert#1\right\rvert}
\renewcommand{\Re}[1]{{\mathfrak{Re}\left(#1\right)}}
\renewcommand{\Im}[1]{{\mathfrak{Im}\left(#1\right)}}
\renewcommand{\bar}[1]{{\overline{#1}}}
\newcommand{\del}[2]{\frac{\partial #1}{\partial #2}}
\newcommand{\argp}[1]{\ensuremath{\arg\left(#1\right)}}

\begin{document}

\title[Geometry, flows \& visualization of singular complex analytic vector fields]
{On the geometry, flows and visualization 
of singular complex analytic vector fields on Riemann surfaces
}

\author[A.~Alvarez--Parrilla, J.~Muci\~no--Raymundo et al]{Alvaro Alvarez--Parrilla}
\address{Grupo Alximia SA de CV, 
M\'exico}
\email{alvaro.uabc@gmail.com}
\thanks{This work partially funded by UABC grant 0196.}

\author[]{Jes\'us Muci\~no--Raymundo}
\address{Centro de Ciencias Matem\'aticas, Universidad Nacional Aut\'onoma de M\'exico, M\'exico}
\email{muciray@matmor.unam.mx}

\author[]{Selene Solorza--Calder\'on}
\address{Facultad de Ciencias, Universidad Aut\'onoma de Baja California, 
M\'exico}
\email{selene.solorza@uabc.edu.mx}

\author[]{Carlos Yee--Romero}
\address{Facultad de Ciencias, Universidad Aut\'onoma de Baja California, 
M\'exico}
\email{carlos.yee@uabc.edu.mx}


\begin{abstract}
Motivated by the wild behavior of isolated essential singularities 
in complex analysis, we study 
singular complex analytic vector fields $X$ on arbitrary 
Riemann surfaces $M$. 
By vector field singularities we understand zeros, 
poles, isolated essential singularities and accumulation
points of the above kind.

\noindent 
In this framework, 
a singular analytic vector field $X$ has canonically associated;
a 1--form, 
a quadratic differential, 
a flat metric (with a geodesic foliation),
a global distinguished parameter or $\CC$--flow box $\Psi_X$, 
a Newton map $\Phi_X$, and  
a Riemann surface $\R_X$ arising from the maximal $\CC$--flow of $X$. 

\noindent 
We show that every singular complex analytic vector 
field $X$ on a Riemann surface is in fact both a 
global pullback of the 
constant vector field under $\Psi_X$
and of the radial vector field  on the sphere under $\Phi_X$.

\noindent 
As a result of independent interest, we show that 
the maximal analytic continuation 
of the a local $\CC$--flow of $X$ is univalued on the 
Riemann surface $\R_{X} \subset M \times \CC_t$, 
where $\R_X$ is the 
graph of $\Psi_{X}$.

\noindent 
Furthermore we explore the geometry of singular complex 
analytic vector fields and present a geometrical method 
that enables us to obtain the solution, without numerical integration, 
to the differential equation that provides the $\CC$--flow of the vector field. 

\noindent 
We discuss the theory behind the method, its implementation, 
comparison with some integration--based techniques, 
as well as examples of the visualization of complex vector 
fields on the plane, sphere and torus. 

\noindent 
Applications to visualization of complex valued functions is 
discussed including some advantages between other methods.
\end{abstract}



\maketitle

\tableofcontents

\section{Statement of the results}\label{intro}
Vector fields related to complex analytic functions are very interesting and useful mathematical objects, 
both from the point of view of pure mathematics as from that of applications. 
They arise in multiple contexts: many physical phenomena can be modelled by vector fields 
(electric fields, magnetic fields, velocity fields, to name a few); and there are many interesting 
applications concerning the geometry and dynamics associated to them 
(\cite{AP-2}, \cite{AP-MR}, \cite{Benzinger}, \cite{DedieuShub}, \cite{HirschSmale}, 
\cite{HR}, \cite{MR}, \cite{MR-VV}, \cite{NewtonLofaro}, \cite{PalmoreBurnsBenzinger}, 
\cite{Smale1}). 
Moreover, the visualization of vector fields, besides being beautiful in itself, can be of great help 
towards the understanding of certain theoretical concepts. In particular, it can be used for the 
visualization of complex functions, which in of itself is a non--trivial problem (\cite{NewtonLofaro}, 
\cite{Needham}, \cite{Braden}, \cite{Braden2}, \cite{Gluchoff}, \cite{PoelkePolthier}, \cite{Kawski}, 
\cite{Frederick}, \cite{Lundmark}).

The main objects of study of this work are 
\emph{singular complex analytic vector fields} 

\centerline{$X(z)=f_{\tt j}(z)\del{}{z}$ \ \ on Riemann surfaces $M$,}

\smallskip
\noindent where ${\tt j}$ refers to the local charts of $M$, 
connected but non necessarily compact.
The singular set $Sing(X)$ can admit zeros, poles, 
essential singularities and 
accumulation points of the above kind of points
(this is the meaning of the adjetive ``singular'').
Very roughly speaking, 
by the flow of $X$ we understand the (local)
$\CC$--flow, 
and 
since $\RR \subset \CC$,
by trajectories of $X$ we understand the trajectories that arise from the (local) $\RR$--flow of $\Re{X}$.
More precisely, 
the differential equation
\begin{equation}\label{ecdifflow}
\left\{
\begin{aligned}
\dot{z}(\tau)&=f_{\tt j}(z(\tau))\\
z(0)&=z_{0}
\end{aligned}
\right.
\ \ \
\text{for} 
\ \ \ 
z(\tau):(\tau_{min} , \tau_{max}) \subset \RR
\longrightarrow M ,
\end{equation}
gives rise to the local real flow $z(\tau)$ of the singular complex analytic vector field 
$X(z)=f_{\tt j}(z)\del{}{z}$.
The real trajectories in \eqref{ecdifflow}
are simply called \emph{trajectories of $X$}.

\medskip 
One can ask the following naive question:

\smallskip
\begin{center}
\emph{What is a singular complex analytic vector field $X$ on a 
Riemann surface $M$ 
\\
and how explicitly can we describe it?}
\end{center}
\smallskip

\noindent
An answer to this question is explored 
in \cite{MR-VV}, \cite{MR}, \cite{AP-MR}, \cite{AP-MR-2}, \cite{AP-MR-3}, and 
references therein.
In these works, the authors introduce as 
a main tool the following dictionary/correspondence between
different singular complex analytic objects, 
providing a rich geometric structure.

\smallskip 

\noindent
\textbf{Singular complex analytic dictionary.}

\noindent
\emph{On any Riemann surface $M$ there exist a one--to--one correspondence between:}
{\it
\begin{enumerate}[label=\arabic*)]
\item Singular complex analytic vector fields $X$.

\item Singular complex analytic differential 1--forms $\omega_{X}$, 
related to $X$ via 
$\omega_{X} (X) \equiv 1$.

\item Singular complex analytic orientable 
quadratic differential forms $\omega_{X} \otimes\omega_{X}$.

\item Singular (real) analytic flat structures $g_{X}$ 
associated to the quadratic differentials  $\omega_{X}\otimes \omega_{X}$,
with suitable singularities, provided with a real geodesic vector field $\Re{X}$.

\item Singular complex analytic 
(possibly multivalued) maps, \emph{distinguished parameters},
$$
\Psi_{X} (z)= \int^z_{z_0} \omega_X : M\longrightarrow \CW_t,
$$
for $z_{0}\in M$ not a zero of $X$, isolated essential singularity of $X$ or accumulation point of the 
above\footnote{
By a careful analysis and suitable choices, the domain of $\Psi_X$ can be considered to be the whole of $M$, 
likewise for the the choice of $z_0$. This is a delicate matter, see Remark \ref{comentariosCorrrespondencia}.3.
}.

\item Singular complex analytic 
(possibly multivalued) \emph{Newton maps}
$$
\Phi_{X}(z)=\exp\left[-\int^z_{z_0} \omega_X\right] : M\longrightarrow \CW_w,
$$
for $z_{0}\in M$ not a zero of $X$, isolated essential singularity of $X$ or accumulation point of the 
above\footnote{
Similarly, the domain of $\Phi_X$ and the choice of $z_0$ can be extended to be the whole of $M$.
}.

\item The pairs $\big(\R_{X},\pi^{*}_{X,2}(\del{}{t})\big)$ 
consisting of ramified Riemann surfaces 
$\R_{X} \subset M \times \CW_t$, 
associated to the maps $\Psi_{X}$, and the vector fields $\pi^{*}_{X,2}(\del{}{t})$ under the projection 
$\pi_{X,2}: \R_{X} \longrightarrow \CW_{t}$.

\item 
The pairs $(\Omega_X, \mathcal{F})$
consisting of maximal domains 
$\Omega_{X} \subset M \times \CW$ 
of the complex flows of $X$ and holomorphic foliations 
$\mathcal{F}$ whose leaves
are copies of the Riemann surface $\R_X$.
\end{enumerate}
}

Let us write diagrammatically 
the correspondence as 

\begin{center}
\begin{picture}(200,138)(0,-40)
\put(60,84){$X(z)=f_{\tt j}(z)\frac{\partial}{\partial z} $}
\put(4,53){$\omega_X (z)= \frac{dz}{f_{\tt j}(z)}$}
\put(120,53){$\Psi_X (z)= \int\limits^z \omega_X$}
\put(50,-35){$\big((M\backslash\mathcal{A} ,g_X ), \Re X\big)$}

\put(58,80){\vector(-1,-1){15}}
\put(44,66){\vector(1,1){15}}

\put(152,66){\vector(-1,1){15}}
\put(138,80){\vector(1,-1){15}}

\put(58,-24){\vector(-1,1){15}}
\put(44,-10){\vector(1,-1){15}}

\put(151,-10){\vector(-1,-1){15}}
\put(138,-23){\vector(1,1){15}}

\put(5,0){$\omega_X \otimes \omega_X (z)$}
 \put(128,0){$(\R_X, \pi^{*}_{X,2} (\frac{\partial}{\partial t}))$}

\put(36,37){\vector(0,-1){20}}
\put(36,22){\vector(0,1){20}}

\put(155,37){\vector(0,-1){20}}
\put(155,22){\vector(0,1){20}}

\put(-104,15){\vbox{\begin{equation}\label{diagramacorresp}\end{equation}}}
\end{picture}
\end{center}

\noindent
here the subindex $X$ means the dependence on the original vector field, 
in all that follows we omit it when it is unnecessary.

The detailed statement and proof of (1)--(5) and (7) of the above dictionary can be found 
as lemma 2.6 of \cite{AP-MR}. A preliminary study of (8) is found as lemma 2.3 of \cite{AP-MR-2}.

\noindent
The unification of (1)--(5) arrises from the idea of (local) distinguished parameters near regular points,
see for instance 
\cite{Jenkins} \S 3.1 and
\cite{Strebel} pp. 20--21. 
However in \cite{AP-MR}, \cite{AP-MR-2}, \cite{AP-MR-3} 
and this work, {\it we exploit the global nature of the maps $\Psi_{X}(z)$ and $\Phi_{X}(z)$} 
in the 1--dimensional case. 
In \cite{Bustundy-Giraldo-MR}, the global nature of $\Psi$ in the $n$--dimensional case is also explored.

\smallskip
\noindent
In the present work, we explore and exploit items (6) and (8) of the dictionary.

\smallskip

For item (6) of the dictionary, following the ideas of S.\,Smale {\it et al.} \cite{HirschSmale}, \cite{Smale1}, 
on \emph{Newton vector fields}, in \S\ref{metodo} we obtain two results:

\noindent
A visualization scheme for 
vector fields $X$.

\begin{theorem}[Visualization of singular complex analytic vector fields]
\label{TeoremaVisual}
Let $X(z)=f(z)\del{}{z}$ be a singular complex analytic vector field on a Riemann surface $M$, 
and let $z(\tau)$ denote any trajectory of $X$ on $M\backslash Sing(X)$.
Then there exist two (probably multivalued) functions $\rho,\theta:M\backslash Sing(X)\longrightarrow \RR$ 
such that
\begin{enumerate}[label=\arabic*)]
\item The $\rho$ is constant along $z(\tau)$, {\it i.e.} 
\ $\rho(z(\tau))= \rho(z(0))$.

\item The $\theta$ defines the natural time parametrization, {\it i.e.} 
\ $\theta(z(\tau))=\tau + \theta(z(0))$.
\end{enumerate}
\end{theorem}
\noindent
Secondly as a counterpart, 
for singular complex analytic functions 
$\Psi_X$ and $\Phi_X$. 

\begin{theorem}[Visualization of singular complex analytic functions]
\label{visualizationPsiPhi}
\hfill
\begin{enumerate}[label=\arabic*)]
\item Let $\Psi:M\longrightarrow\CW$ be a singular complex analytic function.
Then the phase portraits of

$X(z)=\frac{1}{\Psi^{\prime}(z)}\del{}{z}$ \
provides the level curves of 
$-\Im{\Psi}$,

$X^{\perp}(z) \doteq \frac{i}{\Psi^{\prime}(z)}\del{}{z}$ \ 
provides the level curves of $\Re{\Psi}$.

\item
Let $\Phi:M\longrightarrow\CW$ be a singular complex analytic function.
Then the phase portraits of 

$X(z)=-\frac{\Phi(z)}{\Phi^{\prime}(z)}\del{}{z}$ \
provides the level curves of  $\argp{\Phi}$, 

$X^{\perp}(z) \doteq -i\frac{\Phi(z)}{\Phi^{\prime}(z)}\del{}{z}$ \ 
provides the level curves of $-\log{\abs{\Phi}}$.
\end{enumerate}
\end{theorem}

In order to prove
item (8) of the dictionary, the local $\CC$--flow of $X$ a
singular complex analytic vector field is 
holomorphic at its zeros. However,
note that maximal domain $\Omega_{X}$ of the complex flows of 
$X$ are non--trivial at 
the poles, essential isolated singularities or accumulation
points of the above.
In fact one may ask the question:
\smallskip 

\begin{center}
\emph{
Considering the maximal analytic continuation of the local flows,
\\
what kind 
of structure will the maximal analytic continuation have?}
\end{center}

\smallskip
\noindent
Denoting $M^*$ as $M$ minus poles, essential singularities
and accumulation points of the above kind, 
a detailed analysis of $\big(\R_{X},\pi^{*}_{X,2}(\del{}{t})\big)$ 
in \S\ref{subsec:Flows} shows that in fact: 

\begin{theorem}[Maximal domain for the flow]
\label{FlujoMaximal}
Let $X$ be a singular complex analytic vector field on a 
Riemann surface $M$, and let $z_0 \in M \backslash Sing(X)$
be an initial condition. 
\begin{enumerate}[label=\arabic*)]
\item 
The maximal analytic continuation 
of the local flow 

\centerline{$\varphi_{\tt j} (z_0,t) 
: \{ z_0 \} \times ( \CC_t, 0)  \longrightarrow  M^*$}

\noindent 
is univalued on the 
Riemann surface $\R_{X} \subset M \times \CC_t$, which is the 
graph of 
$$
\Psi_{X} (z) = \int_{z_0}^z \omega_X : M^* \longrightarrow \CC_t .
$$ 

\item 
The Riemann surface $\R_X$ is a leaf of the  foliation 
$\mathcal{F}$ defined by the
complex analytic vector field
$$ 
f_{\tt j} (z) \del{}{z} + \del{}{t} \ \ \
 \hbox{ on }  M^* \times \CC_{t}
$$ 
and the changes of the initial conditions $z_0$ 
determine $t$--translations of $\R_X$.
\end{enumerate}
\end{theorem}

\noindent
The study of maximal domains of the flow from the viewpoint of 
complex differential equations is a deep current subject, 
see \cite{Loray}, \cite{Guillot1}, \cite{Guillot3} and references therein.

\smallskip 

Sections \ref{antecedentes}, \ref{pullbackHoloVF}, \ref{camposnewtonianos}, \ref{fund-obs}, 
and \ref{subsec:Flows}
are of theoretical flavor and 
familiarity with Riemann surface theory is recomended. 
Sections \ref{implementation}, \ref{parrallelization}, \ref{meromorphic}, \ref{essential} and \ref{comparison}
are of numerical character, 
which might be of interest for numerical experimentation
or software development.
Section \ref{generalizaciones} provides a panoramic view of possible extensions
to other frameworks. 
Section \ref{visualizingcomplexfunctions} deals with functions and only requires 
elementary Complex Analysis. 

We thank Coppelia Cerda Far\'ias for her help with the images.

\section{Overview and discussion}\label{heuristic}
{\bf Some advantages of singular complex analytic vector 
fields $X$ over the real analytic case on surfaces.}
On $M \backslash Sing(X)$,
$X$ determines a real vector field, $\Re{X}$,  
and a local $\RR^2$--action (both are real analytic).
Furthermore, the singular complex analytic vector fields
$X$ enjoy some very special properties respect to the 
real analytic vector fields and actions, 
on real analytic surfaces.

\smallskip
\emph{Existence of global rectifying maps (flow box and Newton maps).}
Recall the classical 
result, which goes back to Riemann, which states
that ``every compact Riemann surface $M$ can be described as a 
ramified covering 
on the sphere $\CW$, where the placings and orders of 
the ramification points and their
values determine $M$'', see \cite{Mumford} Lecture I, for this synthesis. 
Assertions (5)--(6) of the dictionary provide a generalization to singular complex analytic vector fields:
the following commutative diagram of pairs, 
Riemann surfaces and vector fields, holds true

\begin{center}
\begin{picture}(210,55)
\put(-110,20){\vbox{\begin{equation}\label{diagrama-basico}\end{equation}}}

\put(17,8){$ \Big(\CW_{t},\del{}{t}\Big) $}
\put(62,12){\vector(1,0){68}}
\put(73,0){$\exp(-t)$}

\put(133,8){$ \Big(\CW_{w},-w\del{}{w}\Big) $.}
\put(122,34){$\Phi_{X}$}
\put(113,42){\vector(1,-1){20}}

\put(76,45){$(M,X)$}
\put(47,34){$\Psi_{X}$}
\put(74,42){\vector(-1,-1){20}}

\end{picture}
\end{center}

\noindent
Note that $\del{}{t}$ or $-w\del{}{w}$ are the 
simplest complex analytic vector fields on
the Riemann sphere $\CW$, 
see Example \ref{holoVFonSphere} in \S\ref{GeoDynPullback}.

\noindent
In the language of differential equations:

\smallskip
{\it
\noindent
$\bullet\ X$ admits a global flow--box
\begin{equation}\label{cajaflujoglobal}
X(z)=\Psi_{X}^{*}\Big(\del{}{t} \Big)(z)=\frac{1}{\Psi_X^{\prime}(z)}\del{}{z}.
\end{equation}

\noindent
$\bullet\ X$ is the  Newton vector field 
of $\Phi_{X}$, 
\begin{equation}\label{flujoNewton}
X(z)=\Phi_{X}^{*} \Big( -w\del{}{w} \Big)(z)=-\frac{\Phi_X(z)}{\Phi^{\prime}_X(z)}\del{}{z}.
\end{equation} 
}
\smallskip

In general, equation \eqref{cajaflujoglobal} does not hold  for real 
analytic vector fields on any real analytic  surface, 
see \cite{Palis} ch. 3, \S 1.
As a corollary, no limit cycles appear for complex analytic vector fields, see 
\cite{Lukashevich}, \cite{Benzinger}, and \cite{Sabatini}, for other proofs.
As for equation \eqref{flujoNewton}, recall the ideas of S.\,Smale {\it et al.} \cite{HirschSmale}, \cite{Smale1}: 
the Newton vector field of $\Phi_X$ has attractors (sinks) at the simple roots of $X$, 
thus enabling the search for the zeros of $\Phi_X$ using its Newton vector field and their sinks.

\smallskip
\emph{Finite dimensional families of singular complex analytic 
vector fields.}
Finite dimensional families are natural 
in the complex analytic category, 
in contrast with infinite 
dimensional families in the smooth category.
As examples, recall the polynomial families
studied in \cite{Branner-Dias} and \cite{Frias-Mucino}.

\noindent 
In \cite{AP-MR}, \cite{AP-MR-2} and \cite{AP-MR-3}, the authors 
study the geometry and 
dynamics of singular complex analytic vector fields in the 
vicinity of essential singularities. 
In particular, they 
focus the dictionary on \emph{meromorphic structurally finite 1--order $d$ vector fields with $r$ poles 
and $s$ zeros on $\CC$}. 
These are finite dimensional
families consisting of vector fields on the Riemann sphere with a singular set composed of a 
finite number $s\geq 0$ of zeros on $\CC$, $r\geq0$ of poles on $\CC$ and an isolated 
essential singularity (of finite 1--order $d\geq1$) at $\infty\in\CW$, namely
\begin{multline}\label{familiaEsrd}
\E(s,r,d)=\Big\{ X(z)=\frac{Q(z)}{P(z)}\ \e^{E(z)}\del{}{z} \ \Big\vert \\
Q, P, E\in\CC[z],  \ \deg{Q}=s,\ \deg{P}=r, \ \deg{E}=d \Big\},
\end{multline}
where $d\in\NN$, $s, r\in\NN\cup\{0\}$ and $r+d+s\geq 1$. 

\noindent
In particular, 
when $X\in\E(0,r,d)$, they extend the dictionary (1)--(8) to:

\smallskip

\noindent 
\hskip0.2cm 
9. Classes of $(r,d)$--configuration trees $[\Lambda_{X}]$, see theorem 6.1 of \cite{AP-MR-2}.

\noindent 
\hskip0.2cm 
10. Functions $\Psi$ that can be expressed as quotients of linearly independent solutions of a 
certain Shr\"odinger type differential equation (work in progress). 

\smallskip
\emph{Automorphisms groups of singular complex analytic 
vector fields.}
Furthermore, in \cite{AP-MR-3}, they show that subspace consisting of those $X\in\E(s,r,d)$ 
with trivial isotropy group has a holomorphic trivial principal $Aut(\CC)$--bundle structure.

\smallskip
\emph{Incompleteneess of the flow and its geometric structure.}
In \S\ref{subsec:Flows}, we show that 
the maximal analytic continuation 
of the a local flow of $X$ is univalued on the 
Riemann surface $\R_{X} \subset M \times \CC_t$, 
where $\R_X$ is the 
graph of $\Psi_{X}$.
Furthermore, the maximal domain $\Omega_X$ 
of the complex flow is foliated by copies of $\R_X$ that differ by a $t$--translation.

\medskip
\noindent
{\bf Why are new visualizations methods required near essential singularities?} 
In order to gain insight into the behaviour of singular complex analytic vector fields, the correct 
visualization of vector fields in the neighborhood of points of the singular set $Sing(X)$ is required. 
The visualization of singular complex analytic vector fields at zeros and poles is well understood, 
see Proposition \ref{prop:normalforms} and Figure \ref{forma-normal} in \S\ref{normalforms}. 
However essential singularities present a challenge, see for instance Figures \ref{campoExp}
, \ref{CampoExpZ3}, \ref{campoTanz} and \ref{coshPlus2}.

Complex analytic functions $f(z)$ \emph{behave wildly}
in the neighborhood of an essential singularity. 
This is the meaning of Picard's theorem, 
in particular a function takes on all complex 
values, except possibly one, in 
any neighborhood of an essential singularity.
Obviously, essential singularities of 
complex analytic vector fields $f(z)\del{}{z}$ present the analogous behavior. 
As a consequence,
not much has been done with respect to the visualization 
of the class of vector fields with essential singularities: to our knowledge the only reported works related 
to the visualization of vector fields in neighborhoods of essential singularities is that of \cite{HR} and 
\cite{NewtonLofaro}: in both cases they use \emph{integration--based} visualization schemes.

The main issue with the use of \emph{integration--based} visualization methods and/or algorithms when visualizing vector fields near an essential singularity, is that these algorithms are based on recursive procedures. 
Hence, when the function characterizing the vector fields are evaluated near an essential singularity, they assume arbitrarily small and large values. This in turn causes the numerical errors to quickly become unmanageable, even when considering self--adjusting algorithms.

\medskip
\noindent
{\bf A brief survey of the visualization method for vector fields.}
Considering Diagram \eqref{diagrama-basico}, 
one has the option of studying the (possibly multivalued) singular complex analytic maps $\Psi_{X}$, 
or $\Phi_{X}$; because of correspondence (1)--(7) both are equivalent to the study of $X$. 
For instance, the right hand side 
of Diagram \eqref{diagrama-basico},  $\Phi_{X}$, easily provides a technique which enables us to 
completely solve, by geometrical methods, the differential equation 
\eqref{ecdifflow}, 
{\it i.e.} in particular visualize the trajectories of the singular analytic vector field $X$.

This technique, using $\Phi_X$, 
was originally presented by H.\,E.\,Benzinger, S.\,A.\, Burns and J.\,I.\,Palmore  \cite{Benzinger}, 
\cite{BurnsPalmore}, \cite{PalmoreBurnsBenzinger} in order to visualize \emph{rational} 
vector fields on $\CC$. In this work we show that the technique can be 
extended to work on singular complex analytic vector fields, even those 
that have essential singularities or accumulation points of poles and zeros.

\noindent
We do this by 
\begin{enumerate}[label=\roman*)]
\item \emph{extending the visualization method}, originally presented by H.\,E.\,Benzinger, S.\,A.\,Burns 
and J.\,I.\,Palmore for \emph{rational} vector fields in $\CC$, \emph{to work on all Newton vector fields on an 
arbitrary Riemann surface}, and
\item since \emph{all singular complex analytic vector fields are in fact Newton vector fields}, 
this provides a framework in which we can actually obtain a solution of \eqref{ecdifflow} for 
\emph{all singular complex analytic vector fields} and hence can be visualized.
\end{enumerate}

The conceptual idea behind the proposed geometrical method for 
visualizing singular complex analytic vector fields, 
is the construction of a 
\emph{pair of real valued functions that are constant and linear along the trajectories of the vector field} 
(also known as first integrals or integrals of motion), see Theorem \ref{TeoremaVisual}.

\smallskip
\noindent
As it turns out, the method that we generalize has 
some other very interesting and noteworthy advantages over the usual 
vector field visualization techniques (see \cite{LarameeEtal},
 \cite{PostEtal}, \cite{SalzbrunnEtal} for a classification 
scheme of vector field visualization techniques). 
Amongst them, we state the following properties.

\begin{enumerate}[label=\Alph*)]
	\item The method allows for the global visualization of 
	vector fields on arbitrary Riemann surfaces.
	\item It allows for the efficient visualization of the streamlines, even for specific initial conditions.
	\item It can provide information relative to (the parametrization of) the flows.
	\item It does not propagate numerical errors.
	\item It allows the correct visualization of vector fields even in regions where the usual \emph{integration--based} algorithms fail.
	\item The computer resources needed for the visualization are much less than those needed by other \emph{integration--based} visualization techniques.
	\item The algorithm can be easily parallelized.
	\item Moreover it can be easily extended to work on a larger class of vector fields.
\end{enumerate}

It should be noted from the outset that the method in question exploits a well known characteristic of Newton 
vector fields, namely that their streamlines can be easily recognized by a geometrical argument 
(see Lemma \ref{lemasoluciones}). 
Yet, it is interesting to note that apparently this method is unknown (or at least not actively used), even for 
those who study Newton vector fields: for instance in \cite{HelminckEtAl}, \cite{TwiltEtal}. 
In particular,  though they show that the Newton flow associated to the Weierstrass $\wp$--functions can be 
characterized/classified (up to conjugacy) into three types of behaviour,
and that they actually show phase portraits of the Newton flow associated 
to Weierstrass $\wp$--function and 
to Jacobi's $sn$--function, \emph{they still use a traditional integration--based algorithm 
(4--th order Runge--Kutta) for the visualization of the vector field}.

\medskip
\noindent
{\bf On the visualization of complex functions.}
In \S\ref{visualizingcomplexfunctions}, 
as an application of the techniques and methods developed in the previous sections, we explore the problem 
of \emph{visualization of singular complex analytic functions}. 

\noindent
We start with a quick review of some classical and or traditional methods unrelated to vector fields; 
particularly \emph{images of regions}, \emph{tilings a la Klein}, the \emph{analytic landscape} 
and \emph{domain coloring}. 

\noindent
We then procede to explore two methods based on the visualization of the phase portrait of certain singular 
complex analytic vector fields. 

\noindent
The advantages and disadvantages of the different methods are presented and discussed.
In particular it should be noted that Theorem \ref{visualizationPsiPhi} provides 
\begin{enumerate}[label=\alph*)]
\item a natural tool for the visualization of both $\Psi_X$ and $\Phi_X$,
\item a natural counterpart to Theorem \ref{TeoremaVisual}.
\end{enumerate}

\section{Analytic and geometric aspects of singular 
complex analytic vector fields }
\label{antecedentes}

We define the basic objects of study, 
namely singular complex analytic vector fields, 
then present a quick overview of the basic correspondence. 
The material is 
presented in full detail in 
\cite{MR}, \cite{MR-VV}, \cite{AP-MR} and \cite{AP-MR-2}. 
We describe a summary here for clarity and completeness 
in the exposition.

\subsection{Notation and conventions.}\label{subsec:notacion}

\quad\\
$\mathcal{ M }$ is an oriented smooth ({\it i.e.} $ C^{\infty}$)
two--manifold.

\noindent $J$ is a complex structure on $\mathcal{M}$ (a smooth
isomorphism of $T \mathcal{ M } $ such that $J^{2} = -1$).

\noindent
$M= (\mathcal{ M },J)$ is a Riemann surface.

\noindent
$\CW= \CC\cup\{\infty\}$ is the Riemann sphere.

\noindent
$(\CC,0)$ denotes the usual domain for germs.

\noindent
$i=\sqrt{-1}$. 

\noindent
We will be interested in complex--valued vector fields 
on a Riemann surface that are \emph{analytic} in the following sense.
\noindent
Let $\{  \phi_{\tt j} : V_{\tt j} \subset M \to
\CC \ \vert \ {\tt j} \in {\tt J}\}$,
be a holomorphic atlas for $M$.

\begin{definition}\label{singulardefinition}
By a \emph{singular complex analytic vector field}
$$
X=\Big\{ f_{\tt j} (z) \del{}{z} \ \Big\vert \ z \in \phi_{\tt j} (V_{\tt j}) \Big\}
$$
on $M$, we understand
a (non--vanishing) holomorphic vector field $X$ on $M \backslash Sing(X)$, where 
$Sing(X)$ is the \emph{singular set} of $X$, which consists of:
\begin{itemize}[label=$\bullet$]
\item zeros, denoted by $\MZ$, 
\item poles, denoted by $\MP$, 
\item isolated essential singularities denoted by $\ES$, and
\item accumulation points in $M$ of zeros, poles and isolated essential singularities of $X$, denoted by $\CS$. 
\end{itemize}
So $Sing(X)=(\ES\cup \MP\cup \MZ\cup \CS)$ is the closure in $M$ of the set $(\ES\cup \MP\cup \MZ)$.
\end{definition}

\noindent
We will denote by \\
$M^{\diamond}=M\backslash\overline{\ES}$.\\
$M^{0}=M^{\diamond}\backslash\overline{(\MP\cup \MZ)}=M\backslash Sing(X)$.\\
$M^{'}=M^{\diamond}\backslash\overline{\MZ}= M \backslash\overline{(\ES\cup \MZ)}$.\\
$M^{*}=M^{\diamond}\backslash\overline{\MP}=M\backslash\overline{(\ES\cup \MP)}$.\\

\medskip 

We wish to note that our definition of singular complex analytic vector fields  
includes several of the classical families depending on what 
the singular set $Sing(X)$ is. For instance: 
\begin{itemize}[label=$\bullet$]
\item 
If $\ES=\MP=\emptyset$, then $X$ is a \emph{holomorphic vector field on $M$}. 
Note that in this case $Sing(X)=\MZ$ has no accumulation points in $M$ (unless of course $X$ is the identically zero vector field). 

\item
\emph{Entire vector fields} are precisely the holomorphic vector fields on $\CC$, 
or equivalently singular complex analytic vector fields on $\CW$ with 
$\MP \cup \ES $ equal to $\{\infty \}$ or $\emptyset $.

\item 
If $\ES=\emptyset$ and $Sing(X)$ has no accumulation points in $M$, then $X$ is a \emph{meromorphic vector field on $M$}. 
Thus \emph{rational vector fields} are precisely the \emph{meromorphic vector fields on $\CW$}.

\item If $\ES$ is non--empty and there are no accumulation points of $\MZ$ in $M$, then $e\in \ES$ will consist of an essential singularity of $f_{\tt j}$ that has $0$ as a lacunary value, that is there is a neighborhood $V$ of $e$ where $f_{\tt j}(z)\neq0$, for all $z\in \phi_{\tt j}(V_{\tt j}\cap V)$.
\end{itemize}

\noindent
For other relevant cases one may consider $\E(s,r,d)$, 
meromorphic structurally finite 1--order $d$ vector fields with 
$r$ poles and $s$ zeros on $\CC$
recall \eqref{familiaEsrd}, see \cite{AP-MR}, \cite{AP-MR-2} and \cite{AP-MR-3}.
For geometric structures associated to vector fields and its applications 
see \cite{Guillot3}.

Since a vector field provides a geometric structure for $M$, see 
\S\ref{subsec:equivalencias}, in several places we use
the notation $(M, X)$ as a pair, Riemann surface and vector field. 
Moreover, complex structures on $M^{0}$ having conformal punctures on $\overline{(\ES\cup \MP\cup \MZ)}$ extend in a unique way to complex structures on all of $M$; we do not distinguish between the punctured Riemann surface $(M^{0},J)$ and the extended $(M,J)$.

Moreover, since we will always be dealing with Riemann surfaces we will drop the \emph{``complex''} adjetive (unless we wish to emphasize it), and whenever a \emph{singular complex analytic differential form, singular quadratic differential or singular function} is mentioned, the meaning of \emph{singular} should be that of Definition \ref{singulardefinition}.

\subsubsection{Equivalence between singular complex analytic vector fields and real smooth vector fields, trajectories}\label{realcomplexflows}

\quad\\
On $M$, more precisely on $M^*$, there is a one to one correspondence 
between real smooth vector fields satisfying the Cauchy--Riemann 
equations and $(1,0)$--sections of the holomorphic tangent bundle locally given by
\begin{align*}
F \longrightarrow & X=\frac{1}{2}(F  - i J F)\\
F=X+\bar{X} \longleftarrow &  X .
\end{align*}
In explicit local coordinates $(V_{\tt j} , \phi_{\tt j} )$ of $M$ this is
\begin{equation*}
X=f_{\tt j} (z ) \del{}{z} =\big( u_{\tt j} (x,y)+i\,v_{\tt j} (x,y) \big) \del{}{z}
\end{equation*}
so the \emph{real part of $X$} is
\begin{equation*}
\Re{X}:=F=u_{\tt j} (x,y)\frac{\partial}{\partial x} + v_{\tt j} (x,y)\frac{\partial}{\partial y}.
\end{equation*}
The trajectories of $X$ as in \eqref{ecdifflow} and the 
trajectories of $\Re{X}$ coincide.
In passing, we note that the \emph{imaginary part of $X$} is given by
\begin{equation*}
\Im{X}:=-v_{\tt j} (x,y)\frac{\partial}{\partial x} + u_{\tt j} (x,y) \frac{\partial}{\partial y},
\end{equation*}
and is nothing else than $J F$. \\
In particular since $f_{\tt j}$ represents a holomorphic function on 
$M^{*}$, then $u_{\tt j}$ and $v_{\tt j}$ satisfy the 
Cauchy--Riemann equations.

\subsection{Equivalences between singular vector fields, singular differential forms, 
singular orientable quadratic differentials and singular flat structures}\label{subsec:equivalencias}
\subsubsection{Equivalence with differential forms}\label{subsec:difforms}

To obtain the correspondence with differential forms, consider the singular
analytic vector field $X=\{ f_{\tt j} (z) \del{}{z} \}$ restricted to $M^0$. Since $\CC$ is an algebraic field, it
follows by duality, that the \emph{singular complex analytic} 1--form
$$
\omega_{X} =\Bigl\{  
\frac{dz}{f_{\tt j} (z)} 
\ \Big\vert \ 
z \in \phi_{\tt j} (V_{\tt j} ) \Bigr\}
$$
is such that $\omega_{X}(X)\equiv 1$.
In fact, $\omega_{X}$ is canonically well defined on all $M$; having zeros, poles and essential singularities at 
the points where $X$ has poles, zeros and essential singularities, respectively.

\noindent
The \emph{complex time necessary to travel from 
$z_{0}$ to $z$
in $M^0$ under the complex flow of $X$} 
is given by:
\begin{equation}\label{complexTime}
\Psi_{X,{\tt j}} (z)=\int_{z_{0}}^{z} \omega_{X} : 
V_{\tt j} \subset M' \longrightarrow \CC.
\end{equation}
A priori this depends on the homotopy class of the path from $z_{0}$ to $z$ in $M'$.
One also notices that $\Psi_{\tt j}(z)=\Psi_{\tt k}(z)+a_{\tt j k}$ on $V_{\tt j}\cap V_{\tt k}$, for some 
$a_{\tt j k}\in\CC$.
Hence, by direct analytic continuation we have the (possibly multivalued) \emph{global} singular analytic 
additively automorphic function
\begin{equation}\label{PsiGlobal}
\Psi_X (z)=\int_{z_{0}}^{z} \omega_{X} : M^{0} \longrightarrow \CC.
\end{equation}
See definitions 2.4 and 2.5 of \cite{AP-MR}.

\noindent
Locally, if $\alpha(\tau),\ \beta(s)\ : (-\varepsilon,\varepsilon)\subset\RR \rightarrow \phi_{\tt j} (V_{\tt j} )$ are 
trajectories of $F$ and $J F$ respectively, with $\alpha(0)=\beta(0)=z_{0}$, then
\begin{equation*}
\int_{z_{0}}^{\alpha(\tau)} \omega_{X} = \tau \quad\text{and}\quad \int_{z_{0}}^{\beta(s)} \omega_{X} = i s.
\end{equation*}
In words $F$ and $J F$ describe the real and imaginary time necessary to travel from $z_{0}$ to $z$.

\noindent
Moreover if $z_{1}$ and $z_{2}$ belong to the same real trajectory of $\Re{X}$ then
\begin{equation}\label{tiemporeal}
g_{X}\text{--}length(\overline{z_{1}z_{2}})
=\int\limits_{\overline{z_{1}z_{2}}} \omega_{X} 
=\begin{cases} 
\text{real time to travel from } z_{1} \text{ to } z_{2} \\
\text{under the local real flow of }X,
\end{cases}
\end{equation}
where $\overline{z_{1}z_{2}}$ means the 
geodesic segment in 
$(M^{0},g_{X})$, that will be
defined in \ref{subsec:flatStruct}, where it is understood
that the $g_X\text{--}length$ can assume negative values.

\subsubsection{Equivalence with orientable quadratic differentials}\label{subsec:quaddiff}

A \emph{singular complex analytic quadratic differential} 
$\mathcal{Q}$ on $M$ is by definition o\-rien\-ta\-ble if it is 
globally given as $\omega\otimes\omega$ for some singular complex 
analytic differential 1--form $\omega$ on $M$. 
F.\,Klein \cite{Klein} was the first to implicitly use these objects 
to study complex integrals, J.\,A.\,Jenkins \cite{Jenkins}, and 
K.\,Strebel \cite{Strebel} provide presentations of the subject, 
also recently J.\,C.\,Langer \cite{Langer} provides computer visualizations of quadratic differentials.

\noindent
Given $\mathcal{Q}=\omega_{X}\otimes\omega_{X}$,  we get a canonical holomorphic atlas
$\{ (V_{\tt j},\Psi_{\tt j}) \}$ for $M^{0}$ as above.
Noticing that the changes of coordinates $\Psi_{\tt j}\circ\Psi_{\tt k}^{-1}$ are maps of the form 
$\{z_{\tt k}\mapsto z_{\tt j}=z_{\tt k}+a_{\tt j k}\ | \  a_{\tt j k}\in\CC\}$, it follows that the real horizontal 
foliation on $\CC$ defines a \emph{horizontal foliation} $\mathcal{F_{Q}}$ on $M^{0}$. 
Furthermore, $\mathcal{F_{Q}}$ is defined by a real non--vanishing vector field on $M^{0}$ 
if and only if $\mathcal{Q}$ is orientable. Clearly the horizontal foliation corresponds to the 
trajectories of $F$ and there is a corresponding \emph{vertical foliation} corresponding to $J F$.

\subsubsection{Construction of a flat structure from $X$.}\label{subsec:flatStruct}

Now define the real analytic Riemannian metric
\begin{equation*}
g_{X} = \Bigg\{  \frac{1}{(u_{\tt j} (x,y))^2+(v_{\tt j} (x,y))^2}
\begin{pmatrix}
1 & 0\\
0 & 1
\end{pmatrix} \  \Big\vert  \  (x + iy ) \in \phi_{\tt j} (V_{\tt j}) \Bigg\}
\end{equation*}
on $M^0$, respect to suitable $(V_{\tt j}, \phi_{\tt j})$. $F$ and $J F$ 
define an orthonormal frame for $g_{X}$ on all $M^0$. 
By the Cauchy--Riemann equations $F$ and $J F$ commute, and the 
curvature of $g_{X}$ is zero.  
Equivalently, the functions $\Psi_{X,{\tt j}}:(V_{\tt j} , g_{X})\rightarrow (\CC,\delta)$ 
are isometries, where $\delta$ is the usual flat metric on $\CC$, 
and the trajectories of $F$ and $J F$ are unitary geodesics 
in the smooth flat Riemann surface $(M^0 , g_{X})$.

\begin{remark}
In the language of quadratic differentials, $\Psi_{X,{\tt j}}$, as in \eqref{complexTime}, is 
called a \emph{distinguished parameter near a regular point} 
for the orientable quadratic differential
$\mathcal{Q} = \{ dz^{2}  / (f_{\tt j} (z))^{2} \}$ see \cite{Strebel} p.\,20.  
Thus in the language of differential equations, we can say 
that $\Psi_{X,{\tt j}}$ is a local holomorphic \emph{flow box} for the vector field $X$, that is
\begin{equation}\label{flowBox}
\Psi_{X,{\tt j}} (z)_{*}\Big(f_{\tt j} (z)\del{}{z}\Big)=\del{}{t},
\end{equation}
where again $t\in\CC$ is complex time.

\noindent
Of course, the global singular analytic additively automorphic function $\Psi_X$, see \eqref{PsiGlobal}, 
also satisfies \eqref{flowBox}: thus $\Psi_X$ is a \emph{global} flow box. 
This is further explored in \S\ref{pullbackHoloVF}, particularly in \S\ref{pullbacks}.
\end{remark}

\subsection{The Riemann surface $\R_X$}

The graph of $\Psi_{X}$
\begin{equation}
\R_{X}= \{(z,t) \ \vert \  t=\Psi_{X}(z) \} \subset M\times\CW_{t}
\end{equation}
\noindent 
is a Riemann surface. 
The flat metric
$\big(\R_{X},\pi_{X,2}^{*}(\del{}{t})\big)$ is induced by 
$\big(\CW,\del{}{t}\big)$ 
via the projection of $\pi_{X,2}$, and coincides with $g_{X}=\Psi_{X}^{*}(\delta)$ since $\pi_{X,1}$ is an isometry, 
as in the following diagram:
\begin{center}
\begin{picture}(180,80)(0,10)

\put(-125,40){\vbox{\begin{equation}\label{diagramaRX}\end{equation}}}

\put(12,75){$\big(M,X\big) $}

\put(115,75){$\big(\R_X,\pi^*_{X,2}(\del{}{t})\big)$}

\put(108,78){\vector(-1,0){60}}
\put(65,85){$\pi_{X,1}$}

\put(133,65){\vector(0,-1){30}}
\put(138,47){$ \pi_{X,2} $}

\put(38,65){\vector(2,-1){73}}
\put(55,39){$ \Psi_X $}

\put(115,20){$\big(\CW_t,\del{}{t}\big) $}

\end{picture}
\end{center}

\begin{remark}
1. It should be noted that $\pi_{X,1}:\big(\R_{X},\pi^*_{X,2}(\del{}{t})\big)\longrightarrow (M,X)$ is a biholomorphism if and only if $\Psi_{X}$ is single valued.

\noindent
2. In Diagram \eqref{diagramaRX} we abuse notation slightly by saying that the domain of $\Psi_{X}$ is $M$. 
This is a delicate issue, see 
Remark \ref{comentariosCorrrespondencia}.3 following Proposition \ref{basic-correspondence}.

\noindent 
3. In what follows, unless explicitly stated, we shall use the abbreviated form $\R_{X}$ instead of the more 
cumbersome $\big(\R_{X},\pi^*_{X,2}(\del{}{t})\big)$.
\end{remark}

\begin{example}[Holomorphic vector fields on the Riemann sphere]\label{holoVFonSphere}
The holomorphic vector fields on $\CW_t$ form a three dimensional
complex vector space
$$
\Big\{ Y(t)=(at^2 + bt + c) \del{}{t} 
\ \vert \ (a,b,c) \in  \CC^3 \Big\} ,
$$
which is isomorphic to the Lie algebra of the group of biholomorphisms 
$PSL(2, \CC)$ of the Riemann sphere.
These are the only complete vector fields on the 
Riemann sphere, see \S\ref{subsec:Flows}. 
In $\CW_t$, a non zero holomorphic vector field $Y(t)$ can have: two simple zeros or
one double zero.
Up to automorphisms $Aut(\CW) \cong PSL(2, \CC) $,
we get two qualitatively different families of (non--identically zero)
holomorphic vector fields on $\CW_t$:

\begin{enumerate}[label=\arabic*.,leftmargin=*]

\item
The constant vector fields

\centerline{
$\lambda\del{}{t}$, \ \ \ $\lambda \in \CC^*$,}

\noindent
correspond to the family having a double zero. The vector field
$\Re{\lambda\del{}{t}}$ has a dipole at infinity, see \S\ref{normalforms}.

\noindent
$(\CC, g_X)$ is isometric to the euclidean plane foliated by (geodesic)
trajectories of $\Re{\lambda}\del{}{x}+\Im{\lambda}\del{}{y}$. 
Notice that for any $\lambda\neq0$;
$(\CW, \lambda\del{}{t})$ is global holomorphically
equivalent to $(\CW, \del{}{t})$ or $(\CW, t^2 \del{}{t})$. 
$(\CC, \del{}{t})$ is isometric to the euclidean plane foliated by (geodesic)
trajectories of $\del{}{x}$.
See Figure \ref{holoVfSphere}.

\item
The linear vector fields

\centerline{
$\frac{t}{\lambda } \del{}{t}$, \ \ \
$\lambda \in \CC^*$, }

\noindent 
which correspond to the family having two simple zeros. 
Contrary to the previous family, 
$\frac{t}{\lambda }\del{}{t}$ is global holomorphically
equivalent 
to $\frac{t}{\nu} \del{}{t}$ if and only if $\lambda = \pm \nu$.
The vector fields
$\Re{ \frac{t}{\lambda } \del{}{t}} $ 
have; two centers if $\Re{\lambda } = 0$;  one source, one sink otherwise, see \S\ref{normalforms}.
\\
In particular, the pullback of $ Y= - t \del{}{t}$ will produce a Newton vector field on $M$ 
(see \S\ref{camposnewtonianos} for the definition), and the Riemannian manifold
$(\CC^*, g_Y)$ is isometric to the euclidean cylinder $\CC / 2\pi i \ZZ$
foliated by (geodesic) trajectories of $\del{}{x}$.
See Figure \ref{holoVfSphere}.
\end{enumerate}
\begin{figure*}[htbp]
\centering
\includegraphics[width=0.9\textwidth]{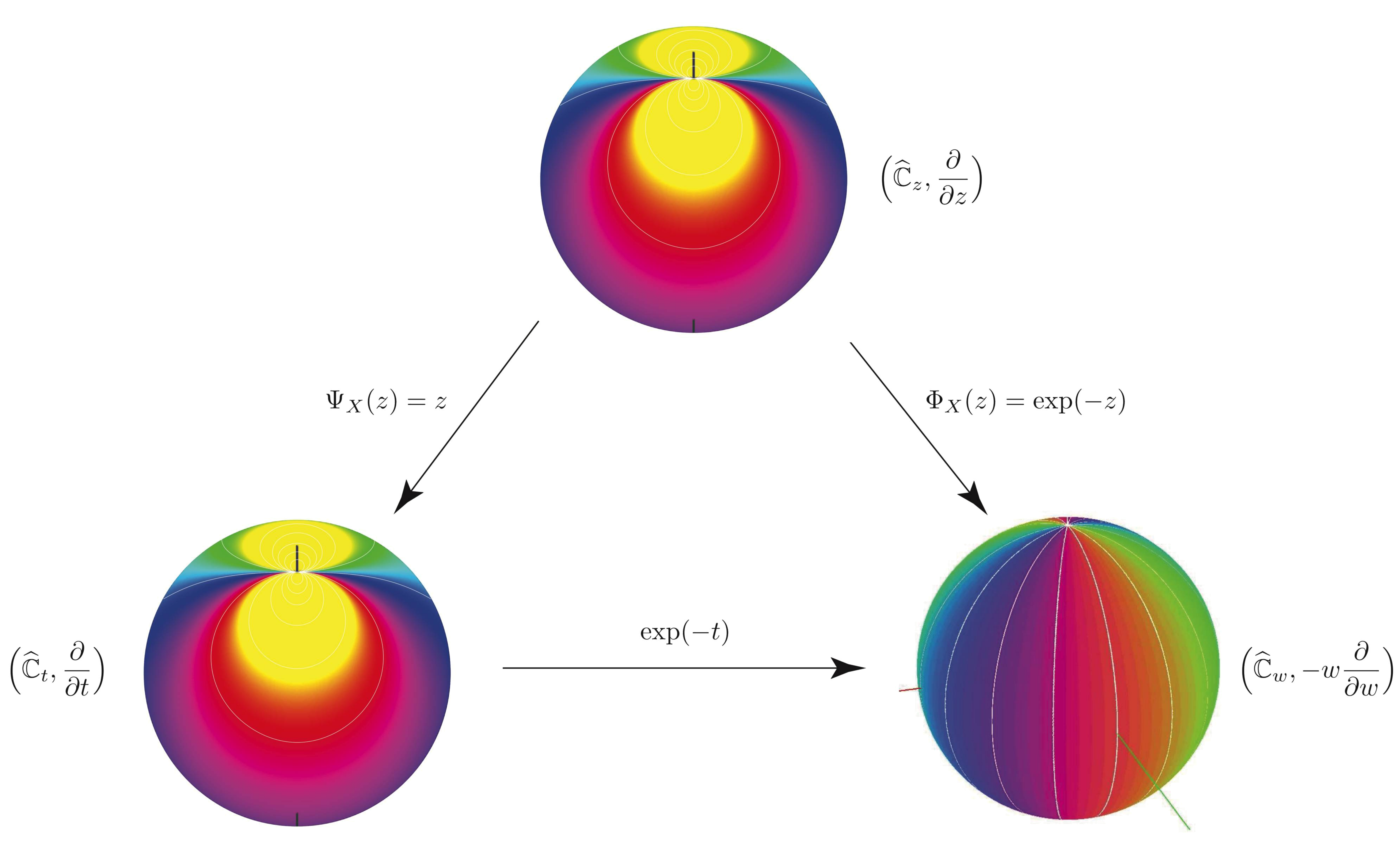}
\caption{Diagram \eqref{diagrama-basico} for $X(z)=\del{}{z}$. 
The holomorphic vector fields $\del{}{z}$ and $-w\del{}{w}$ on 
the Riemann sphere appear in a very natural 
context.}
\label{holoVfSphere}
\end{figure*}
\end{example}

\noindent
From these examples the case of pullbacks of $\del{}{t}$ 
and $t\del{}{t}$ (or $\pm w \del{}{w}$) 
should be relevant, as we will conclude in \S\ref{holofamilycorresp}.

\begin{example}[Vector fields having maximal domain of their flows 
different from $\CC_t$]\label{ejemplo-con-psi-racional}
Let 

\centerline{$\big(\CW,X(z)=\frac{1}{R^\prime(z)}\del{}{z}\big)$}

\noindent 
be a rational vector field, for $R(z)$ a rational function of degree at least two.
$X$ has at least one pole and note that 
\emph{the holomorphic differential equations theory
can not be applied}. 
However, in accordance with Diagram \eqref{diagramaRX},

\centerline{$\Psi_X (z) =R(z) : \CW_z \longrightarrow \CC_t $}

\noindent 
is single valued. 
Thus, $\pi_{X,1}$ provides a single valued global 
flow of $X$, with the property that 
\begin{equation}\label{tiempo-algo}
R(z_2) - R(z_1)
=\int\limits_{z_{1}}^{z_2} \omega_{X} 
=\begin{cases} 
\text{complex time to travel from } z_{1} \text{ to } z_{2} \\
\text{under the flow of }X,
\end{cases}
\end{equation}
for $z_1, z_2 \in \CW$. In particular for $z_1$ a cero and 
$z_2 \in \CW \backslash Sing(X)$  the complex time $\infty$
makes sense. Moreover, 

\centerline{
$
\R_X = \Omega_X = \{ (z, R(z)) \ \vert \ z \in \CW \}. 
$}

\noindent 
For further discussion see \S\ref{subsec:Flows}.
\end{example}

\subsection{The singular complex analytic dictionary}

\begin{definition}\label{automorfa-aditiva}(\cite{BerensteinGray}, p.\,579)
Let $\Psi: M \to \CC$ be a singular complex analytic possibly multivalued function 
with a non--dense countable singular set $\overline{\mathcal{A}}$ such that 
the restriction of $\Psi$ to $M\backslash\overline{\mathcal{A}}$ is holomorphic. 
$\Psi$ is called \emph{additively automorphic} if given two branches 
one has $\Psi_{\alpha }(z)=\Psi_{\beta}(z) + a_{\alpha\beta}$, for some constants $a_{\alpha \beta}\in\CC$.
\end{definition}
\noindent
Note that $\Psi$ is single valued if and only if $a_{\alpha \beta}=0$ 
for all $\alpha, \beta $. However, the 1--form
$d \Psi$ is always single valued on $M$, when 
Definition \ref{automorfa-aditiva} holds true. 
For instance $\log(z)$ and $\log( P(z))$, for $P(z)$ a polynomial, 
are additively automorphic, however $P(z)\log(z)$ is not. 

\noindent
In summary one has the following result.

\begin{proposition}[Singular complex analytic dictionary] 
\label{basic-correspondence}
On any Riemann surface $M$ there is a canonical one--to--one correspondence between:
\begin{enumerate}[label=\arabic*)]
\item Singular complex analytic vector fields $X(z)=f(z)\del{}{z}$.
\item Singular complex analytic differential forms $\omega_{X}$, related to $X$ via $\omega_{X}=\frac{dz}{f(z)}$.
\item Singular complex analytic 
orientable quadratic differential forms given by $\omega_{X} \otimes\omega_{X}$.
\item Singular (real) analytic flat structures $g_{X}$, satisfying $g_{X}(X)\equiv 1$,
with suitable singularities on a non--dense countable set $\mathcal{A} \subset M$,
trivial holonomy in $M\backslash \mathcal{A}$
and a (real) geodesible unitary 
vector field $W_X$ whose singularities are exactly $\mathcal{A}$.
\item Singular complex analytic 
(possibly multivalued) maps, \emph{distinguished parameters},
$$
\Psi_{X} (z)= \int^z_{z_{0}} \omega_X : M^\diamond \to \CW
$$
where $z_0 \in M'$ and $z \in M^{\diamond}$.
\item Singular complex analytic 
(possibly multivalued) \emph{Newton maps}
$$
\Phi_{X}(z)=\exp\left[-\int^z_{z_{0}} \omega_X\right] : M^\diamond \to \CW
$$
where $z_0 \in M'$ and $z \in M^{\diamond}$.
\item The pairs $\big(\R_{X},\pi^{*}_{X,2}(\del{}{t})\big)$ consisting of branched Riemann surfaces $\R_{X}$, 
associated to the maps $\Psi_{X}$, and the vector fields $\pi^{*}_{X,2}(\del{}{t})$ under the projection 
$\pi_{X,2}: \R_{X} \longrightarrow \CW_{t}$.
\item 
The pairs $(\Omega_X, \mathcal{F})$
consisting of maximal domains 
$\Omega_{X} \subset M \times \CW$ 
of the complex flows of $X$ and holomorphic foliations 
$\mathcal{F}$ whose leaves
are copies of the Riemann surfaces $\R_X$.
\end{enumerate}
\end{proposition}

\begin{proof}[Sketch of proof]
The equivalence between (1), (2) and (3) 
is well known and extensively used; 
it is only necessary
to verify that the local complex analytic tensors
transform in the required way, 
see \S\ref{subsec:difforms} 
and \S\ref{subsec:quaddiff}.

That (4) follows from 
(3) uses
the flat metric associated to $\omega_X \otimes \omega_X$, see \S\ref{subsec:flatStruct}.

For the converse assertion, we 
start with a flat structure $g$ 
on $M \backslash \mathcal{A}$.
Since the riemannian holonomy of $g$,
$\pi_1 (M \backslash \mathcal{A})  \longrightarrow O(2)$, 
is the identity, we 
recognize  $W_X$ and its counterclockwise $\pi/2$ rotated vector field, 
say
$\e^{i\pi/ 2} W_X$,  as the real and imaginary parts 
of a holomorphic vector field $X$ on $(M \backslash \mathcal{A}, g)$. 
The extension of $X$ to $\mathcal{A}$ depends 
on the nature of the singularities of $X$. 
The 
suitable singularities hypothesis in (5), means that the 
extension exists. 
Obviously, 
poles and zeros of $X$ at $\mathcal{A}$ are suitable singularities and 
can be recognized
by their normal forms in punctured neighborhoods, 
see Proposition \ref{prop:normalforms}. 
For further details see \cite{AP-MR} lemma 2.6 and
theorem D.

The equivalence between (5) and (6) and their relationship is further explored in \S\ref{holofamilycorresp}.
The correspondence between (5) and (7) 
follows from Diagram \eqref{diagramaRX}. 
Equivalence between (7) and (1) is postponed to  \S\ref{subsec:Flows}, see Corollary \ref{DeRaX}.
The same is true for the equivalence between (7) and (8): this is the content of Theorem \ref{FlujoMaximal} 
in \S\ref{subsec:Flows}.
\end{proof}

\begin{remark}\label{comentariosCorrrespondencia}
Some comments are in order:

\noindent
1. $\Psi_{X}$ and $\Phi_{X}$ as in (5) and (6) of Proposition \ref{basic-correspondence} are
well defined and holomorphic maps for $z$ at
the poles of $\omega_X$.

\noindent 
2. The local map $\Psi_{\tt j} = \int  d \zeta / f_{\tt j}(\zeta)$ 
in \eqref{complexTime} is called 
\emph{distinguished parameter} by K.\,Strebel \cite{Strebel}
p.\,20
and also by L.\,V.\,Ahlfors \cite{Ahlfors-Conformal}, we will
continue using this name for the \emph{global} map $\Psi_{X}$ described in (5) of 
Lemma \ref{basic-correspondence}. 

\noindent
3. The choice of initial and end points $z_{0}, z$ for the integral defining $\Psi_X$ and $\Phi_X$ 
can be relaxed to include the essential singularities by integrating along asymptotic paths 
associated to asymptotic values of 
$\Psi_{X}$ at the essential singularities $E\subset M$, see remark 1.1 of \cite{AP-MR-2}.

\end{remark}

\section{Pullback of singular complex analytic vector fields}\label{pullbackHoloVF}
We start by recalling the classical local notion of \emph{holomorphically equivalent} or 
\emph{conformally conjugated} vector fields,
see \cite{BrickmanThomas}, \cite{IlyashenkoYakovenko} p.\,9 for the usual 
concepts.
Moreover, the following remains valid
for  regular points and singularities in the sense of Definition \ref{singulardefinition}  
(namely zeros, poles, isolated essential singularities of $X$  
and accumulation points of the above at the origin).

\begin{definition}
\label{holomorphically equivalent} Let $X(z)=f(z)\del{}{z}$ and
$Y(z)=g(z)\del{}{z}$ be two germs of singular complex analytic vector fields
on $(\CC_z,0)$ and let 

\centerline{ 
$\varphi_{f}(z,t)$, 
$\varphi_{g}(z,t):\big(\CC^2_{z \, t},(z_0, 0)  \big)
\longrightarrow (\CC_{z},0)$,}

\noindent 
for a point $z_0$ where $f$ and $g$ are holomorphic, be their local holomorphic flows.
\begin{enumerate}[label=\arabic*.,leftmargin=*]
\item $X$ and $Y$ are \emph{topologically equivalent} if there exists an orientation preserving 
homeomorphism $\Upsilon:(\CC,0)\longrightarrow (\CC,0)$ which takes  trajectories of $\Re{X}$  
to trajectories of $\Re{Y}$ preserving their orientation but not necessarily the parametrization.
\item  $X$ and $Y$ are \emph{holomorphically equivalent} if there
exists a biholomorphism  $\Upsilon:(\CC ,0) \to (\CC ,0)$ such that
\begin{equation}\label{confconjDef}
\Upsilon(\varphi_{f}(z,t))=\varphi_{g}(\Upsilon(z),t)
\end{equation}
whenever both sides are well defined, for the maximal analytic continuations.
\end{enumerate}
\end{definition} 
\noindent
Note that, under the assumption that $\Upsilon$ is a biholomorphism, \eqref{confconjDef} 
is equivalent to $X=\Upsilon^* Y$.

\begin{lemma}\label{confconjEquiv}
Two germs of singular complex analytic vector fields $X(z)=f(z)\del{}{z}$ and
$Y(z)=g(z)\del{}{z}$ on $(\CC, 0)$ are 
holomorphically equivalent
if and only if there exists a biholomorphism 
$\Upsilon :(\CC ,0) \to (\CC ,0)$
such that
\begin{equation}\label{confconjDef2}
f(z)=\frac{g(\Upsilon (z))}{ \Upsilon '(z)}, \ \ \ \text{ for all }z\in
(\CC , 0 ), \ \ \ z\neq0.
\end{equation}
\end{lemma}
\begin{proof}
The proof follows by taking the derivative with respect to $t$ in \eqref{confconjDef}.
\end{proof}

From a global point of view, 
two singular complex analytic vector fields $(M,X)$, $(N,Y)$ on arbitrary
Riemann surfaces are holomorphically equivalent if there exists a
biholomorphic map $\Upsilon :M \to N$ such that
$\Upsilon (\varphi_{X}(z,t)) = \varphi_{Y} (\Upsilon (z),t)$
whenever both sides are well defined.

\subsection{Pullbacks of singular complex analytic vector fields by singular complex analytic maps}
\label{sec:coveringypullback}

The pullback $\Upsilon^*$ is a natural operation when considering vector fields.

\begin{lemma}
\label{pullbackformula} 
1. Given a singular complex analytic vector field
$Y(t)=\Big\{g_{\tt k} (t)\del{}{t}\Big\}$ on $N$ and a non--constant, 
singular complex analytic map
$$
\Upsilon: M \rightarrow N,
$$
the pullback vector field $X=\Upsilon^*Y=\Big\{f_{\tt j} (z)\del{}{z}\Big\}$ is a singular complex analytic vector
field well defined on $M$. In particular
\begin{equation}\label{eqpullback}
f_{\tt j}(z_{\tt j})=\frac{g_{\tt k} (\Upsilon_{\tt j k} (z_{\tt j}))}{\Upsilon'_{\tt j k} (z_{\tt j})},
\end{equation}
where $\Upsilon_{\tt j k} = \phi_{\tt jk} \circ \Upsilon \circ \phi_{\tt j}^{-1}$,
$\Upsilon'_{\tt j k} = \frac{d\Upsilon_{\tt j k} }{dz}$, and
$\{ \phi_{\tt j}:V_{\tt j} \subset M \rightarrow \CC\}$,
$\{ \phi_{\tt k}: U_{\tt k} \subset N \rightarrow \CC \}$
are the charts of $M$ and $N$ respectively.

\noindent
2. Conversely, if $X$, $Y$ are given singular complex analytic vector fields on $M$, $N$ respectively and 
$\Upsilon$ is a (possibly multivalued) singular complex analytic function that satisfies \eqref{eqpullback}, then 
$$X=\Upsilon^{*}Y.$$
\end{lemma}

\begin{proof} Follows from Lemma \ref{confconjEquiv} and an easy computation in local coordinates.
\end{proof}
\noindent
The second statement concerning multivalued functions will be used in our work in \S\ref{Flowbox} 
and \S\ref{Newton}.

\medskip
We make a further convention: since we will be working on the Riemann surface $M$, no mention will be made 
of the local coordinates if these are not needed.

\subsection{Normal forms}\label{normalforms}
Clearly special attention is needed in the neighborhoods of singularities of the vector 
field $X$. 
Recall the description of the associated real flow and 
the normal forms for vector fields, see Figure \ref{forma-normal}. 
Several authors have contributed with proofs, see
J. A. Jenkins \cite{Jenkins} ch.~3, 
J. Gregor \cite{Gregor1}, \cite{Gregor2},
L. V. Ahlfors \cite{Ahlfors-Conformal} p.\,111, 
L. Brickman {\it et al.} \cite{BrickmanThomas}, 
K. Strebel \cite{Strebel} ch.~III,
A. Garijo {\it et al.} \cite{GarijoGasullJarque1}.
Further discussion on the origin of normal forms can be found
in \cite{Guillot2},
\cite{AP-MR} pp. 133, 159, and  references therein.

\begin{proposition}\label{prop:normalforms}
Let $X$ be a meromorphic vector field germ on $(\CC, 0)$
having a pole or zero  at the origin. Up to local biholomorphism 
$X$ is holomorphically equivalent to one of the following normal forms.  

\noindent 1) 
For a pole of order\footnote{
We convene that the order/multiplicity of a pole is to be negative.
} 
$- \kappa \leq -1$ 
$$
\frac{1}{z^\kappa} \del{}{z}.
$$

\noindent 2)
For  simple zero 
$$
\frac{z}{\lambda }  \del{}{z}, \ \ \ 
\lambda= Res(\omega_X, 0).
$$

\noindent 3)
For  zero  of order $s \geq 2$
$$
\frac{z^s}{ 1+ \lambda z^{s-1} }  \del{}{z}, 
\ \ \ \lambda= Res(\omega_X, 0).
$$ 
\hfill $\Box$
\end{proposition}

In the case of functions the normal forms are simpler. 

\begin{lemma}\label{LemaFormaNormal}
Let $g$ be a meromorphic function germ on $(\CC, 0)$ having a pole, regular point or zero at the origin.
Up to local biholomorphism 
$g$ is equivalent to the following normal form

\centerline{
$z^k$ \ \ \   for 
$k \in \ZZ$
the multiplicity of $f$  at $0$.}
\hfill
$\Box$
\end{lemma}
\begin{figure*}[htbp]
\centering
\includegraphics[width=0.87\textwidth]{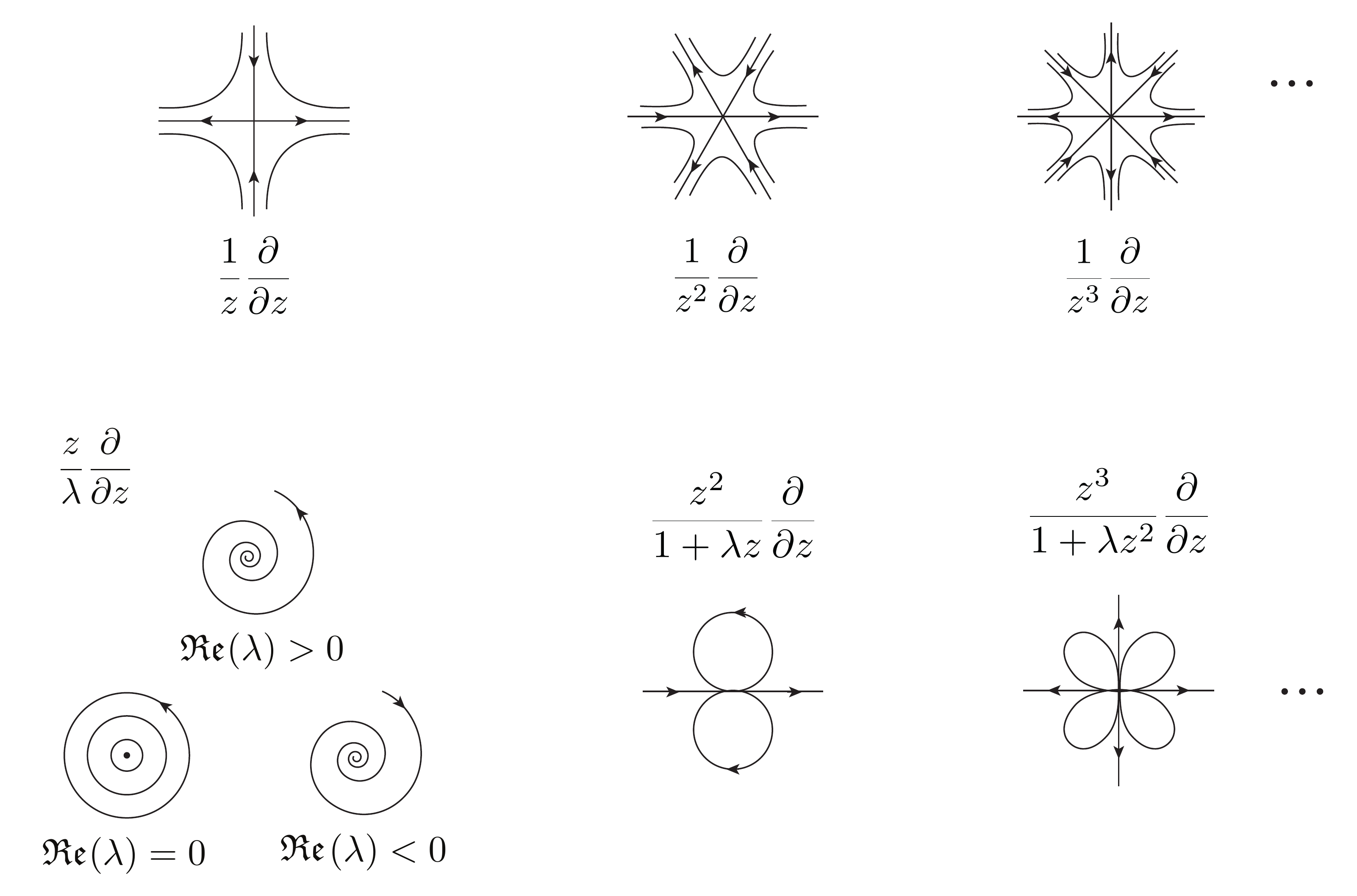}
\caption{Phase portrait and normal forms of $X$ at a pole or zero. 
Top row: for a pole of order $-\kappa\leq -1$, the phase portrait has $2(\kappa+1)$ separatrices 
arriving or leaving the pole. 
Bottom row: simple zeros and zeros of order $s\geq 2$, here $\lambda=Res(\omega_{X},0)$. 
For simple zeros, the phase portrait is the pullback via $\Psi_{X}(z)=\lambda \log z$ of the constant vector 
field $Y(t)=\del{}{t}$. 
For $s\geq2$ the trajectories of $X$ form a flower with exactly $2(s-1)$ petals. 
For further details see Examples \ref{holoVFonSphere}, \ref{ejemploPolos}, \ref{ejemploCerosSimples}, 
and \ref{ejemploCerosMultiplesResiduoNoCero}.
} 
\label{forma-normal}
\end{figure*}

\subsection{Geometry and dynamics of the pullback}\label{GeoDynPullback}
Recall that 
a \emph{covering} $\Psi:V\to W$ is a continuous surjective mapping such that for all $w\in W$, there exists an 
open set $U\ni w$ in $W$ with the characteristic that $\Psi^{-1}(U)$ is a disjoint union of open sets $O\subset V$ 
each of which satisfies that $\Psi:O\to U$ is a homeomorphism.

A \emph{branched or ramified covering} $\Psi:V\to W$ is a covering except at a finite number of 
points of $W$. Said points are known as \emph{branch points or ramification points}.
\begin{remark}[Geometrical interpretation of the pullback]
In the setting of Lemma \ref{pullbackformula}, it is now natural to consider singular complex analytic
maps $\Upsilon:M\longrightarrow N$ as singular complex analytic ramified covering maps, thus providing a 
geometric interpretation of the pullback:
\begin{center}
\emph{The trajectories of $X=\Upsilon^{*}Y$ are the pre--images, via $\Upsilon$ of the trajectories of $Y$.}
\end{center}
\end{remark}

Considering biholomorphisms as the covering maps we obtain:

\begin{example}[A $PSL(2,\CC)$--action on vector fields]\label{FLTvectorfield} 
Let $Y(t)=g(t)\del{}{t}$ be a complex vector field on $\CW_t$ 
and consider the pullback via a biholomorphism, 
$T(z)=\frac{az+b}{cz+d}$ with $ad-bc\neq0$,
of $\CW$.
Then
$$
T^*\big(Y(t)\big)(z)= \frac{(cz+d)^2}{(ad-bc)} g(T(z))\del{}{z} . 
$$ 
A useful particular case is when $T(z)=(1/z)$, so that 
$(T^*Y)(z)= -z^2 g\left(1/ z\right)\del{}{z}$.
\end{example}

Considering finitely ramified coverings, 
we unify several known examples,
poles are the simplest.

\begin{example}[Poles of order $-\kappa\leq-1$]\label{ejemploPolos}
The vector field

\centerline{
$X(z)=\frac{1}{z^{\kappa}}\del{}{z}$ \ \ on $(\CC, 0)$, 
}
\noindent
has a pole of order $-\kappa$ at the origin. 
The natural diagram is 
\begin{center}
\begin{picture}(180,80)(0,10)


\put(-45,70){$\big(\CC_z , X(z)=\frac{1}{z^{\kappa} }\del{}{z} \big) $}

\put(115,70){$\big( \{ \frac{z^{\kappa+1}}{\kappa+1}-t= 0 \},
\pi^*_{X,2}(\del{}{t})\big)$}

\put(108,73){\vector(-1,0){60}}
\put(70,80){$\pi_{X,1}$}

\put(133,65){\vector(0,-1){30}}
\put(138,47){$ \pi_{X,2} $}

\put(40,63){\vector(2,-1){73}}
\put(0,39){$\Psi_{X}(z)=
\frac{z^{\kappa+1}}{\kappa+1}  $}

\put(115,20){$\big(\CW_t,\del{}{t}\big) $ \ .}
\end{picture}
\end{center}

\noindent 
$X(z)$ is the pullback  
via the distinguished polynomial parameter 

\centerline{
$\Psi_{X}(z)=\int^{z} \zeta^{\kappa}\,d\zeta =\frac{z^{\kappa+1}}{\kappa+1}$ 
}
\noindent
of the constant vector field $Y(t)=\del{}{t}$.
The $2(\kappa +1)$ separatrices arrive or leave the pole 
in finite real time. 
See top row of Figure \ref{forma-normal} and Figure \ref{diagrama-triangulo-polo-doble} (a).
\begin{figure*}[htbp]
\centering
\includegraphics[width=\textwidth]{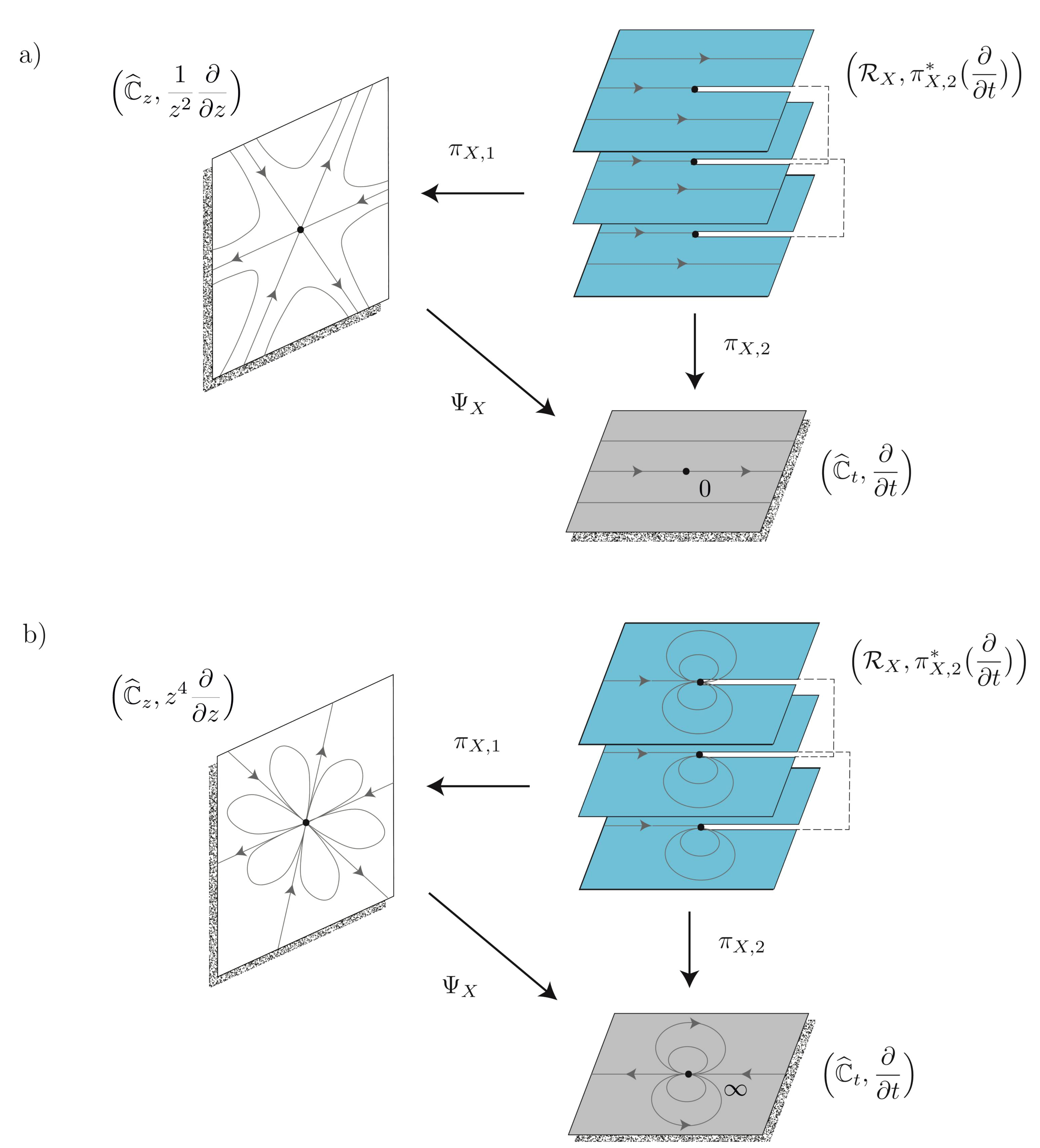}
\caption{The diagrams for the Riemann surfaces $\R_{X}$ corresponding to 
$\frac{1}{z^2}\del{}{z}$ and $z^{4}\del{}{z}$.}
\label{diagrama-triangulo-polo-doble}
\end{figure*}
\end{example}

Allowing $\log (z)$ as a ramified covering we obtain zeros.

\begin{example}[Simple zeros $s=1$]\label{ejemploCerosSimples}
The vector field

\centerline{
$X(z)=\frac{1}{\lambda} z \del{}{z}$ \ \ on $\CC$, $\lambda \in \CC^*$, 
}
\noindent
has a simple zero at the origin. 
The natural diagram is 
\begin{center}
\begin{picture}(180,80)(0,10)

\put(-45,70){$\big(\CC_z , X(z)= \frac{1}{\lambda z} \del{}{z} \big) $}

\put(115,70){$\big( \{ \lambda \log(z)-t= 0 \},
\pi^*_{X,2}(\del{}{t})\big)$}

\put(108,73){\vector(-1,0){60}}
\put(70,80){$\pi_{X,1}$}

\put(133,65){\vector(0,-1){30}}
\put(138,47){$ \pi_{X,2} $}

\put(40,63){\vector(2,-1){73}}
\put(-10,39){$\Psi_{X}(z)= \lambda \log(z) $}

\put(115,20){$\big(\CW_t,\del{}{t}\big) $ \ .}
\end{picture}
\end{center}

\noindent 
$X(z)$ is the pullback  
via the distinguished additively automorphic parameter  

\centerline{
$\Psi_{X}(z)=\int^{z}_1 \frac{\lambda}{\zeta }\,d\zeta =
\lambda \log(z)$ 
}
\noindent
of the constant vector field $Y(t)=\del{}{t}$. 
See bottom left of Figure \ref{forma-normal}.
\end{example}

\begin{example}[Multiple zeros $s\geq 2$, with residue $\lambda\in\CC$]
\label{ejemploCerosMultiplesResiduoNoCero}
The vector field

\centerline{
$X(z)=\frac{z^{s}}{1+\lambda z^{s-1}}\del{}{z}$, \ \
$\lambda\in\CC$, 
}
\noindent
has a zero or order $s\geq 2$ at the origin, and residue $\lambda$. 
The natural diagram is 
\begin{center}
\begin{picture}(180,80)(25,10)

\put(-70,70){$\big(\CC_z , X(z)= \frac{z^{s}}{1+\lambda z^{s-1}}  
\del{}{z} \big) $}

\put(100,70){$\Big( \{ \frac{1}{(1-s) z^{s-1}} + \lambda \log(z)-
\frac{1}{1-s} -t=0\},\pi^*_{X,2}(\del{}{t})\Big)$}

\put(98,73){\vector(-1,0){50}}
\put(70,80){$\pi_{X,1}$}

\put(133,60){\vector(0,-1){28}}
\put(138,47){$ \pi_{X,2} $}

\put(40,63){\vector(2,-1){73}}
\put(-25,44){$\Psi_{X}(z)= $}
\put(-30,29){$\frac{1}{(1-s) z^{s-1}} + \lambda \log(z) -\frac{1}{1-s}$}

\put(115,20){$\big(\CW_t,\del{}{t}\big) $ \ .}
\end{picture}
\end{center}

\noindent 
$X(z)$ is the pullback of 
via the distinguished additively automorphic parameter 

\centerline{
$\Psi_{X}(z)=\int^{z}_1 \big( \frac{1}{\zeta^s} + 
\frac{\lambda}{\zeta} \big)\,d\zeta =
\frac{1}{(1-s) z^{s-1}} + \lambda \log(z) - \frac{1}{1-s}$ 
}
\noindent
of the constant vector field $Y(t)=\del{}{t}$.
See Figure \ref{forma-normal}  and Figure \ref{diagrama-triangulo-polo-doble} (b).
\end{example}

Considering infinitely ramified covering maps we obtain essential singularities.

\begin{example}[Essential singularity]
\label{ejemplo-diagrama-esencial}
The entire vector field $X(z)=\e^{z} \del{}{z}$ has an essential singularity at $\infty\in\CW$ 
and no zeros or poles on $\CC$. 
It is the pullback via the $\infty$ to $1$ ramified covering, $\Psi_{X}(z)=-\e^{-z}$, 
of the constant vector field $\del{}{t}$. 
It is considered to be the simplest example of an essential singularity.
See Figure \ref{diagrama-esencial}.
\begin{figure*}[htbp]
\centering
\includegraphics[width=\textwidth]{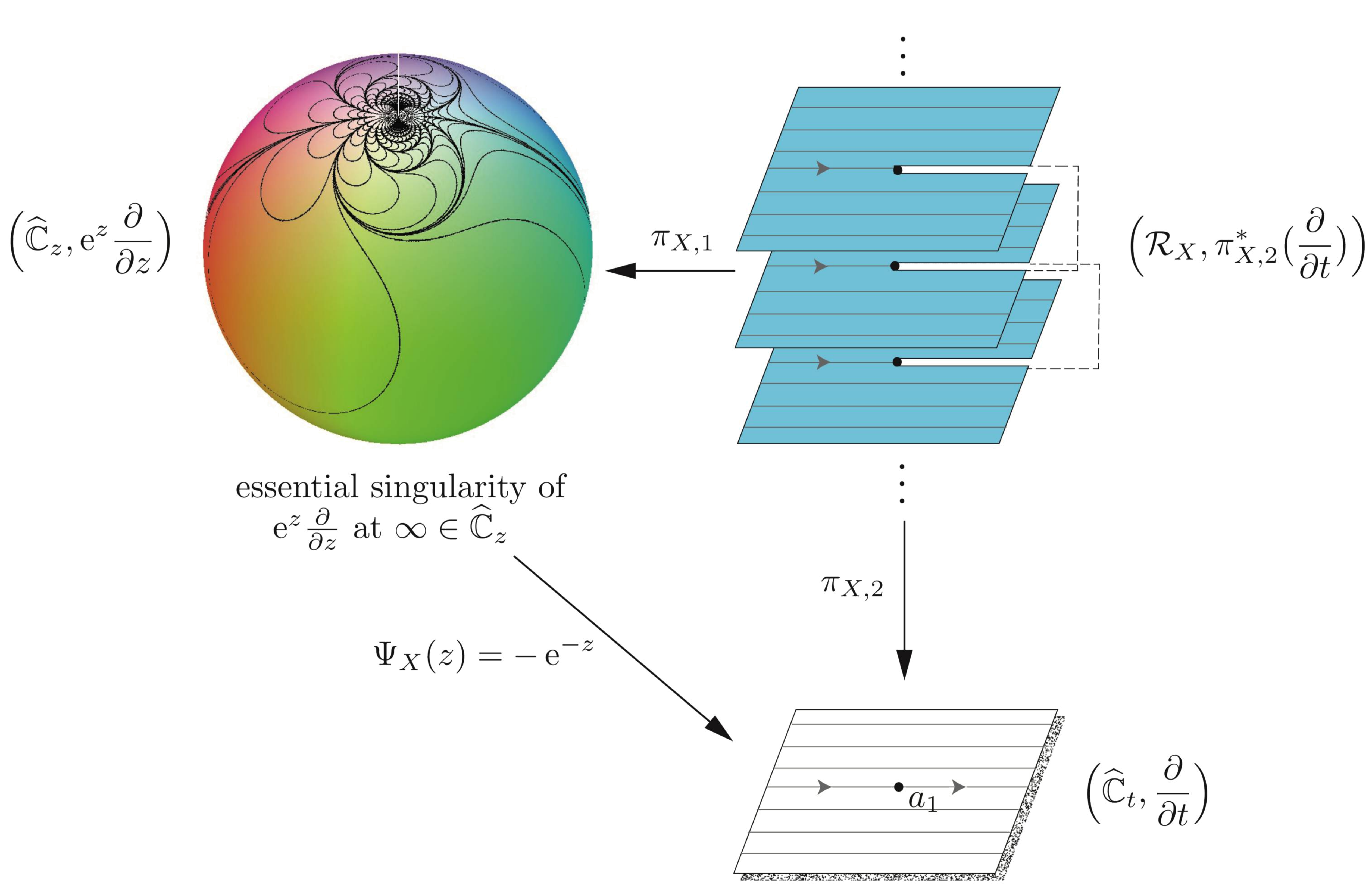}
\caption{The diagram for a Riemann surface $\R_{X}$ corresponding to
$\e^{z}\del{}{z}$.}
\label{diagrama-esencial}
\end{figure*}
\end{example}

\subsection{Every singular complex analytic vector field is the pullback of  $\del{}{t}$ and of $-w\del{}{w}$}
\label{pullbacks}
\subsubsection{Every singular complex analytic vector field admits a global flow box}
\label{Flowbox} 
A rather surprising result is the following, first presented in \cite{AP-MR}. 

\begin{theorem}[Pullbacks of $\del{}{t}$]
\label{VectorFieldsHaveFlowBox}
Every singular complex analytic vector field $X$ on $M$ is the pullback of $\del{}{t}$ on
$\CW_t$ via an additively automorphic singular complex analytic map $\Psi_{X}: M \longrightarrow \CW_t$,  i.e. 
$${\Psi_{X}}_{*}X=\del{}{t}.$$
Moreover
\begin{equation}\label{pullbackdt}
\Psi_{X}(z)=\int_{z_0}^z \omega_{X},
\end{equation}
for $z_0 \in M'$, and is a single--valued singular complex analytic function if and only if
the periods and residues of $\omega_{X}$ are zero,  i.e.
\begin{equation}\label{condicionunivaluada}
\int_{\gamma} \omega_{X}=0  \ \ \ \ \text{ for  every}  \ \  [\gamma] \in
H_1(M', \ZZ).
\end{equation}
\hfill\qed
\end{theorem}

Note that \eqref{condicionunivaluada} implies that the zeros of $X$, if there are any, are
of order $\geq2$ (\emph{i.e.} the poles of $\omega_{X}$ are non--simple).
It also says that in this case $\omega_{X}$ is an exact differential 1--form.

The above result is of particular interest from the 
point of view of the theory of differential equations. 
Explicitly, the trajectories of $X$ 
are mapped to trajectories of $\del{}{t}$, which are horizontal straight 
lines in $\CC$ (recall that we speak of real trajectories as in \eqref{ecdifflow}).
This is a remarkable property:  the existence and uniqueness of trajectories for singular complex analytic 
vector fields admits a very simple proof, that is also global on Riemann surfaces $M$. 
Recall that for real analytic vector fields, in general
we can only obtain ``long flow boxes'', see \cite{Palis} ch.~3 \S1.
Moreover, 
$X$ does not present limit cycles, see
\cite{Lukashevich}, \cite{Benzinger}, \cite{Sabatini}.
Note that the point  $z_0$ can be a pole of $X$
and the fact remains that ${\Psi_X}_{*}X = \del{}{t}$ provides a flow box around $z_0$, 
as in Example \ref{ejemploPolos}.

\subsubsection{Every singular complex analytic vector field is a 
pullback of $\lambda w\del{}{w}$}
\label{Newton}
Analogously one has the following result, also first presented in \cite{AP-MR}.
\begin{theorem}[Pullbacks of  
$\frac{1}{\lambda} w \del{}{w}$]
\label{everythingisnewton}
Every singular complex analytic vector field $X$ on $M$ is 
the pullback of $\frac{1}{\lambda} w \del{}{w}$ on $\CW_w$, for any $\lambda\in\CC^{*}$, 
via a  (possibly multivalued) singular complex analytic map
$\Phi_{X}(z): M \longrightarrow \CW_w$, i.e. 
$${\Phi_X}_{*}\big(X\big)(w)=\frac{1}{\lambda} w\del{}{w}.$$
Moreover
\begin{equation}\label{pullbacktdt}
\Phi_{X}(z)=\exp \Big( \frac{1}{\lambda} \int_{z_0}^z \omega_{X} \Big),
\end{equation}
for $z_0 \in M'$, and is a single--valued singular complex analytic function if and only if
the periods and residues of $\omega_{X}$ are integer multiples of
some complex number $\Pi \in \CC^*$, i.e.
\begin{equation}
n \Pi = \int_{\gamma}\omega_{X} \ \ \ {\it for \ [\gamma] \in H_1(M',
\ZZ)} , \   \ n \in \ZZ \ .
\end{equation}
\end{theorem}

\begin{proof}
By Lemma \ref{pullbackformula},
$X(z)$ is the pullback of $\frac{1}{\lambda} w\del{}{w}$ iff
$$
f(z)=\frac{1}{\lambda}\frac{\Phi_{X}(z)}{\Phi_{X}'(z)},
$$
for some (possibly multivalued) singular complex analytic  function $\Phi_{X}$ on $M$. Equivalently
$1/(\lambda f(z))$ is the logarithmic derivative of $\Phi_{X}(z)$, hence upon integration and exponentiation one obtains the expression \eqref{pullbacktdt} for $\Phi_{X}$. \\
On the other hand, by virtue of the explicit form of $\Phi_{X}$ and since $\exp$ is $2\pi i$--periodic, the differential form $\omega_X$
has a residue or period that is not an integer multiple of $\Pi=2\pi i \lambda$ if and only if $\Phi_{X}(z)$ is multivalued.
\end{proof}

\subsubsection{Correspondence between $\del{}{t}$ and $-w\del{}{w}$}\label{holofamilycorresp}

As was seen in the previous two sections, one can express any singular complex analytic vector field $X$ on $M$ as the pullback of either 
$\del{}{t}$ or $-w\del{}{w}$ via the (possibly multivalued) functions
$$
\Psi_{X}(z)= \int_{z_{0}}^z \omega_{X}
\qquad\text{  and  }\qquad
\Phi_{X}(z)=\exp \Big( - \int_{z_0}^z \omega_{X} \Big),
$$
respectively.

\begin{corollary}\label{corodiagramabasico}
Let $X$ be a singular complex analytic vector field on $M$. 
It is the pullback of $\del{}{t}$ and it is also the pullback of $- w\del{}{w}$. 
Furthermore $\del{}{t}$ is the pullback via $\exp(-t)$ of $-w\del{}{w}$ and we have the following 
commutative diagram
\begin{center}
\begin{picture}(210,55)
\put(-110,20){\eqref{diagrama-basico}}

\put(17,8){$ \Big(\CW_{t},\del{}{t}\Big) $}
\put(62,12){\vector(1,0){70}}
\put(73,0){$\exp(-t)$}

\put(135,8){$ \Big(\CW_{w},-w\del{}{w}\Big) $,}
\put(122,34){$\Phi_{X}$}
\put(113,42){\vector(1,-1){20}}

\put(76,45){$(M,X)$}
\put(43,34){$\Psi_{X}$}
\put(70,42){\vector(-1,-1){20}}

\end{picture}
\end{center}
\noindent
for the respective (possibly multivalued) singular complex analytic functions.
\hfill \qed
\end{corollary}

Corollary \ref{corodiagramabasico} is exemplified in Figure \ref{DosZeroSimple} for the entire vector field 
$X(z)=-\e^{-z\e^{z}}\del{}{z}$ which has a class 2 essential singularity at $\infty\in\CW$, see \cite{AP-MR}.
Yet it is the pullback of both $\del{}{t}$ and $-w\del{}{w}$.
\begin{figure*}[htbp]
\centering
\includegraphics[width=\textwidth]{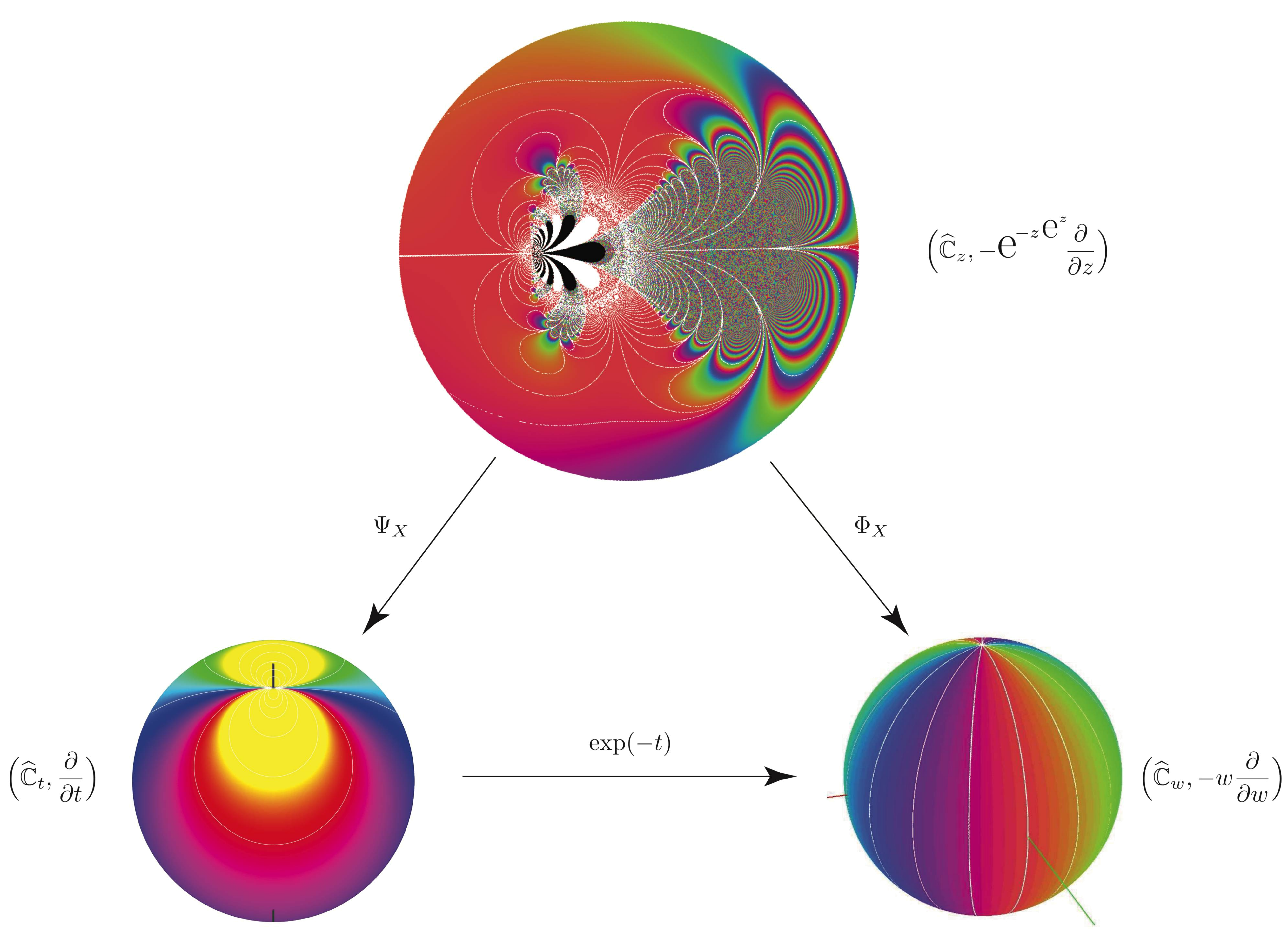}
\caption{Example of Diagram \ref{diagrama-basico} for $X(z)=-\e^{-z\,\e^z} \del{}{z}$. }
\label{DosZeroSimple}
\end{figure*}

\section{Newton vector fields: pullbacks of $-w\del{}{w}$}
\label{camposnewtonianos}
We now recall a special kind of complex analytic vector fields that were first studied in the 80's, 
by M.\,W.\,Hirsch, S.\,Smale and M.\,Shub (\cite{HirschSmale}, \cite{Smale1}, \cite{Shub}, \cite{DedieuShub}). 
The concept of Newton vector field was introduced together with that of  Newton graphs that arise from studying 
Newton's method of root finding for a complex polynomial. 
Taking their definition as a guide we have:
\begin{definition}
A singular complex analytic vector field $X(z)=f_{\tt j}(z)\del{}{z}$ on $M$ is said to be a \emph{Newton vector 
field} if it can be represented as 
$$X(z)=-\frac{\Phi_{\tt j}(z)}{\Phi'_{\tt j}(z)}\del{}{z},$$
for some (possibly multivalued) singular complex analytic function $\{\Phi_{\tt j}\}$ on $M$.
\end{definition}
\noindent
From the definition and the results of \S\ref{Newton} one has as an immediate consequence that 
$\Phi=\Phi_{X}$ and the following.
\begin{corollary}
Every singular complex analytic vector field $X$ on an arbitrary Riemann surface $M$ is a Newton vector field of 
a suitable $\Phi_{X}$.
\end{corollary}

\begin{example}\label{tanvectorfield}
The complex vector field 
$$X(z)=-\tan(z)\del{}{z}$$
is the pullback via $\Phi_{X}(z)=\sin(z)$ of the complex vector field $Y(t)=-w\del{}{w}$, so in fact it is a Newton 
vector field.
\end{example}

\medskip
On the other hand, recalling the geometrical interpretation of the pullback, and considering 
$\Phi_{\alpha}:M_{z}\longrightarrow\CW_{w},$ a singular complex analytic ramified covering over the sphere, one can 
construct Newton vector fields as the pullback via $\Phi_{\alpha}$ of the radial vector field $-w\del{}{w}$ on the 
Riemann sphere. Note also that the composition of ramified coverings is still a ramified covering, hence:
\begin{corollary}
Let $S_{w}$, $N_{t}$ and $M_{z}$ be Riemann surfaces and 
let $Y(w)=g(w)\del{}{w}$ be a singular complex analytic vector field on $S_{w}$.
Further suppose that
$$\Phi_{1}:M_{z}\longrightarrow N_{t}, \quad \Phi_{2}:N_{t}\longrightarrow S_{w}$$ 
are singular complex analytic ramified coverings, then $X(z)=f(z)\del{}{z}$ is the pullback via 
$\Phi=\Phi_{2}\circ\Phi_{1}$ of $Y(w)$ if and only if
$$f(z)=\frac{g(\Phi(z))}{\Phi'(z)}=\frac{g(\Phi_{2}(\Phi_{1}(z)))}{\Phi_{2}'(\Phi_{1}(z))\ \Phi_{1}'(z)}.$$
\end{corollary}
\begin{proof}
The proof is a direct consequence of the chain rule and Lemma \ref{pullbackformula}.
\end{proof}
So by considering $Y(w)=-w\del{}{w}$ we can construct many Newton vector fields.
Some examples follow.

\begin{example}\label{examplecoshPlus2}
The complex vector field 
$$X(z)=-(\cosh(z)+1)\del{}{z}$$
is obtained from $-w\del{}{w}$ via pullback with $\Phi_{X}=\Phi_3\circ\Phi_2\circ\Phi_1$, where 
$\Phi_3(s)=\e^{s}$, $\Phi_2(t)=\frac{t-1}{t+1}$ and $\Phi_1(z)=\e^{z}$. It has zeros at  $z_n=i(2n+1)\pi,\ n\in\ZZ$.
\end{example}

\begin{example}\label{exampleExp}
The complex vector field 
$$X(z)=\e^{z}\del{}{z}$$ 
is obtained from $-t\del{}{t}$ via pullback with $\Phi_{X}=\Phi_2\circ\Phi_1$, where $\Phi_{2}(w)=\e^{w}$ 
and $\Phi_{1}(z)=\e^{-z}$. It has an essential singularity at $\infty$.
\end{example}

\begin{example}\label{exampleExp3}
The complex vector field 
$$X(z)=\frac{\e^{z^{3}}}{3z^{3}-1}\del{}{z}$$ 
is obtained from $-w\del{}{w}$ via pullback with $\Phi_{X}=\Phi_2\circ\Phi_1$, where 
$\Phi_{2}(w)=w\e^{-w^{3}}$ and $\Phi_{1}(z)=\e^{-z}$.
It has an essential singularity at $\infty$ and 3 poles on the finite plane.
\end{example}

\begin{remark}
It should be noted that when $\Psi$ is rational, the $\omega$--limit set of almost any\footnote{
There will be a finite number of trajectories, corresponding to the separatrices of the poles, where the 
$\omega$ limit set will not be a zero.
} trajectory for the flow of

\centerline{
$X(z)= \frac{\Psi(z)}{\Psi'(z)}\del{}{z}$  
}

\noindent
on $M$ corresponds to the zeros of $\Psi$, in agreement with the work of M.\,W.\,Hirsch, 
S.\,Smale and M.\,Shub (\cite{HirschSmale}, \cite{Smale1}, \cite{Shub}, \cite{DedieuShub}).

\noindent
However, when $\Psi$ is not rational, the $\omega$--limit could be an isolated essential singularity 
(Examples  \ref{exampleExp} and \ref{exampleExp3}) or an accumulation point of poles or zeros 
(Examples \ref{tanvectorfield} and \ref{examplecoshPlus2}).
\end{remark}

In Table \ref{relaciones} we present a summary of the different objects and their relations, encountered so far. 
There we can observe that the residue of $\omega_{X}$ plays an important role in the description of the objects. 
This was already observed in \cite{AP-MR}, \S5.7.
\begin{table}
\caption{Relationship between vector fields, 1--forms, 
distinguished parameter and the Newton covering map germs.}
\begin{center}
\begin{tabular}{|c|c|c|c|}
\hline
Complex analytic & Complex analytic & Distinguished & Newton \\
vector field & 1--form & parameter & covering map \\
$X(z)=f(z)\del{}{z}$ & $\omega_{X}=\frac{dz}{f(z)}$ & $\Psi_{X}(z)=\int\limits^{z}\omega_{X}$ & 
$\Phi_{X}(z)=\e^{-\Psi_{X}(z)}$ \\
\hline
\hline
pole of & zero of & zero of & \\
order $-\kappa\leq -1$ & order $\kappa$ & order $\kappa+1$ & \\
$\frac{1}{z^{\kappa}} \del{}{z}$ & $z^{\kappa}\, dz$ & $\frac{1}{\kappa+1} z^{\kappa+1}$ & 
$\e^{-\frac{1}{\kappa+1} z^{\kappa+1}}$  \\[6pt]
\hline
simple zero & simple pole & & \\
$\frac{1}{\lambda}z\del{}{z}$ & $\frac{\lambda}{z}\,dz$ & $\lambda\log(z)$ & $z^{-\lambda}$ \\[6pt]
\hline
multiple zero $s\geq2$ & multiple pole &  & \\
$\frac{z^{s}}{1+\lambda z^{s-1}}\del{}{z}$ & 
$\left(\frac{1}{z^{s}}+\frac{\lambda}{z}\right) dz$ & 
$\Big(\frac{1}{(1-s) z^{s-1}}\hfill$ & $z^{-\lambda}\e^{(s-1)z^{s-1}}$  \\
 & $\lambda=Res(\omega_{X},0)$ & $ \hfill+\ \lambda\log(z)\Big)$ & \\[6pt]
\hline
essential & essential & & \\
singularity at $\infty$ & singularity at $\infty$ & & \\
$\e^{P(z)}\del{}{z}$ & $\e^{-P(z)} dz$ & $\int\limits^{z}\e^{P(\zeta)}d\zeta$ & 
$\e^{-\int\limits^{z}\e^{P(\zeta)}d\zeta}$  \\[6pt]
\hline

\end{tabular}
\end{center}
\label{relaciones}
\end{table}%

\section{Visualization of Newton vector fields}
\label{metodo}
We now move on to describe the actual method that will enable us to solve for the trajectories of (and hence 
visualize) Newton vector fields. As mentioned in the introduction, the idea behind is that there are two 
auxiliary real valued functions: one whose level curves correspond to trajectories (streamlines) of the 
complex vector field (in differential equations this auxiliary function is called a first integral), and the 
other one which is linear along the trajectories, 
hence providing the parametrization of the solution. 

\subsection{The fundamental observation}\label{fund-obs}
Let
$$X(z)=f(z)\del{}{z}=-\frac{\Phi_X (z)}{\Phi^{\prime}_X (z)}\del{}{z}$$ 
be a Newton vector field, recall that the trajectories $z(\tau)$ are the solutions to
\begin{equation}\label{flujoNewt}
\left\{
\begin{aligned}
z'(\tau)&=-\frac{\Phi_X \big(z(\tau)\big)}{\Phi^{\prime}_X \big(z(\tau)\big)},\\
z(0)&=z_{0},
\end{aligned}
\right.
\end{equation}
for $\tau\in\RR$, as in \eqref{ecdifflow}.

\begin{lemma}[Fundamental observation]\label{lemasoluciones}
A trajectory $z(\tau)$ of $X$ satisfies
\begin{equation}\label{keyFact}
\Phi_X \big(z(\tau)\big) = \Phi_X (z_{0})\e^{-\tau}
\end{equation}
if and only if $z(\tau)$ is a trajectory, passing through $z_{0}$, of the Newton vector field corresponding 
to $\Phi_X $.
\end{lemma}
\begin{proof}
The proof follows from implicitly differentiating the equation
$$\Phi_X \big(z(\tau)\big) = \Phi_X (z_{0})\e^{-\tau},$$
that is 
$$
\Phi^{\prime}_X \big(z(\tau)\big) z'(\tau) = - \Phi_X (z_{0})\e^{-\tau}
$$
so that 
$$
z'(\tau)=- \frac{\Phi_X \big(z(\tau)\big)}{\Phi^{\prime}_X \big(z(\tau)\big) },
$$
hence $z(\tau)$ is indeed a solution to \eqref{flujoNewt}.
\end{proof}

The fundamental observation can be understood in terms of the pullback as follows: \emph{The trajectory $z(\tau)$ is the flow of the pullback, via $\Phi$, of the field $-w\del{}{w}$ (whose trajectories are straight lines, parametrized by $\e^{-\tau}$, that start at $\infty$ and end at $0$ in $\CW_{t}$).}
See Figure \ref{PullbackCampoNewt}.
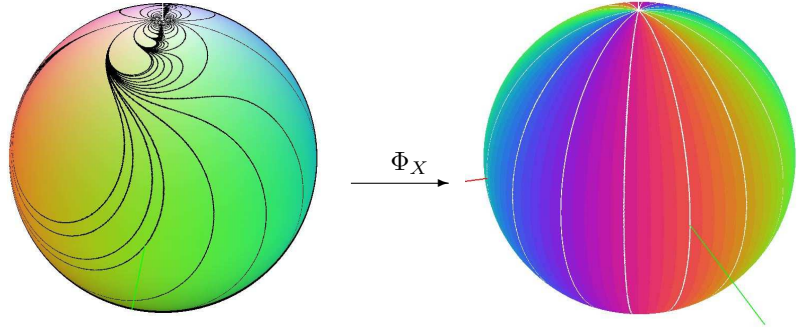
\begin{figure}[htbp]
\centering
	\includegraphics[height=0.3\textwidth]{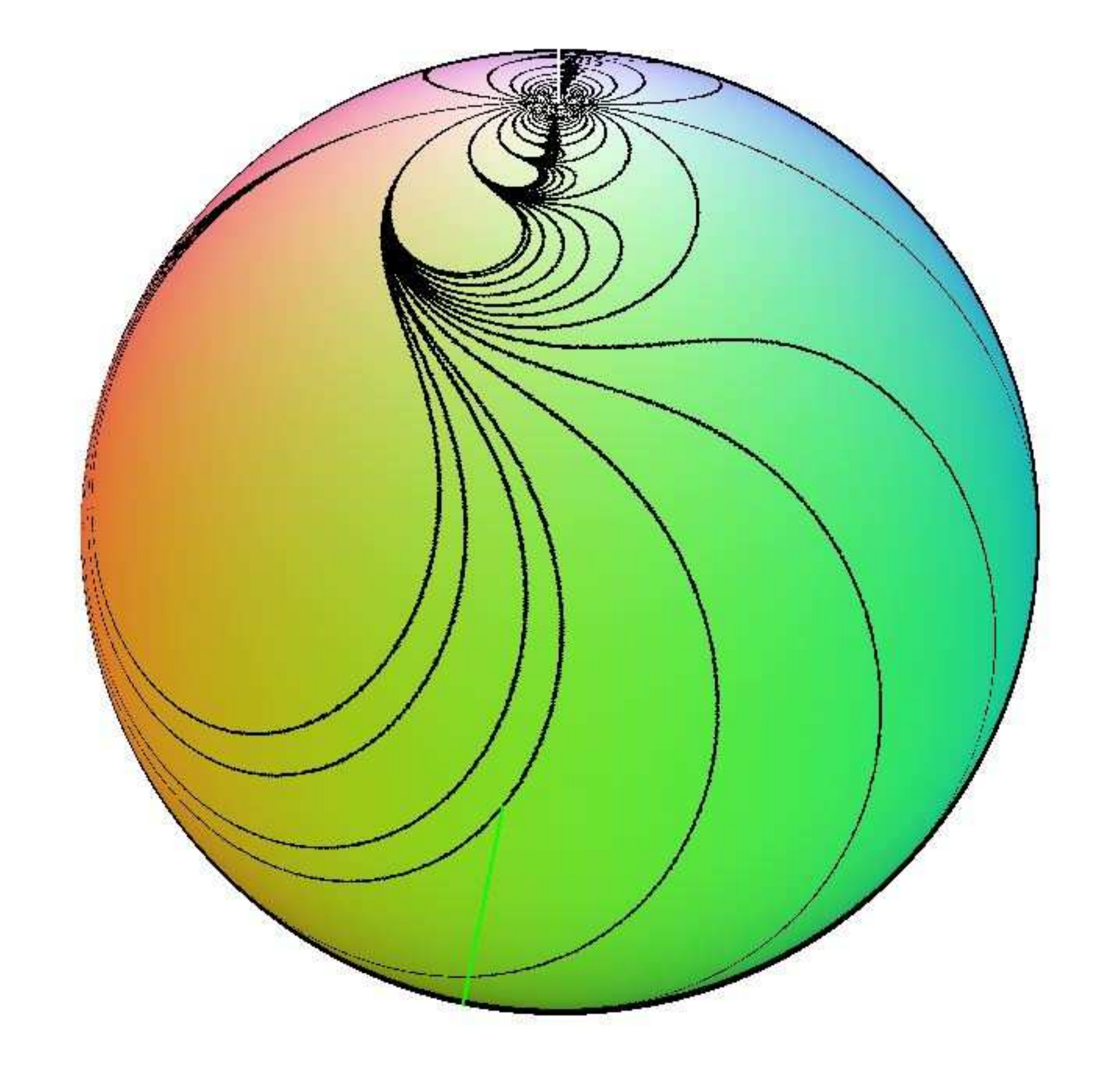}
	\begin{picture}(40,55)
	\put(15,60){$\Phi_X $}
	\put(0,55){\vector(1,0){37}}
	\end{picture}
	\includegraphics[height=0.3\textwidth]{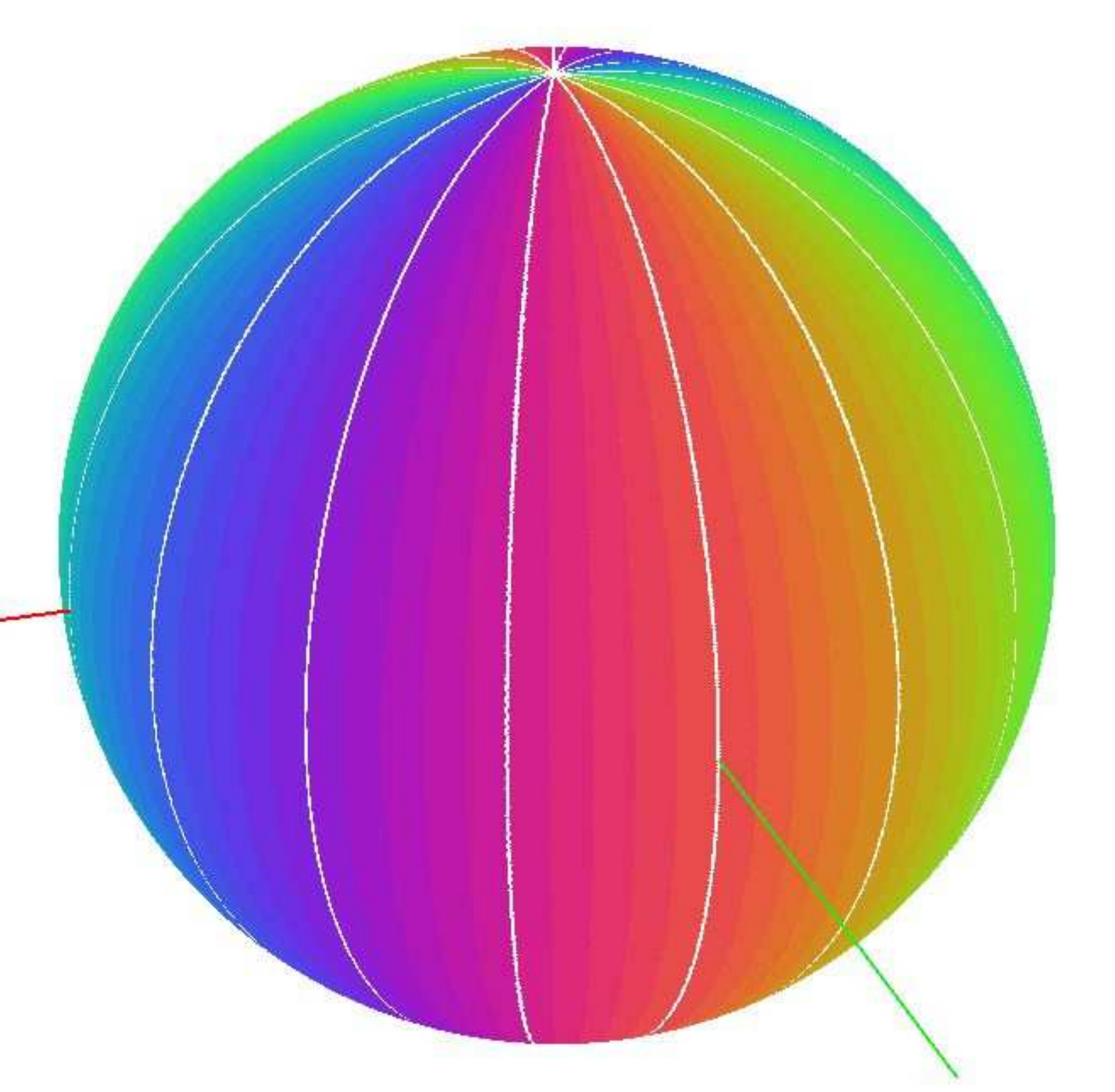}
\caption{The trajectories $z(t)$ correspond to trajectories of $-w\del{}{w}$ under the covering map $\Phi_X $.}
\label{PullbackCampoNewt}
\end{figure}

Consider now the Newton vector field normal to $X$,
$$X^{\perp}(\zeta)=i\ f(\zeta)\del{}{\zeta}=-\frac{{\widetilde{\Phi}_X }(\zeta)}
{{\widetilde{\Phi}}'_X (\zeta)}\del{}{\zeta},$$
with ${\widetilde{\Phi}_X }$ being its corresponding covering map. 
Then one has
$$-i\left(\log\Phi_X \right)'=-i\frac{\Phi'_X }{\Phi_X }=\frac{{\widetilde{\Phi}}'_X }{{\widetilde{\Phi}}_X }
=\left(\log{\widetilde{\Phi}_X }\right)',$$
and it follows that
 ${\widetilde{\Phi}_X }=\Phi^{-i}_X $.
Thus taking $\log (z)$ on both sides we obtain that 
\begin{equation}\label{anteslemma}
\begin{aligned}
\rho(z) \doteq \log\abs{\widetilde{\Phi}_X (z)} &= \argp{ \Phi_X (z) } \\
\theta(z) \doteq \argp{ {\widetilde{\Phi}}_X (z) } &= -\log\abs{\Phi_X (z)}.
\end{aligned}
\end{equation}

\begin{proposition}[Solving $\dot{z}=-\frac{\Phi_X (z)}{\Phi^{\prime}_X (z)}$]
\label{solNewton}
Let $z(\tau)$ be a trajectory of a singular complex analytic field 
$$X(z)=-\frac{\Phi_X (z)}{\Phi^{\prime}_X (z)}\del{}{z},$$
then
\begin{enumerate}[label=\arabic*)]
\item $\argp{ \Phi_X (z(\tau)) } = \argp{ \Phi_X (z_{0})}$ is constant along the 
trajectories $z(\tau)$ of $X$, and
\item $\log\abs{\Phi_X (z(\tau))}=-\tau+\log\abs{\Phi_X (z_{0})}$ is linear along the 
trajectories $z(\tau)$ of $X$,
\end{enumerate}
where $z_{0}=z(0)$.
\end{proposition}
\begin{proof}
For (1) use \eqref{anteslemma} 
and the fundamental observation (Lemma \ref{lemasoluciones}). 

\noindent
For (2), a similar argument works. 
\end{proof}

As a direct consequence we can now state the main theorem related to the visualization of 
singular complex analytic vector fields.

\theoremstyle{plain}
\newtheorem*{theo1}{Theorem \ref{TeoremaVisual}}

\begin{theo1}[Visualization of singular complex analytic vector fields]
Let $X(z)=f(z)\del{}{z}$ be a singular complex analytic vector field on a Riemann surface $M$, 
and let $z(\tau)$ denote any trajectory of $X$ on $M\backslash Sing(X)$.
Then there exist two (probably multivalued) functions $\rho,\theta:M\backslash Sing(X)\longrightarrow \RR$ 
such that
\begin{enumerate}[label=\arabic*)]
\item The real valued function $\rho$ is constant along $z(\tau)$. 
Hence in order to visualize the trajectories that pass through the point $z_0\in M$, 
one needs only plot the level curve $\rho(z)=\rho(z_0)$.
\item The real valued function $\theta$ defines a natural time parametrization along the trajectory $z(\tau)$. 
In other words, if $z(\tau)$ is the trajectory passing through $z_0$ at $\tau=0$, then the point $z(\tau_1)$ is 
given by the intersection of the curves $\rho(z)=\rho(z_0)$ and $\theta(z)=\theta(z_0) + \tau_1$.
\end{enumerate}
Moreover, $\rho$ and $\theta$ can be expressed in terms of the distinguished parameter $\Psi_X$ and the 
Newton map $\Phi_X$ as:
\begin{equation}\label{dosmaneras}
\begin{array}{c}
\rho(z)=\argp{ \Phi_X (z) }=-\Im{\Psi_{X}(z)},\\
\theta(z)=-\log\abs{\Phi_X (z)}=\Re{\Psi_{X}(z)}.
\end{array}
\end{equation}
\end{theo1}

\begin{proof}
(1) Is a direct consequence of Proposition \ref{solNewton}.a where $\rho(z)=\argp{ \Phi_X(z) }$.

(2) Follows directly from Proposition \ref{solNewton}.b, {\it i.e.} the linear behaviour of 
$\theta(z)=\log\abs{\Phi_X (z)}$ along the trajectory $z(\tau)$: 

\centerline{
$\theta(z(\tau_{1}))=-\log\abs{\Phi_X (z(\tau_{1}))}=-\log\abs{\Phi_X (z_{0})} + \tau_{1},$}

\noindent
and from the description of the trajectory $z(\tau)$:

\centerline{
$\rho(z(\tau_{1}))=\rho(z_{0}).$}

Finally, Corollary \ref{corodiagramabasico}, particularly Diagram \ref{diagrama-basico}, shows that $\rho$ 
and $\theta$ can be expressed in terms of 
$\Psi_X$ or $\Phi_X$ so \eqref{dosmaneras} follows.
\end{proof}

\begin{remark}
1. The auxiliary functions $\rho$ and $\theta$ are known as \emph{constants of motion}, 
\emph{integrals of motion}, or \emph{first integrals}.
On the other hand, recall that a generic real analytic vector field does not have a first integral. 

\noindent
2. The acute reader will note that $\rho$ and $\theta$ determine visualizations of the functions $\Phi_X$ and 
$\Psi_X$ as polar and rectangular representations respectively.
\end{remark}
This last remark provides a counterpart for Theorem \ref{TeoremaVisual}

\theoremstyle{plain}
\newtheorem*{theo2}{Theorem \ref{visualizationPsiPhi}}

\begin{theo2}[Visualization of singular complex analytic functions]
\hfill
\begin{enumerate}[label=\arabic*)]
\item Let $\Psi:M\longrightarrow\CW$ be a singular complex analytic function.
Then the phase portraits of 

$X(z)=\frac{1}{\Psi^{\prime}(z)}\del{}{z}$ \
provides the level curves of 
$-\Im{\Psi}$,

$X^{\perp}(z) \doteq \frac{i}{\Psi^{\prime}(z)}\del{}{z}$ \ 
provides the level curves of $\Re{\Psi}$.

\item
Let $\Phi:M\longrightarrow\CW$ be a singular complex analytic function.
Then the phase portraits of 

$X(z)=-\frac{\Phi(z)}{\Phi^{\prime}(z)}\del{}{z}$ \
provides the level curves of  $\argp{\Phi}$, 

$X^{\perp}(z) \doteq -i\frac{\Phi(z)}{\Phi^{\prime}(z)}\del{}{z}$ \ 
provides the level curves of $-\log{\abs{\Phi}}$.
\end{enumerate}
\hfill \qed
\end{theo2}

\begin{remark}[Solution for the flow of $X$ and no propagation of error along the trajectories of $X$]
The above theorem shows that we are not only visualizing the flow of the singular complex analytic vector field, 
but we are in fact 

\centerline{\emph{completely solving the system of differential 
equations that define the flow \,}}

\centerline{\emph{including parametrization
of the singular complex analytic vector field $X(z)$.}}

\noindent 
Contrast this with the usual 
visualization techniques where information relating to the 
parametrization is not observed.

\noindent
Moreover, \emph{the solutions are exact} up to the numerical round--off errors incurred by the precision of the 
mathematical routines used in the implementation. In other words, there is no error propagated along the 
trajectories of $X$.
\end{remark}

\subsection{The algorithms}\label{implementation}
Given a singular complex analytic vector field
$$X(z)=f(z)\del{}{z}=-\frac{\Phi_X (z)}{\Phi^{\prime}_X (z)}\del{}{z},$$ 
according to Theorem \ref{TeoremaVisual} and \eqref{dosmaneras}, we require the plotting of 
the level curves of 
the real valued function $\rho(z)=\argp{ \Phi_X (z) }=-\Im{\Psi_{X}(z)}$. 

This can be done both on the (complex) plane and more generally on the Riemann surface $M$. 
Moreover, it will be convenient to plot \emph{strip flows}\footnote{A
concept due to \cite{Benzinger}, see \cite{AP-MR} \S11. }, where, given an interval $[a,b)\subset\RR$ we define 
the \emph{strip flow associated to $[a,b)$} as 
\begin{equation}\label{levelband}
B_{[a,b)}=\big\{z\in M\ \vert\ \rho(z)\in [a,b) \big\}\subset M.
\end{equation}
In this way the border is 
the level curves  $\rho(z)=a$ and $\rho(z)=b$, which correspond to 
trajectories of the singular complex analytic 
vector field $X$. 
Note that $B_{[a,b)}$ can be a multiply connected subset of $M$.

In the following algorithm we present the case of $M$ being either the plane $\CC$ or the Riemann sphere 
$\CW$, the case of a general Riemann surface is similar and is further discussed in \S\ref{generalizaciones}.
 
We can use strip flows to visualize the streamlines on the plane or 
Riemann sphere, using the following:

\smallskip 
\noindent{ \bf Visualization algorithm.}

\noindent 
(p refers to the plane, s to the sphere.)

\begin{enumerate}
\item[1)] Partition $\RR$ into intervals $\big\{ [a_{\iota},b_{\iota}) \big\}$ and select a color $C_{\iota}$ for 
each interval $[a_{\iota},b_{\iota})$ of the partition.

\item[2p)] Choose a rectangular region of the plane (say $R=[x_{min},x_{max}]\times[y_{min},y_{max}]$) 
where the visualization is to take place, and a window size of say $N$ by $M$ pixels, then subdivide the 
rectangular region of the plane into $N\times M$ rectangular regions of size 
$\Delta x=(x_{max}-x_{min} )/N$ by $\Delta y=(y_{max}-y_{min})/M$, note that each of these 
rectangular regions corresponds to a pixel on the window.

\item[2s)] A triangulation of the Riemann sphere is constructed using a recursive algorithm that ensures that the 
triangulation is almost uniform (start with a octahedron with vertices on the sphere and recursively add a vertex 
at the center of each triangle and then normalize the vertex to obtain a better triangulation for the sphere, 
consult \cite{BourkeTriangulacionEsfera} for more details).

\item[3p)] For each rectangle of the subdivision of $R$ (a pixel) 
calculate its center $z\in R\subset\CC$.

\item[3s)] For each triangle in the triangulation of the sphere, one finds the barycentre $z^{*}$ and using 
stereographic projection we identify the corresponding $z\in\CC$.

\item[4)] We proceed to calculate $\rho(z)$ (or $\rho(z^*)$).

\item[5)]  Since $\RR$ is partitioned into intervals 
$\rho(z)\in [a_{\iota},b_{\iota})$ for some $\iota$,
we proceed to color the pixels (triangles), on the plane (Riemann sphere), corresponding to  
$z$ (or $z^{*}$) with the color $C_{\iota}$.
\end{enumerate}
\begin{remark}
Note that steps (2p)  and (2s) 
ensure that both the resolution on the rectangular region $R$ of the plane and on  
the Riemann sphere, are uniform.
\end{remark}

\subsubsection{Plotting specific level curves}

Due to the fact that there are an infinite number of trajectories intersecting a given zero, it is very easy to identify 
the zeros of the corresponding Newton vector field. 

\noindent
However the actual position of the poles are not so easily identified: by the nature of the poles 
(recall Proposition \ref{prop:normalforms} and Example \ref{ejemploPolos}), a pole of order 
$-\kappa\leq-1$ has exactly $(2\kappa+2)$ separatrices. 

\noindent
Similarly, it was shown in \cite{AP-MR} definition 4.11, 
that for essential singularities there exists trajectories that are analogs 
of the separatrices of poles: 
the \emph{horizontal asymptotic paths} 
that have as $\alpha$ or $\omega$ 
limit set the essential singularity.

\noindent
With the above in mind and recalling that singular complex analytic vector fields on $\CW$ can not have 
limit cycles, we then can make the following observation.

\begin{remark}[Plotting separatrices and horizontal asymptotic paths]
\label{importanciatrayectoriasola}
In order to correctly visualize the phase portrait of singular complex analytic vector fields, it is convenient to plot 

\noindent $\bullet$
the separatrices of a pole and 

\noindent $\bullet$
the horizontal asymptotic paths for essential singularities,

\noindent {\it i.e.} specific 
trajectories of the field (that is specific level curves of 
the real valued function $\rho$). 

\noindent 
Recall the classical ideas of L. Markus and H. Benzinger 
on the decomposition of the phase portraits using the above
kind of specific trajectories, 
\cite{Markus}, 
\cite{Benzinger}, more recently 
\cite{Branner-Dias}, \cite{AP-MR} \S11.

\end{remark}

Suppose we want to plot a specific trajectory, say one that passes through or has $\alpha$ or $\omega$ limit set 
a point $z_{0}$ on the plane. Then by Theorem \ref{TeoremaVisual} we need to plot the level curve of $\rho$ 
corresponding to the value $\rho_{0}=\rho(z_{0})$. Since the fundamental unit we are using for the visualization 
is the subdivision by rectangles on the rectangle $R\subset\CC$ (or the triangulation of the Riemann sphere 
$\CW$), we need to color those rectangles (or triangles) that intersect the level curve. For this note that
\begin{itemize}[label=$\bullet$]
\item the function $\widehat{\rho}(z,z_{0})=\rho(z)-\rho(z_{0})$ can be zero in the interior of a rectangle (triangle), 
even though at the center (or barycentre) it could be different from zero,
\item we only want to color rectangles (or triangles) that intersect the level curve $\widehat{\rho}(z,z_{0})=0$.
\end{itemize}

\noindent 
To achieve this we have the following:

\smallskip 
\noindent{ \bf Visualization algorithm for specific level curves.}

\begin{enumerate}[label=(\arabic*)]
\item Once again we find the center $z$ of each rectangle (or the barycentre $z^{*}$ of each triangle), 
as well as the 
maximum distance $\delta$ from the center (or barycentre) to each of the vertices. Note that in the case of the 
plane, $\delta=\sqrt{{\Delta x}^{2}+{\Delta y}^{2}}$, since the basic unit is a rectangle.

\item Using the fact that the gradient of a real valued function points in the direction of maximum growth, let 
$\hat{e}=\frac{\nabla \widehat{\rho}(z,z_{0})}{\abs{\nabla \widehat{\rho}(z,z_{0})}}=\frac{\nabla \rho(z)}
{\abs{\nabla \rho(z)}}$ 
be the unit vector that points in the direction of maximum growth of $\widehat{\rho}(\cdot,z_{0})$ at $z$ 
(or $z^{*}$). 
Note that $\widehat{\rho}(\cdot,z_{0})$ is a $C^{\infty}$ function since it is the real part of an analytic function.
In fact, $\hat{e}=\frac{X^{\perp}(z)}{\abs{X^{\perp}(z)}}$  
$\left(\text{or } \hat{e}=\frac{X^{\perp}(z^*)}{\abs{X^{\perp}(^*)}}\right)$.

\item Recalling that the sign of $\widehat{\rho}$ changes if and only if $\widehat{\rho}$ assumes the value zero, 
we consider the product 
$$\widehat{\rho}(z+\delta \hat{e},z_{0})\ \widehat{\rho}(z-\delta \hat{e},z_{0}),$$
and so the level curve $\widehat{\rho}=0$ intersects the rectangle (or triangle) if the above product is less than 
zero, so we color the rectangle (or triangle) associated to $z$ (or $z^{*}$) if this happens.
\end{enumerate}
In this way those rectangles (or triangles) that intersect the level curve $\widehat{\rho}=0$ are colored.

\begin{remark}
Note that the above algorithm:
\begin{enumerate}[label=(\arabic*)]
\item Is optimal with respect to resolution for the case of the plane: that is one obtains the best 
possible \emph{observable} resolution. If one would increase the size of the rectangular mesh 
$\Delta x\times\Delta y$ one would observe pixelation, and if one decreases the size of the 
rectangular mesh $\Delta x\times\Delta y$, then no gain in resolution would be observed, 
since the size of the rectangular mesh would be smaller that the size of each pixel on the screen.

\item Is not optimal for the case of the sphere: since in this case the actual observed resolution will depend on 
the particular parameters (viewpoint, distance of the camera to the sphere, etc.) used in the visualization of the 
sphere as a 3D object on the screen.

\item However, given a specific resolution, the algorithms ensure that no error is made as to which streamlines 
intersect the chosen basic rectangles (or triangles) that specify the resolution, 
hence for the chosen resolution the visualization is the best possible.
\end{enumerate}
\end{remark}

\subsection{Parallelization of the visualization algorithms}\label{parrallelization}
It is clear that this method is a prime candidate for parallelization, due to the fact that the visualization scheme 
for a particular pixel does not depend on the neighboring pixels\footnote{This type of parallelization is known as 
\emph{embarrassingly parallelizable}.}. Hence a simple parallelization scheme where blocks of pixels are 
assigned to distinct processors can be readily implemented. An extensive analysis and implementation of this is 
yet to be done and will be presented elsewhere.

\section{Analytic recognition of the ramified covering $\Phi_{X}$}
\label{meromorphic}
As was shown in the previous section, Newton vector fields benefit from the visualization scheme just presented, 
and \emph{since all singular complex analytic vector fields are Newton vector fields}, this makes the visualization 
scheme presented much more appealing. 

Of course one still has to be able to explicitly calculate the real valued functions $\rho(z)$ and $\theta(z)$ of 
Theorem \ref{TeoremaVisual} in order to make use of the method.

In this aspect, the first author {\it et.al.} shows in \cite{AP-2} that there is a large class of vector fields 
meromorphic on the plane, for which it is possible to explicitly construct the ramified covering $\Phi_{X}$ 
characterizing the vector field as a Newton vector field (and hence $\rho(z)$ and $\theta(z)$ of Theorem 
\ref{TeoremaVisual} can also be calculated explicitly). In the same work they show again by an explicit 
construction, that all doubly periodic (elliptic) vector fields (and hence vector fields on the torus) 
for which it is possible to analytically recognize the ramified covering $\Psi_{X}$ 
(see \S\ref{InvariantVectFields} for an example of an elliptic vector field).

We recall these results in this section.
Let $\mathscr{F}$ be the family of functions that satisfy the requirements of Cauchy's Theorem on Partial 
Fractions (see \cite{AP-2} for further details).
Suppose that $g\in\mathscr{F}$, denote by
$$G_{k}(z)=\sum_{j=1}^{N_{k}}\frac{a_{j k}}{(z-b_{k})^{j}},$$
the principal part of $g(z)$ at the pole $b_{k}$ of order $N_{k}$; and let 
$$P_{k}(z)=\sum_{j=0}^{p}\widetilde{a}_{j k}z^{j},$$
denote the corresponding polynomials (in case $p>-1$). Then in \cite{AP-2} it was proved that:

\begin{theorem}
\label{trascendental}
Let $f(z)\in\mathscr{F}$. Then there is a meromorphic (possibly multivalued) function 
\begin{equation}
\Phi_{X}(z)=\prod_{k=1}^{\infty}\left[(z-b_k)^{-A_{k}}\e^{q_k\left(\frac{1}{z-b_k}\right)} \e^{Q_k (z)}\right],
\end{equation}
where $q_{k}(z)$ and $Q_{k}(z)$ are unique polynomials with $q_{k}(0)=0$ and $Q_{k}(0)=0$,
such that
$$f(z)=-\frac{\Phi_{X}'(z)}{\Phi_{X}(z)}.$$
\hfill \qed
\end{theorem}

As a quick, and illustrative, example of the explicit construction of the $\Phi_{X}$ defining the Newton vector field, we consider the case of rational functions: let
\begin{equation}\label{defcamp}
f(z)=-\frac{p(z)}{q(z)}
\end{equation}
with $q,p\in\CC[z]$ without common factors, and $p$ monic. In particular consider 
\begin{equation}\label{defpq}
p(z)=\prod_{j=1}^{J}(z-z_{j})^{m_{j}}, \quad
q(z)=b\prod_{k=1}^{K}(z-s_{k})^{n_{k}},
\end{equation}
where $m_{j},n_{k}\in\ZZ$ and $b\in\CC$ are constants.

We then obtain theorem 2.3 of \cite{Benzinger} as a corollary of our Theorem \ref{trascendental} (for the particular case of rational functions):
\begin{theorem}[H.\,E.\,Benzinger \cite{Benzinger}]\label{BenzingerRational}
$f(z)$ is a rational function as in \eqref{defcamp} that satisfies
$$f(z)=-\frac{\Phi_{X}(z)}{\Phi_{X}'(z)}$$
if and only if there exist unique polynomials $P_{j}$, with $P_{j}(0)=0$, 
and unique constants $A_{j}\in\CC$, $j=1,\dots,J$, such that 
\begin{equation}
\Phi_{X}(z)=C\e^{P_{0}(z)}\prod_{j=1}^{J}(z-z_{j})^{A_{j}}\e^{P_{j}(\frac{1}{z-z_{j}})},
\end{equation}
where $C\in\CC$ is an arbitrary constant.
\end{theorem}
\noindent
Since an alternative (direct) proof is instructive and short we provide a sketch of proof.
\begin{proof}[Sketch of proof]

\noindent
($\Rightarrow$):
Consider $-\frac{1}{f}=\frac{q}{p}$, with $p$ and $q$ polynomials as described above.
By Euclid's division algorithm we have that 
$\frac{q}{p}=d+\frac{r}{p}$,
with $d,r\in\CC[z]$ of degree less than that of $p$.
Next consider the partial fraction decomposition of $\frac{r}{p}$:
\begin{equation}
\frac{r(z)}{p(z)}=\sum_{j=1}^{J}\left(\frac{A_{j 1}}{z-z_{j}} + \sum_{k=2}^{m_{j}}\frac{A_{j k}}{(z-z_{j})^k}\right),
\end{equation}
and then integrate explicitly so that finally by exponentiation and renaming $A_{j}=A_{j 1}$ we have
\begin{equation}
\Phi_{X}(z)=C\e^{P_{0}(z)}\prod_{j=1}^{J}(z-z_{j})^{A_{j}}\e^{P_{j}(\frac{1}{z-z_{j}})}.
\end{equation}

\noindent
($\Leftarrow$):
This is an elementary calculation left to the reader.
\end{proof}

We have then an explicit characterization, and more importantly, a method of calculating the $\Phi_{X}(z)$, in the 
case that the complex analytic vector field is defined by a rational 
function $f(z)$. For the more general case 
when $f(z)$ is a meromorphic function, one uses the 
Mittag--Leffler expansion instead of the partial fraction 
decomposition in the above sketch of proof. 
Note that if a residue of $dz/f(z)$ is not an integer then $\Phi_{X}(z)$ is 
in fact a multivalued function.

As examples of these explicit calculations consider the following.
\begin{example}\label{exampratfunc}
1. Let $f(z)=\frac{z(2z-i)^{2}}{(2z+i)^{2}}$ then we find that
$$\frac{q(z)}{p(z)}=\frac{(2z+i)^{2}}{z(2z-i)^{2}},$$ 
so that $d(z)=0$ and by partial fraction decomposition and explicit integration
$$\Phi_{X}(z)=\frac{\e^{\frac{4i}{2z-i}}}{z}.$$
This shows that the complex vector field 
$$X(z)=\frac{z(2z-i)^{2}}{(2z+i)^{2}}\del{}{z}$$ 
is a pullback of $-w\del{}{w}$ via $\Phi_{X}$, hence it is a Newton vector field.

\noindent
2. Of course we can also consider the opposite case: 
suppose we know that  
$$\Phi_{X}(z)=\e^{1/z}(z-1),$$
then we can find the rational vector field
$$X(z)=-\frac{z^{2}(z-1)}{z^{2}-z+1}\del{}{z}.$$
\end{example}

\begin{example}
The vector field
$$X(z)=-\frac{1}{\sqrt{2}+1}\frac{z(z-1)}{z-\frac{\sqrt{2}}{\sqrt{2}+1}}\del{}{z},$$
is a Newton vector field that comes from pullback of $-w\del{}{w}$ via 
$$\Phi_{X}(z)=z^{\sqrt{2}}(z-1).$$
\end{example}

The elliptic case is handled via the following theorem, where $\sigma$ and $\zeta$ are 
the Weierstrass sigma 
and zeta functions respectively  (see again \cite{AP-2} for further details)
\begin{theorem}\label{ellipticNewtonian}
Let $f(z)$ be an elliptic function with fundamental periods $2w_{1}$ and $2w_{3}$, let $b_{1},...,b_{r}$ be the 
poles of f(z) in the fundamental period parallelogram. Suppose $b_{k}$ is of order $\beta_{k}$, with principal part
$$G_{k} = \frac{A_{1 k}}{z - b_{k}} + \cdots + \frac{A_{\beta_{k} k}}{(z - b_{k})^{\beta_{k}}}\quad (k = 1, . . . , r).$$ 
Then
$f(z) = - \frac{\Phi_{X}'(z)}{\Phi_{X}(z)},$
and in fact there exist constants $C'$ and $C$ such that
$$\Phi_{X}(z) = C'\e^{-C z}\prod_{k=1}^{r}\left\{\sigma(z - b_{k})^{-A_{1 k}} \exp\left[\sum^{\beta_{k}}_{j=2} 
(-1)^{j} \frac{A_{j k}}{(j-1)!} \zeta^{(j-2)}(z - b_{k})\right]\right\}.$$
\hfill \qed
\end{theorem}

\begin{example}
The elliptic vector field 
$$X(z)=-\frac{\wp(z)}{\wp'(z)}\del{}{z},$$
is of course a Newton vector field with $\Phi_{X}(z)=\wp(z)$.
\end{example}

\section{Examples of the visualization of singular complex analytic vector fields on $\CW$}\label{essential}

Let $X(z)=f(z)\del{}{z}$ be a singular complex analytic vector field on the Riemann sphere $\CW$. If $f(z)$ has 
only poles or zeros in $\CC$ ($f$ is meromorphic on $\CC$), then three cases arise:
\begin{enumerate}[label=(\arabic*)]
\item $f(z)$ has a pole or a regular point at $\infty\in\CW$, in which case $f(z)$ is a rational function.
\item $f(z)$ has an essential singularity at $\infty\in\CW$.
\item $\infty\in\CW$ is an accumulation point of zeros or poles of $f(z)$.
\end{enumerate}
In all three cases one needs to find the ramified covering $\Phi_{X}(z)$ explicitly and proceed to calculate the 
real valued function $\rho(z)=\argp{ \Phi_{X}(z) }$ in order to plot its level curves.

\subsection{Rational vector fields on $\CW$}
The first cases are handled as in \S\ref{meromorphic}, that is Theorem \ref{BenzingerRational} provides us with 
the explicit ramified coverings that allows us to visualize the corresponding vector field. 

\subsubsection{The case of $X(z)=\frac{z(2z-i)^{2}}{(2z+i)^{2}} \del{}{z}$}
Considering Example \ref{exampratfunc}.1, the phase portrait of the rational vector field  

\centerline{$X(z)=\frac{z(2z-i)^{2}}{(2z+i)^{2}}  \del{}{z}$}
 
\noindent
can be visualized in Figure \ref{ratfunc}, both in the plane $\CC$ and on the Riemann sphere $\CW$. 
In this case we have a simple zero at the origin, a double zero at $\frac{i}{2}$, a double pole at $-\frac{i}{2}$, and 
$\infty\in\CW$ is a first order zero. 

\noindent
An advantage of plotting strip flows is that it also provides us with information regarding the parametrization of 
the solutions, hence for instance one can see that the trajectories that approach the zero at $\frac{i}{2}$ are slow 
compared to the trajectories that approach the pole at $-\frac{i}{2}$.
\begin{figure*}[htbp]
\centering
\includegraphics[height=0.48\textwidth]{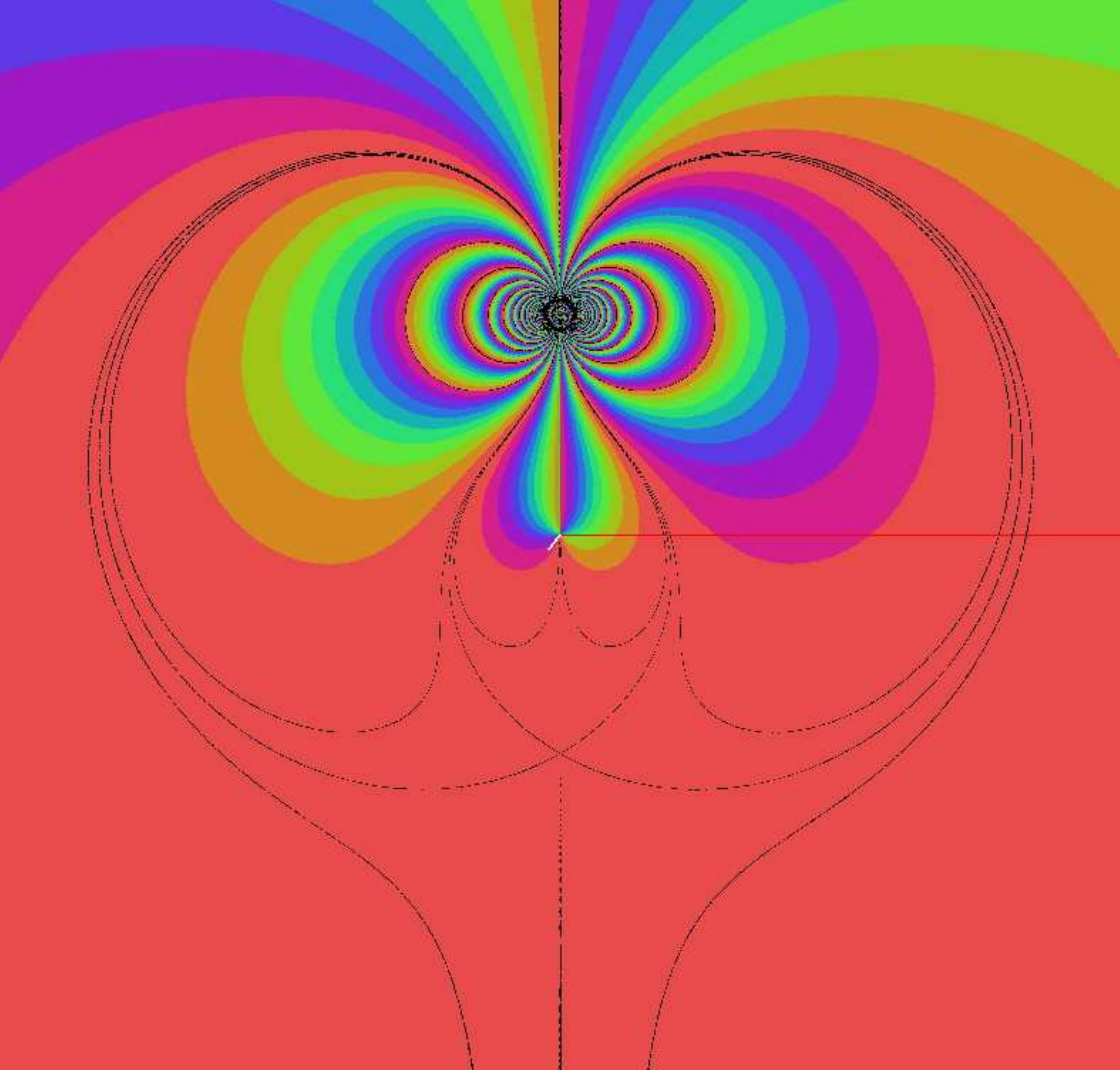}
\includegraphics[height=0.48\textwidth]{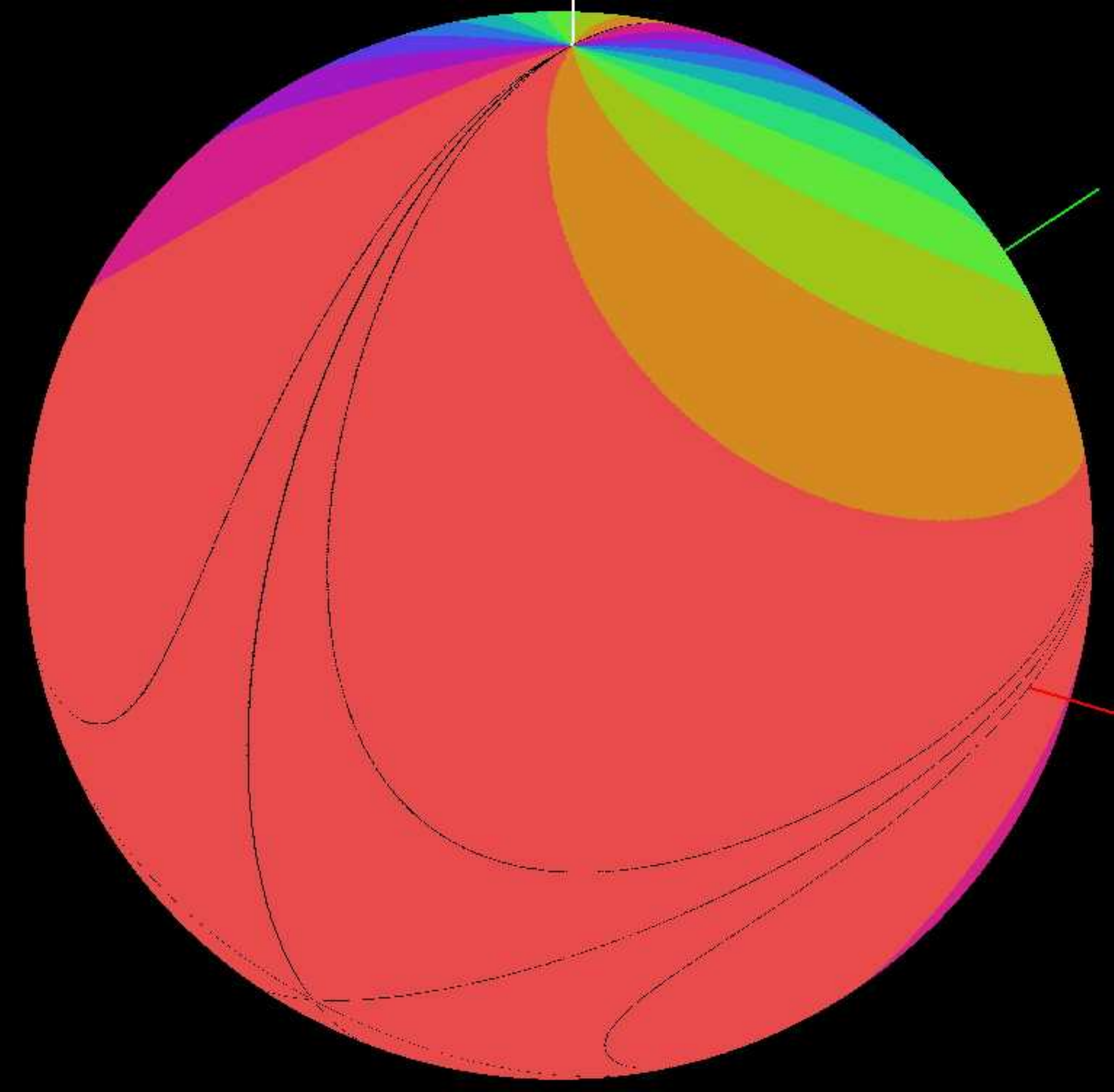}
\caption{Visualization of  the field $X(z)=\frac{z(2z-i)^{2}}{(2z+i)^{2}} \del{}{z}$.
(a) Shows a vicinity of the origin and one can observe a simple zero at the origin, a double zero at $\frac{i}{2}$, 
and a double pole at $-\frac{i}{2}$. (b) Shows the vector field on the Riemann sphere where one can see a 
simple zero at $\infty\in\CW$ and an order 2 pole at  $-\frac{i}{2}$.}
\label{ratfunc}
\end{figure*}

\subsubsection{The case of $X(z)=-\frac{z^{2}(z-1)}{z^{2}-z+1}\del{}{z}$}
This corresponds to Example \ref{exampratfunc}.2, that is the rational vector field 

\centerline{$X(z)=-\frac{z^{2}(z-1)}{z^{2}-z+1}\del{}{z}$.}
 
\noindent
In Figure \ref{campoPolos} we present the visualization of the phase portrait. As can be observed the borders of 
the strip flows correspond to streamlines of the field. We are also plotting some explicit trajectories that pass 
through the poles (of order $-1$ at the roots of $z^{2}-z+1$) and zeros (of order 2 at 0 and of order 1 at 1 
and $\infty$) of the vector field. 
\begin{figure*}[htbp]
\centering
\includegraphics[height=0.43\textwidth]{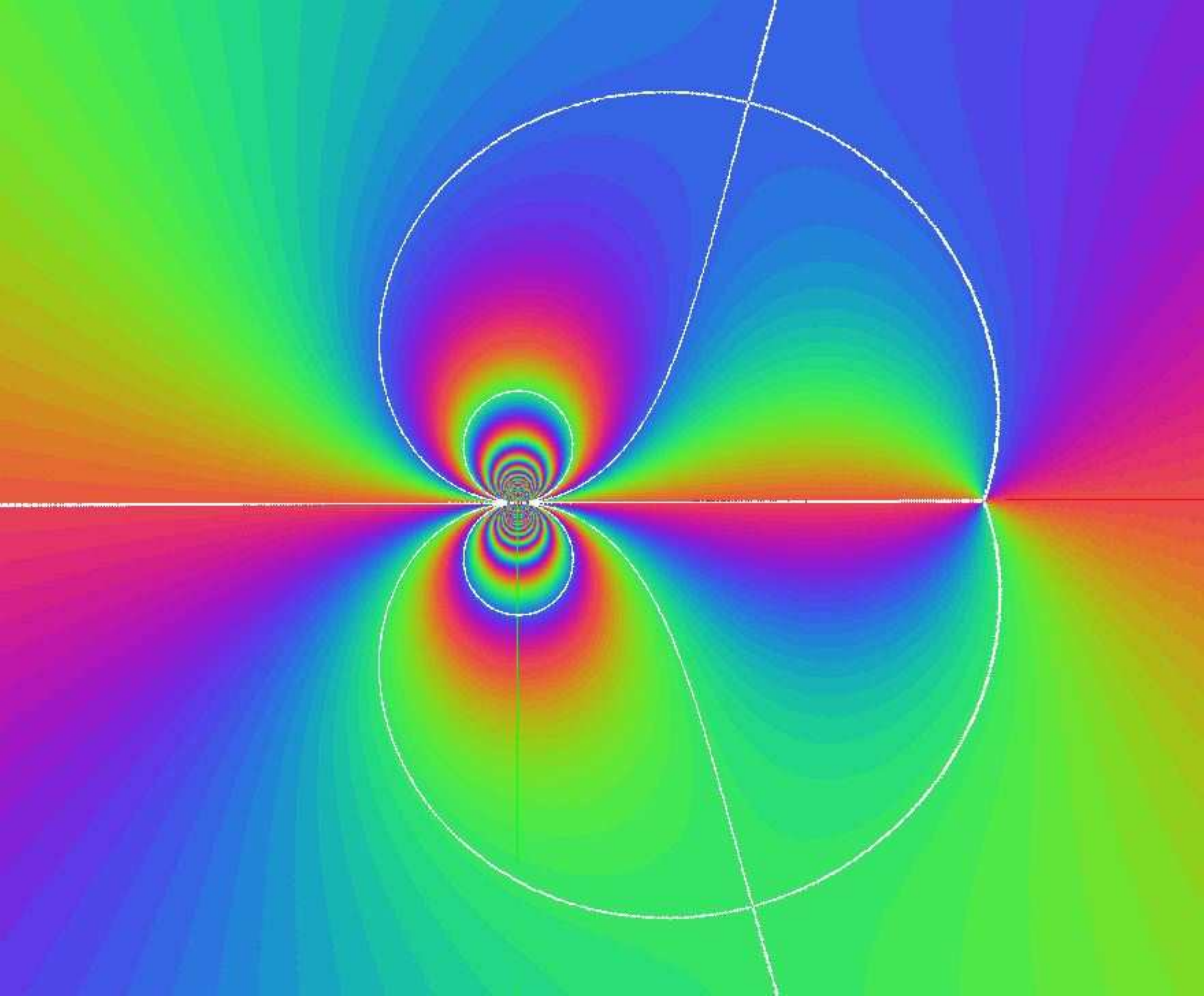}
\includegraphics[height=0.43\textwidth]{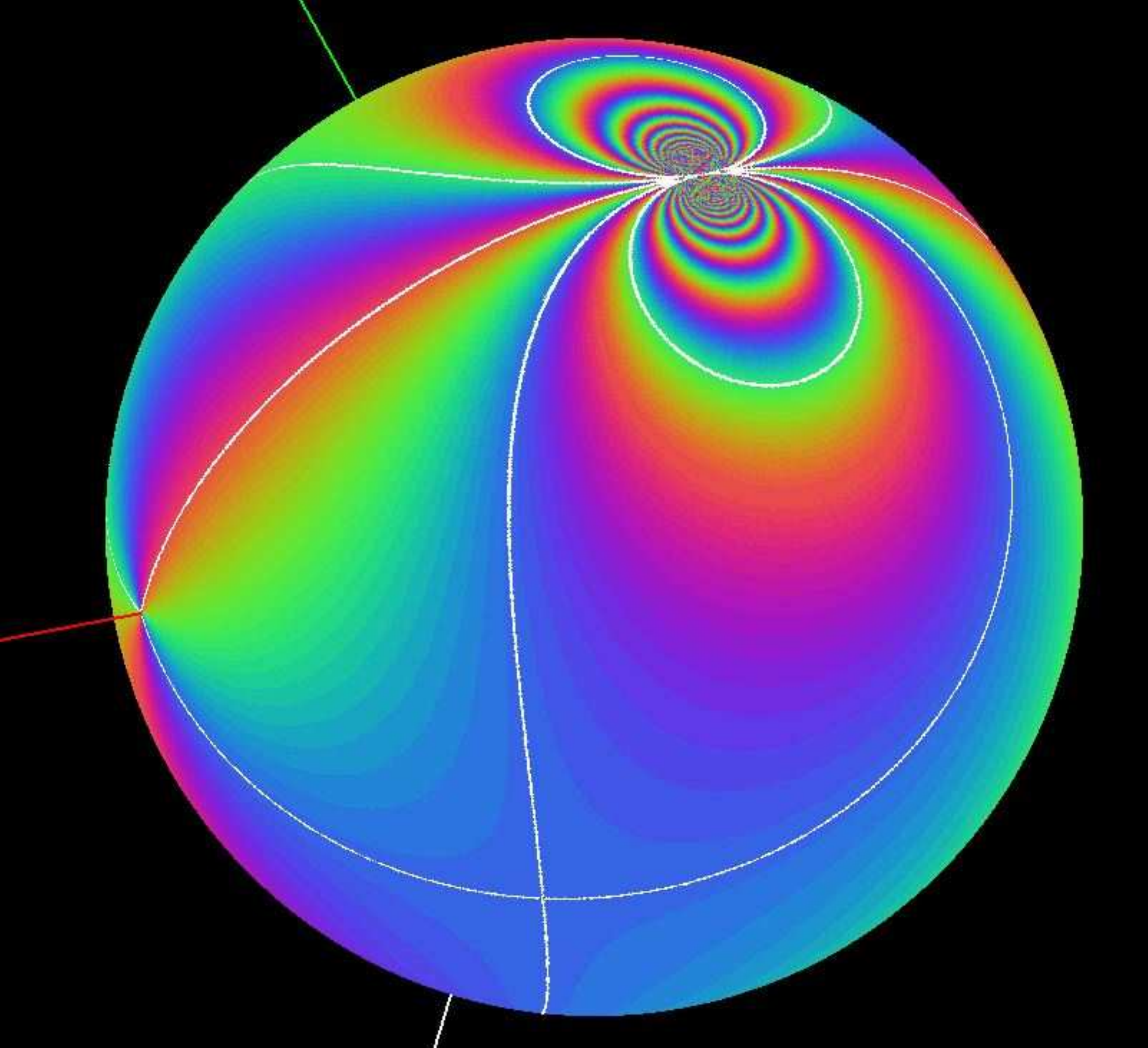}
\caption{Visualization of the field $X(z)=-\frac{z^{2}(z-1)}{z^{2}-z+1}\del{}{z}$.
The borders of the strip flows correspond to streamlines of the field.
We have also plotted the separatrices associated to the poles of the field.}
\label{campoPolos}
\end{figure*}

\subsection{Singular complex analytic vector fields with an isolated essential singularity at $\infty\in\CW$}
As examples of analytic vector fields with an isolated essential singularity at $\infty\in\CW$ we present the cases 
of $X(z)=\e^{z}\del{}{z}$ in Figure \ref{campoExp}, and of $X(z)=\frac{\e^{z^3}}{3z^3-1}\del{}{z}$ in 
Figure \ref{CampoExpZ3}. They correspond to Examples \ref{exampleExp} and \ref{exampleExp3} 
so the covering maps are $\Phi_{X}(z)=\e^{\e^{-z}}$ and $\Phi_{X}(z)=\e^{-z\e^{-z^{3}}}$ respectively.

Worth noticing is that visualization of the phase portrait of a vector field near an essential singularity is rather 
difficult with the usual methods. This is due mainly to the fact that the algorithms, involving numerical 
integration, propagate errors along the trajectories, and Picard's theorem tells us that near an essential 
singularity the vector field takes on all but at most two values in $\CW$; hence numerical integration breaks 
down rather quickly near the essential singularity. Even though numerical errors are also present in our 
visualization scheme (as can be seen in particular in the case of Figure \ref{CampoExpZ3}), these do not 
propagate along the trajectories, and are due solely to the numerical accuracy of the routines used to evaluate 
the auxiliary function $\rho(z)$. A deeper exploration of these errors is presented in \S\ref{comparison}.

\subsubsection{The case of $X(z)=\e^{z}\del{}{z}$}
In Figure \ref{campoExp} (a), we show the strip flows on the Riemann sphere of the vector field 
$X(z)=\e^{z}\del{}{z}$ in a vicinity of the essential singularity at $\infty$. Notice that the strips cluster together 
and it is difficult to appreciate the behavior of the vector field. On the other hand, by plotting specific trajectories 
we can examine the behavior of the flow near the essential singularity, and since the error does not propagate 
along the trajectory, we can visualize the actual trajectories (see Figure \ref{campoExp} (b)).
\begin{figure*}[htbp]
\centering
\includegraphics[height=0.45\textwidth]{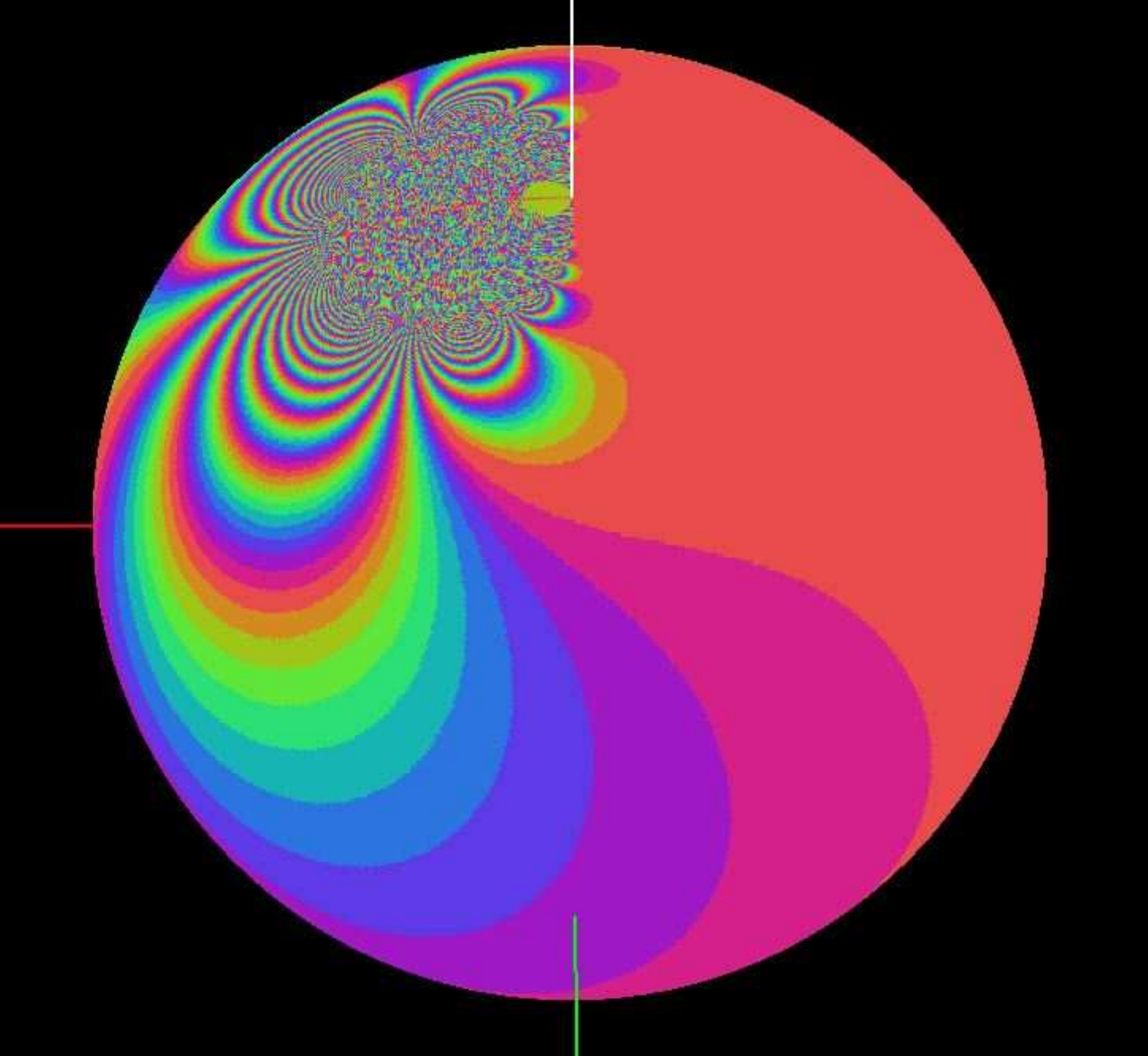}
\includegraphics[height=0.45\textwidth]{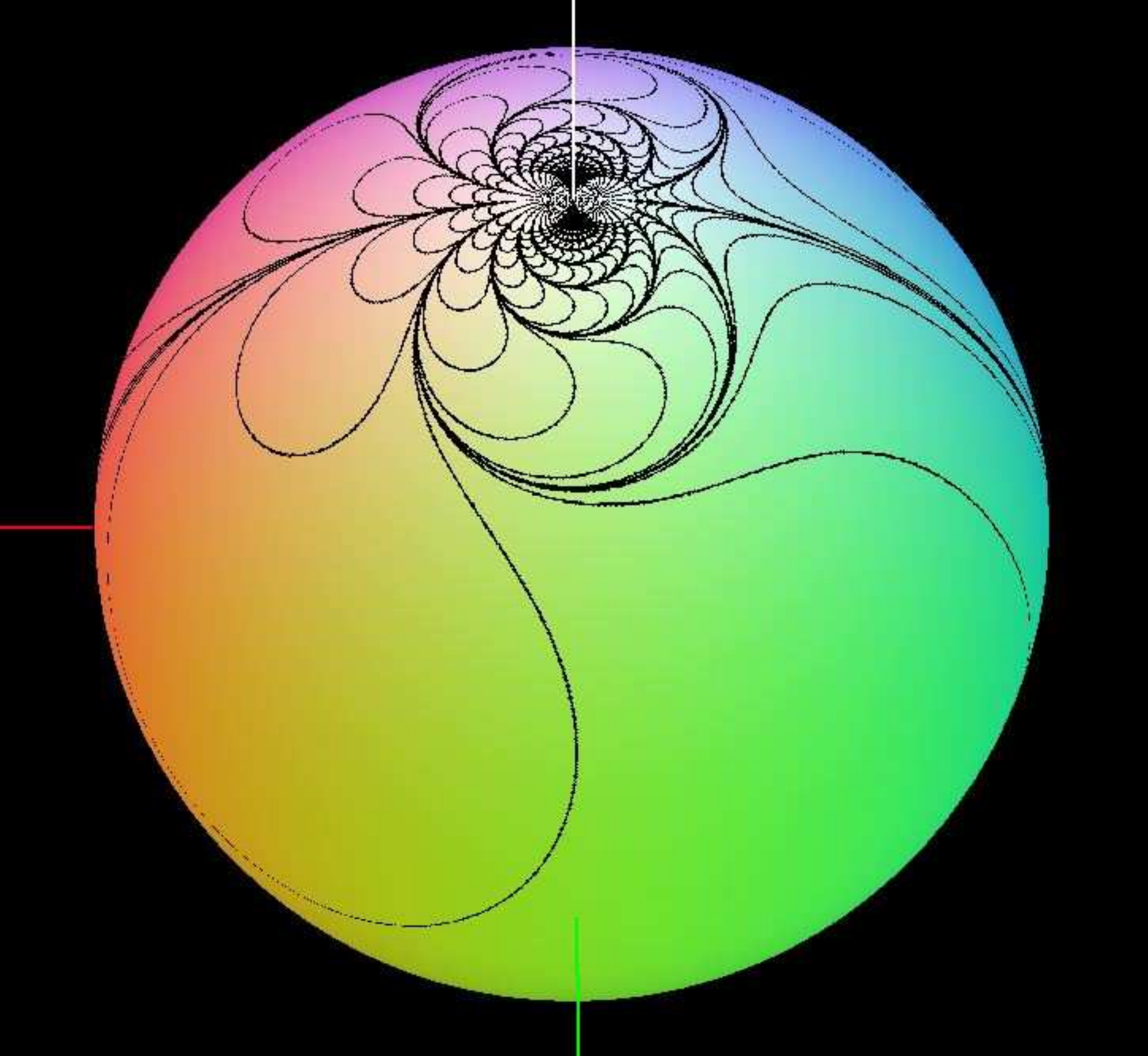}
\caption{Visualization of $X(z)=\e^{z}\del{}{z}$ near the essential singularity at $\infty$ on the Riemann sphere. 
In (a) we have plotted the strip flows. In (b) we have plotted some of the trajectories.}
\label{campoExp}
\end{figure*}

Once again, one can use the strip flows to gather information regarding the parametrization of the flow. For 
instance, one can observe that even though the trajectories appear symmetrical in Figure \ref{campoExp} (b), 
the strip flows in Figure \ref{campoExp} (a) indicate that the trajectories approach $\infty$ from the right 
much faster: in fact in finite time\footnote{
As observed in \cite{AP-MR} p.\,198, the singular analytic vector field $X(z)=\e^{z}\del{}{z}$ 
has two asymptotic values associated to the essential singularity at $\infty\in\CW_z$; 
0 and $\infty$, each with its own exponential tract. 
The trajectories that approach the essential singularity inside the exponential tract associated to the finite 
asymptotic value arrive in finite time (these are the trajectories on the right in Figure \ref{campoExp} (b)), 
while the trajectories that approach the essential singularity inside the exponential tract associated to the 
asymptotic value $\infty$ (trajectories on the left in Figure \ref{campoExp} (b)) 
take infinite time to reach the essential singularity.
}. 
This is a clear advantage over the results reported in \cite{NewtonLofaro} where 
only the trajectories are observed.

\subsubsection{The case of $X(z)=\frac{\e^{z^3}}{3z^3-1}\del{}{z}$}
In Figure \ref{CampoExpZ3} (a) we show the phase portrait of the vector field 
$X(z)=\frac{\e^{z^3}}{3z^3-1}\del{}{z}$ in the plane $\CC$ where we can observe the three first 
order poles at $\frac{1}{\sqrt[3]{3}}$, $\frac{\e^{i 2\pi/3}}{\sqrt[3]{3}}$ and $\frac{\e^{-i2\pi/3}}{\sqrt[3]{3}}$, 
while in Figure \ref{CampoExpZ3} (b) the phase portrait of the same vector field is shown on the sphere 
in a vicinity of the essential singularity at $\infty\in\CW$. 

As mentioned before, some numerical errors are still present (solid color regions near the essential singularity) 
due to the nature of the essential singularity of $f(z)$, but there seems to be a strong suggestion of some pattern 
characteristic to the essential singularities. This last remark is explored further in \cite{AP-MR} and 
\cite{AP-MR-3} from a theoretical viewpoint.
\begin{figure*}[htbp]
\centering
\includegraphics[height=0.37\textwidth]{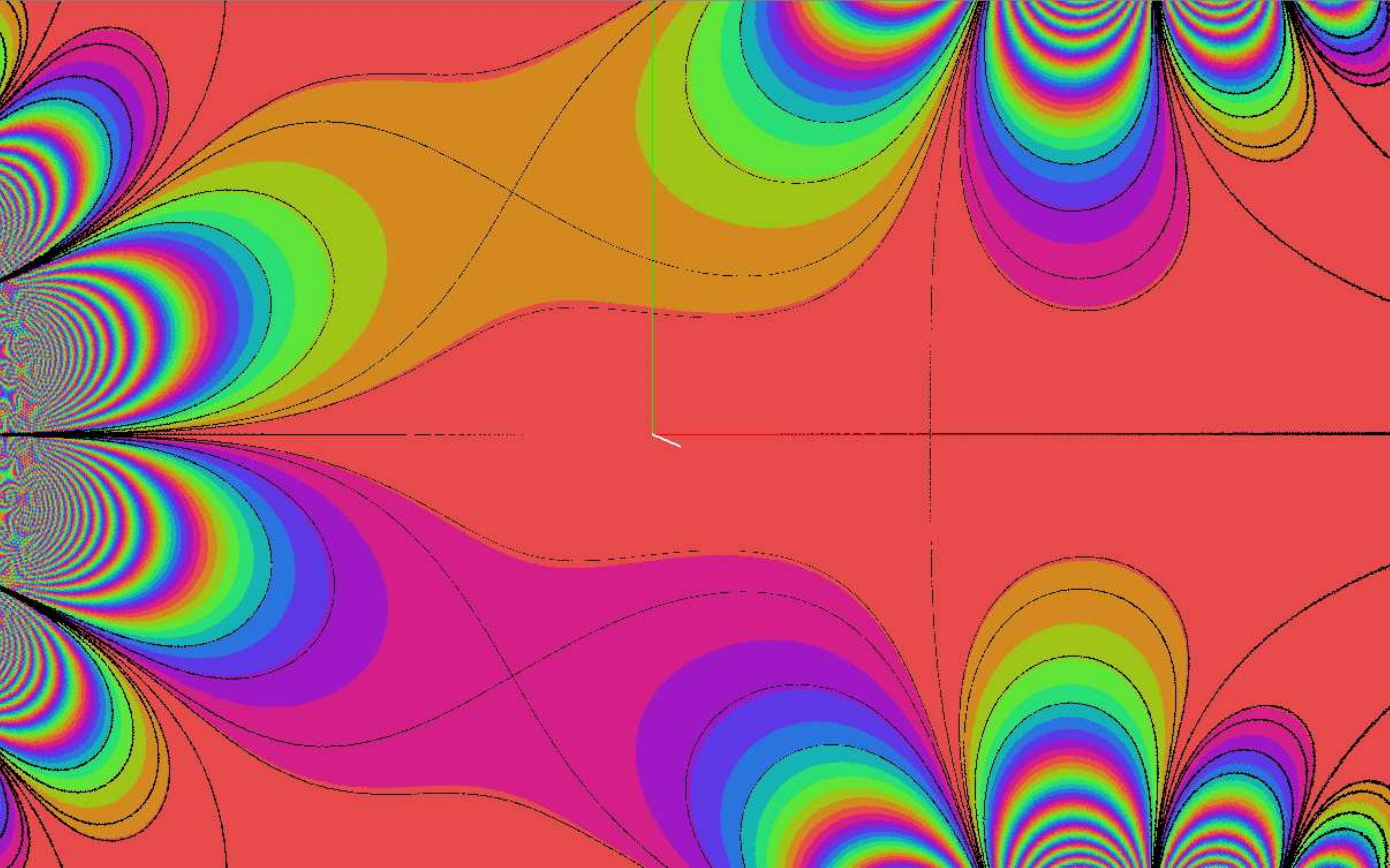}
\includegraphics[height=0.37\textwidth]{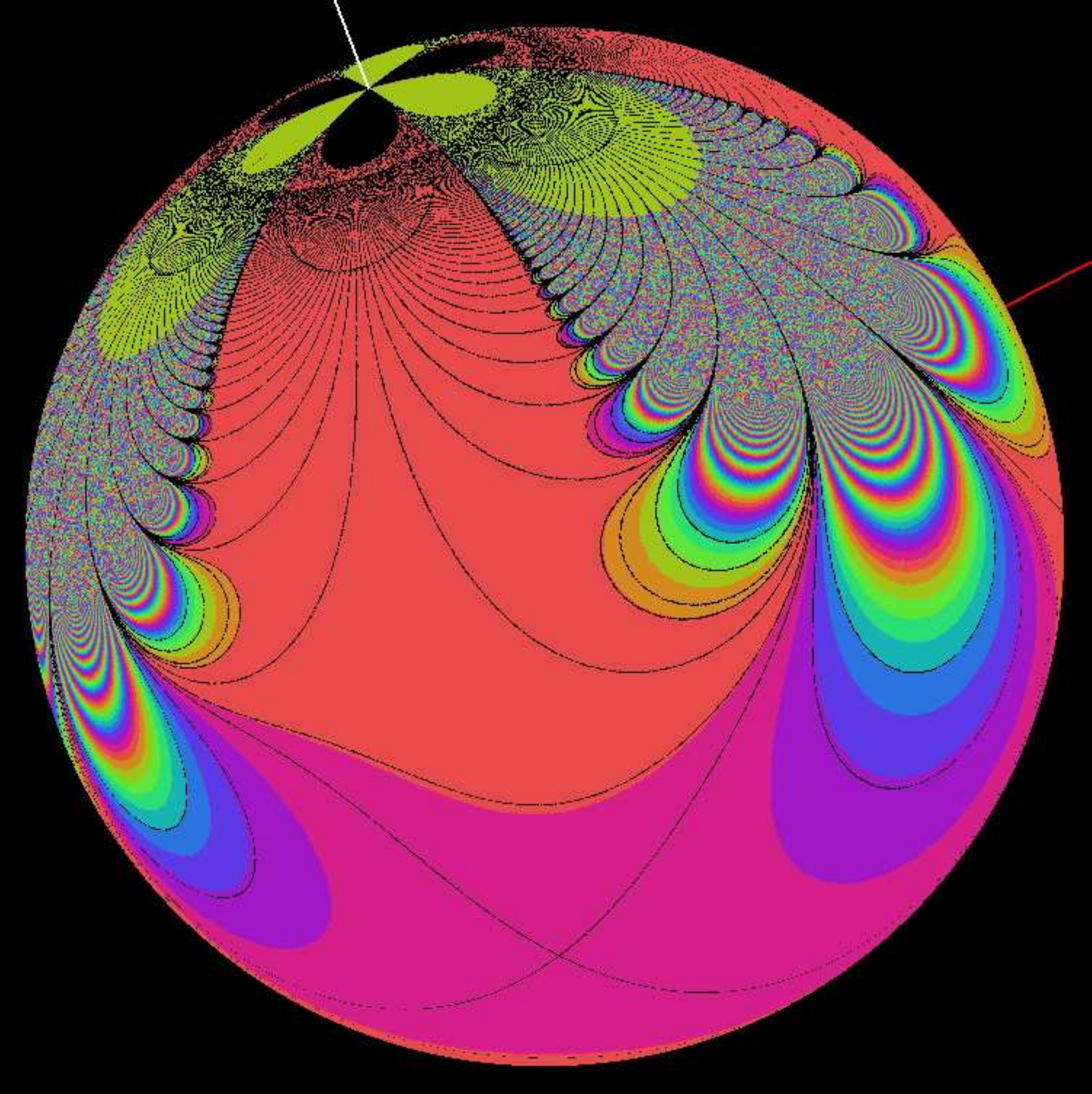}
\caption{Visualization of $X(z)=\frac{\e^{z^3}}{3z^3-1}\del{}{z}
\in \E (0, 3,3)$. 
(a) Shows the field in the plane in a vicinity of the origin where one can observe the three 
simple poles at $\frac{1}{\sqrt[3]{3}}$, $\frac{\e^{i 2\pi/3}}{\sqrt[3]{3}}$ and $\frac{\e^{-i2\pi/3}}{\sqrt[3]{3}}$. 
(b) Shows the field on the Riemann sphere, one can observe the essential singularity at $\infty$ 
and a simple pole at $\frac{\e^{-i2\pi/3}}{\sqrt[3]{3}}$.}
\label{CampoExpZ3}
\end{figure*}

In particular, note that $X\in\E(0, 3,3)$, hence
$\Psi_{X}$ is single valued, 
$\R_{X}$ is an infinitely ramified Riemann surface 
over $\CW_{t}$ and $\pi_{X,1}$ 
provides a global flow of $X$. 
For further details see \cite{AP-MR-3}, 
where the combinatorial concept of 
{\it $(r,d)$--configuration trees} allows for an accurate 
description of the Riemann surfaces $\R_{X}$.
A very rough drawing of a  generic $\R_X$ is provided in 
Figure  \ref{tres-y-tres}.

\begin{figure*}[htbp]
\centering
\includegraphics[height=0.6\textwidth]{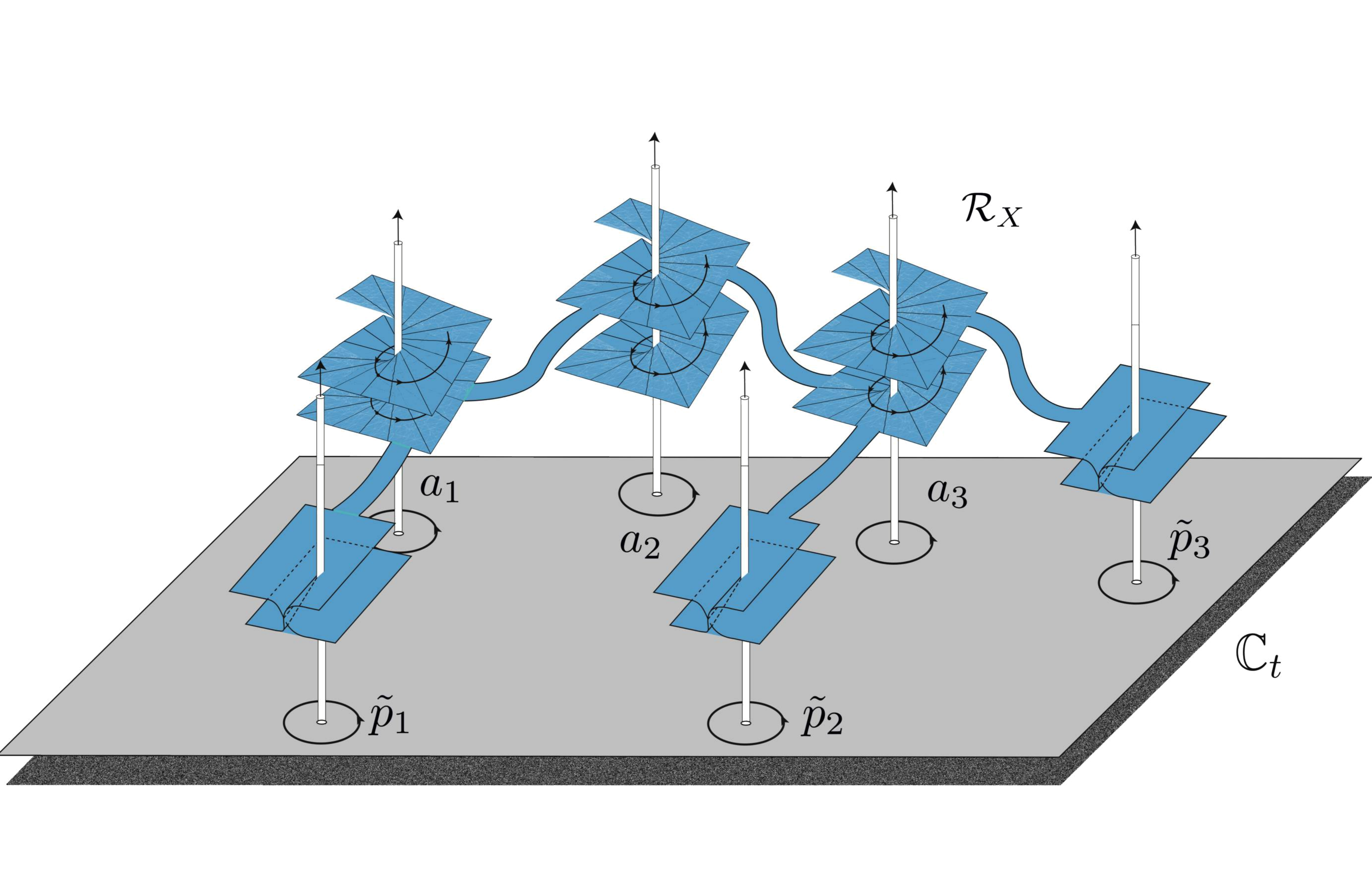}
\caption{A sketch of a Riemann surface $\R_X$ for 
$X(z)=\frac{\lambda}{(z-p_1) (z-p_2) (z-p_3)}
\e^{b_0z^3 + \ldots + b_3  }\del{}{z}
\in \E(0,3,3)$ according to Diagram \eqref{diagramaRX}.
The horizontal plane is $\CC_t$,
generically there are three finite 
asymptotic values $\{a_1, a_2, a_3\}$ and three finite 
critical values 
$\{ \widetilde{p}_1=\Psi_X(p_1), 
\widetilde{p}_2=\Psi_X(p_2),  \widetilde{p}_3=\Psi_X(p_3) \}$.
}
\label{tres-y-tres}
\end{figure*}

\subsection{Singular complex analytic vector fields with an accumulation point at $\infty\in\CW$}
The third case, that in which $\infty\in\CW$ is an accumulation point\footnote{
Note that $\infty\in\CW$ is not an \emph{isolated} essential singularity for $X$, however it is a 
\emph{non--isolated} essential singularity, both in the sense of its Laurent series expansion and 
in the sense that the conclusion of Picard's theorem is still satisfied.
} of zeros or poles of $f(z)$ is also of interest. Here we present two examples:

\subsubsection{The case of $X(z)=-\tan(z)\del{}{z}$} 
As was seen in Example \ref{tanvectorfield}, 
the ramified covering characterizing this vector field is $\Phi(z)=\sin(z)$. In Figure \ref{campoTanz} 
we provide the visualization of the phase portrait of $X$. Clearly we observe a sequence of 
alternating simple poles and zeros along the real axis accumulating to $\infty\in\CW$.

\noindent
By plotting the separatrices associated to the poles we can immediately observe their position. 
It should be noted that in this case the separatrices are also horizontal asymptotic paths associated 
to the (non--isolated) essential singularity at $\infty\in\CW$.
\begin{figure*}[htbp]
\centering
\includegraphics[height=0.4\textwidth]{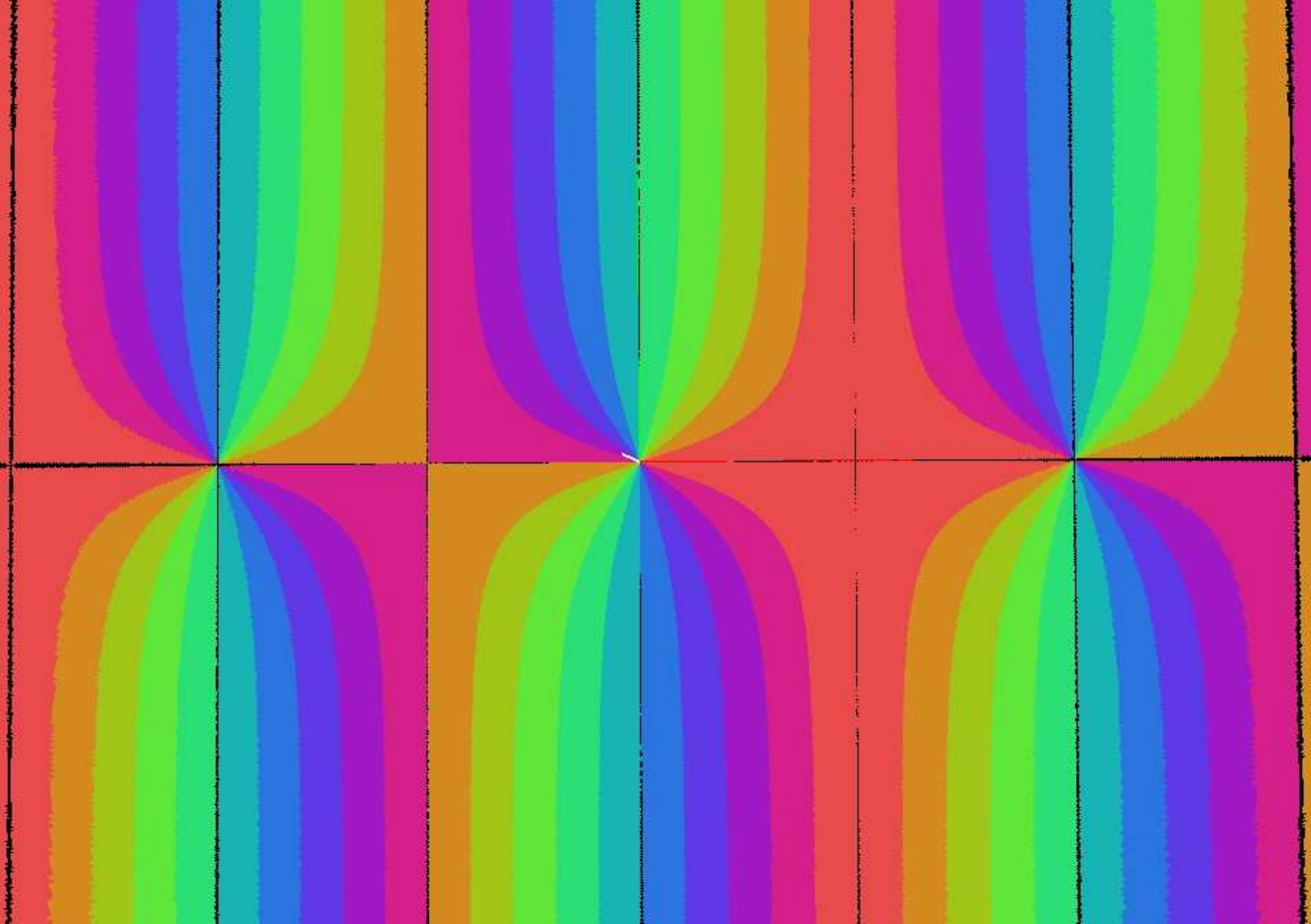}
\includegraphics[height=0.4\textwidth]{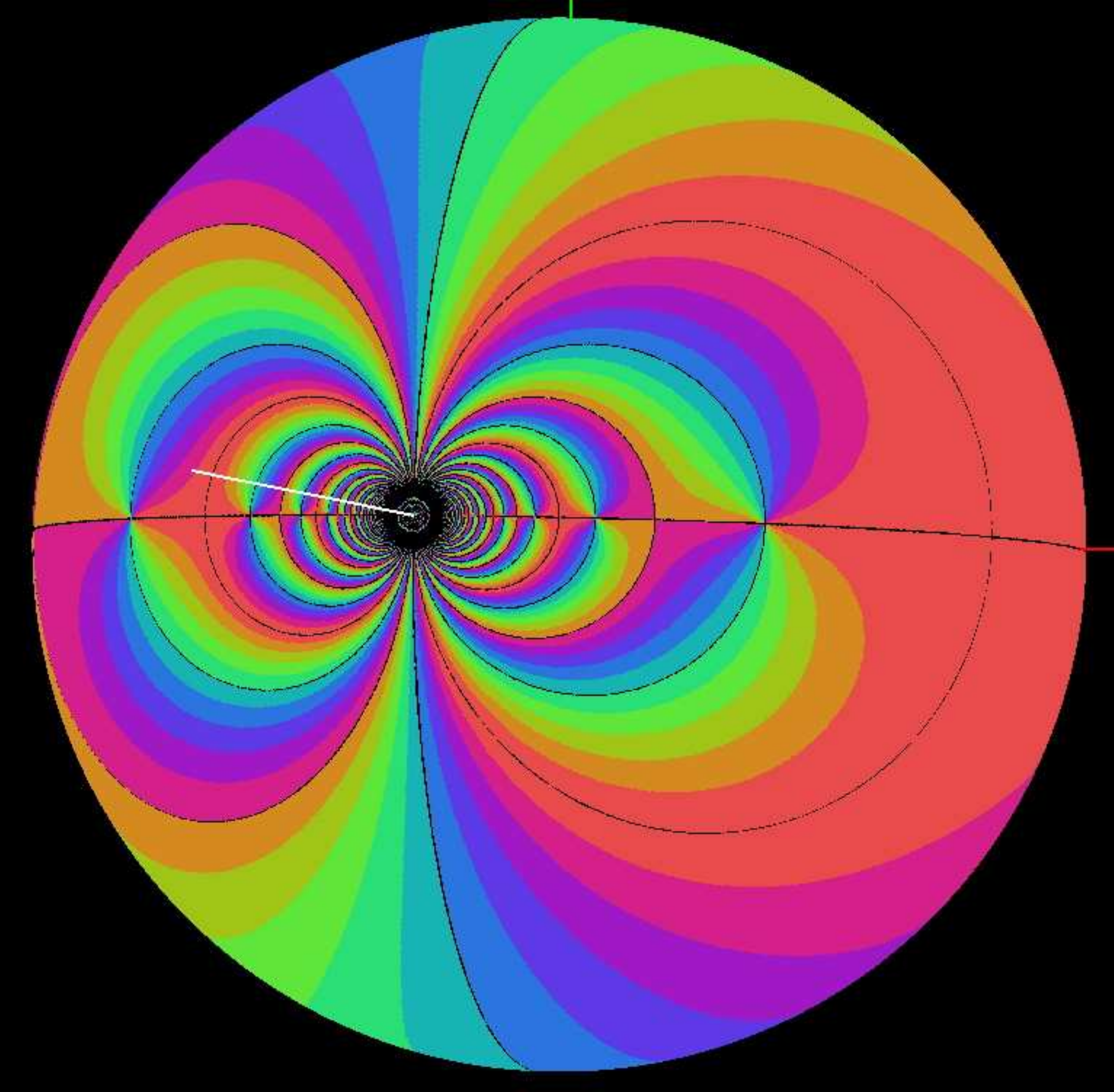}
\caption{The field $X(z)=-\tan(z)\del{}{z}$ visualized in the complex plane and on the Riemann Sphere, 
using the techniques described in the text. One can observe a sequence of alternating simple poles and 
zeros along the real axis accumulating to $\infty\in\CW$.}
\label{campoTanz}
\end{figure*}

\subsubsection{The case of $X(z)=-(\cosh(z)+1)\del{}{z}$} 
The ramified covering that characterizes this vector field is given by 
$\Phi(z)=\e^{\frac{\e^{z}-1}{\e^{z}+1}}$, as was remarked previously in 
Example \ref{examplecoshPlus2}. The visualization of the phase portrait of $X$ 
is provided in Figure \ref{coshPlus2}. 
We observe a sequence of order two zeros along the imaginary axis accumulating to $\infty\in\CW$.

\noindent
Notice that $X$ does not have any poles, hence there are no separatrices. 
However by plotting some specific level curves we can observe some of the horizontal asymptotic 
paths associated to the the (non--isolated) essential singularity at $\infty\in\CW$. 
Thus the behaviour of the flow of $X$ on neighborhoods of $\infty\in\CW$ is better understood.
\begin{figure*}[htbp]
\centering
\includegraphics[height=0.45\textwidth]{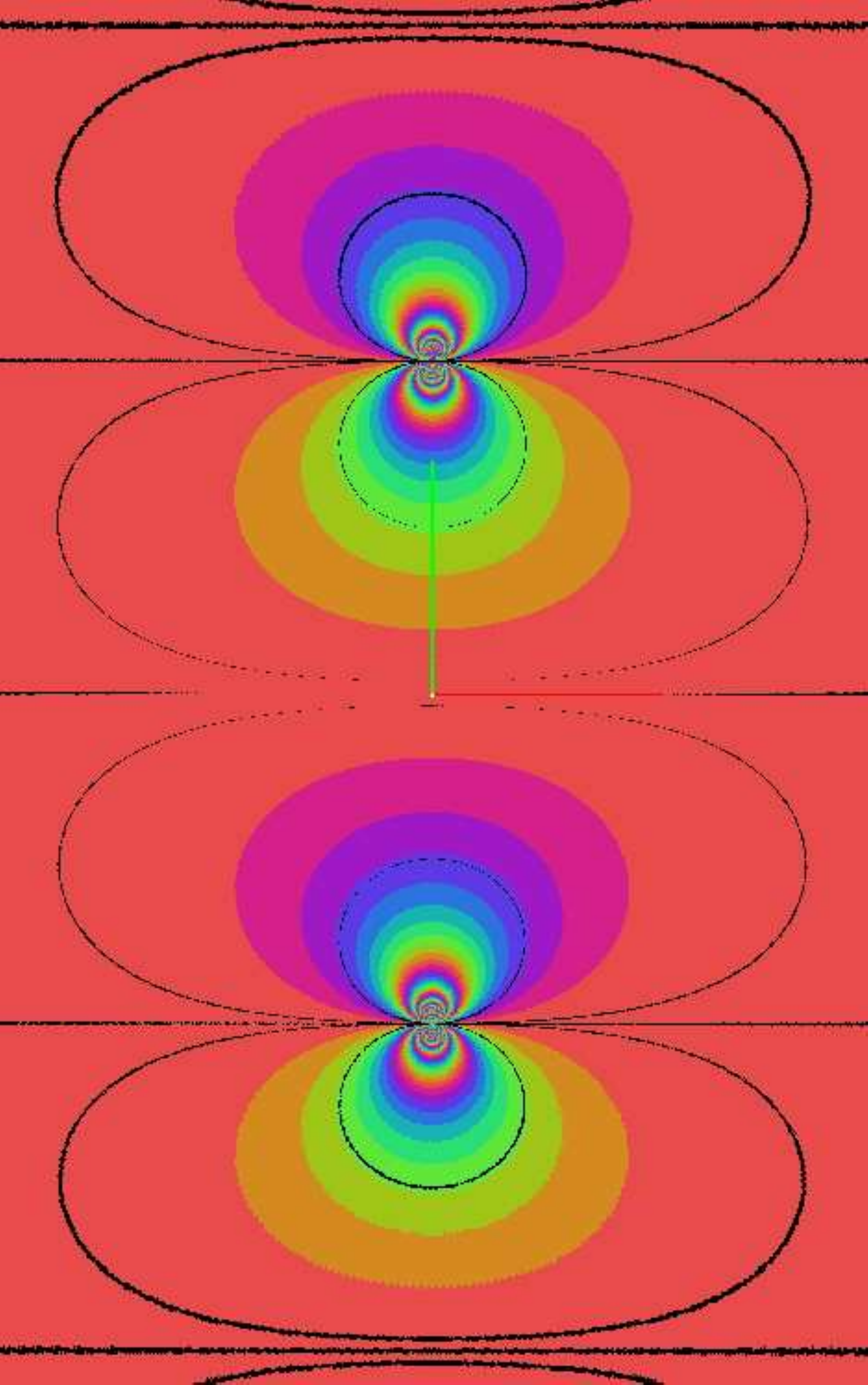}
\includegraphics[height=0.45\textwidth]{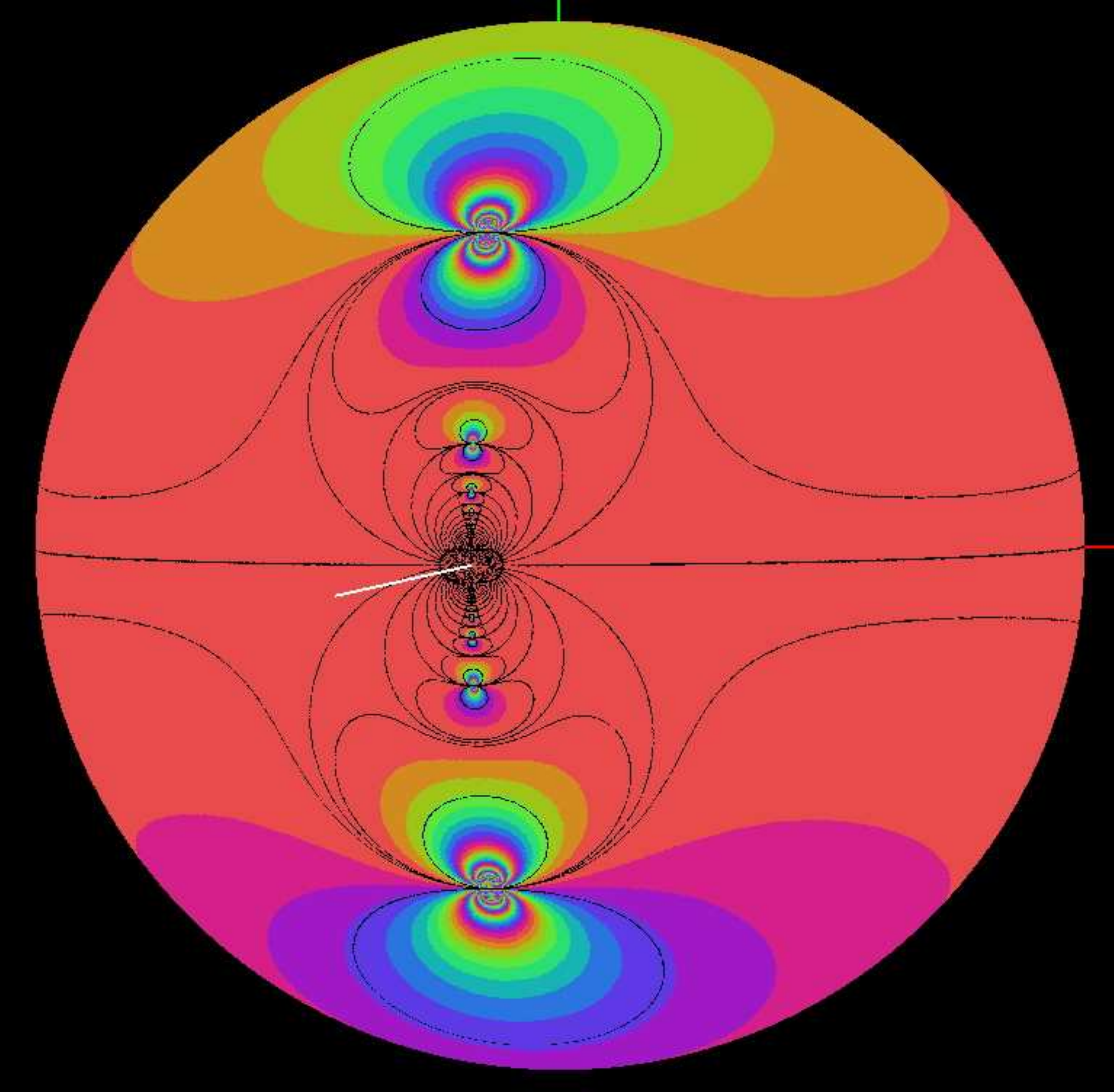}
\caption{The field $X(z)=-(\cosh(z)+1)\del{}{z}$ visualized in the complex plane and on the 
Riemann Sphere, using the techniques described in the text. In this case one observes a 
sequence of zeros of order 2 along the imaginary axis accumulating to $\infty\in\CW$.}
\label{coshPlus2}
\end{figure*}
 
\section{Comparison with usual integration--based algorithms}\label{comparison}
In this section we compare the proposed method with two of the most widely used \emph{integration--based} 
algorithms: namely with the $4^{th}$ order Runge--Kutta (RK4) and the Runge--Kutta--Fehlberg (RKF) 
algorithms. We do not consider Euler's method because it uses a first order approach and its results are 
expected to be worse than those obtained by the Runge--Kutta algorithms.

As seen in \S\ref{normalforms}, given a singular complex analytic vector field $X(z) = f(z)\del{}{z}$ 
the generic behaviour of the flow is different in the neighborhood of:
\begin{enumerate}[label=(\arabic*)]
	\item \emph{non--singular} points of $X(z)$,
	\item \emph{singular} points of $X(z)$, which are further subdivided as
	\begin{enumerate}[label=(\alph*)]
		\item \emph{zeros},
		\item \emph{poles},
		\item \emph{essential singularities} and 
		\item \emph{accumulation points} of the above types.
	\end{enumerate}
\end{enumerate}
We compare the behaviour in cases (1), (2a), (2b) and (2c) above using
\begin{enumerate}[label=(\Alph*)]
	\item two \emph{integration--based} algorithms:
	\begin{enumerate}[label=(\roman*)]
		\item \emph{$4^{th}$ order Runge--Kutta} algorithm (RK4),
		\item \emph{Runge--Kutta--Fehlberg} algorithm (RKF), and
	\end{enumerate}
	\item the \emph{Newton} method proposed in this note.
\end{enumerate}
The \emph{integration--based} algorithms, RK4 and RKF, are usually used to solve first order ODE systems. 
The RK4 is a constant step-size method whose implementation is very simple and well known, however one 
does not have control over the error incurred. The Runge--Kutta--Fehlberg is an adjustable step-size method 
which allows some control on the error\footnote{In the RKF algorithm the error is controlled by decreasing the 
step--size of the recursive algorithm, hence increasing the computational requirements.}. This method is a 
combination of the Runge--Kutta of order four and five, hence is also known as RKF45. For further information 
and explicit implementations of these algorithms consult \cite{Borden}.

Since the \emph{Newton} method proposed in this note provides us with \emph{exact solutions\footnote{
Recall that the solution is exact up to the numerical error incurred in the evaluation of the constants of 
motion $\rho$ and $\theta$.}} 
to the problem of finding trajectories (including parametrization) of the flow of a given complex analytic 
vector field, then it is possible to calculate the (absolute) error involved while using \emph{integration--based} 
algorithms:

\noindent
Let $\tilde{z}_{\tau}=\tilde{z}(\tau)$ denote the trajectory that passes through $z_{0}$ at time $\tau=0$
obtained using an \emph{integration--based} algorithm, and let $z_{\tau}=z(\tau)$ denote the exact solution 
obtained with the \emph{Newton} method. Then the absolute error incurred by the \emph{integration--based} 
algorithm is given by
\begin{equation}\label{error}
\text{AbsErr}(\tau)=\abs{z_{\tau}-\tilde{z}_{\tau}}.
\end{equation}
So, in order to calculate the error at time $t$ one has to:
\begin{enumerate}[label=(\arabic*)]
	\item calculate $\tilde{z}_{\tau}$ using the \emph{integration--based} algorithm,
	\item calculate $z_{\tau}$ as the intersection of $\rho(z)=\rho(z_{0})$ with $\theta(z)=\theta(z_{0})-\tau$,
	\item calculate the error using \eqref{error}.
\end{enumerate}

On the other hand, the error of the exact solution ({\it i.e.} the one obtained with the \emph{Newton} method) 
can be estimated by an indirect method as the  
\emph{relative deviation of the exact solution} given by
$$\text{RelDev}_{h}(\tau)=\frac{\abs{\rho(z_{\tau})-\rho(z_{0})}}{\abs{\rho(z_{0})}}.$$
This measurement can be interpreted as the unit--less distance of the calculated point $z_{\tau}$ from the actual 
trajectory.
Moreover using this indirect method, one can also measure the deviation that the \emph{integration--based} 
algorithms have from the exact solution by calculating the \emph{relative error of the integration--based solution} 
as:
$$\text{RelError}_{h}(\tau)=\frac{\abs{\rho(\tilde{z}_{\tau})-\rho(z_{0})}}{\abs{\rho(z_{0})}}.$$
These last two measurements can be used to compare side by side the solution obtained by the \emph{Newton} 
method and the \emph{integration--based} methods.

\subsection{Results of the comparison}
We compared the associated errors obtained by the usual integration methods {\it vs.} the exact 
solution obtained with the proposed methodology; also the CPU time used by the different 
approaches is reported.

In all cases we restricted our analysis to a rectangular region of the plane, since to visualize the results on the 
Riemann sphere stereographic projection is used independently of which visualization method is chosen, hence 
this restriction does not affect the comparison results. 

This was done in neighborhoods of: a \emph{regular} (R) value of the flow, a \emph{zero} (Z) of $f$, 
a \emph{pole} (P) of $f$, and an \emph{essential singularity} (E) of $f$. 
The different errors $\big(\text{AbsErr}(t)$, $\text{RelDev}_{h}(t)$ and $\text{RelError}_{h}(t)\big)$ were plotted 
as a function of $\tau$ in a vicinity of $\tau=0$ resulting in 
Figures \ref{R} (a), \ref{Z} (a), \ref{P} (a), \ref{E} (a) 
for the 
case of the \emph{$4^{th}$ order Runge--Kutta} algorithm; and in 
Figures \ref{R} (b), \ref{Z} (b), \ref{P} (b), \ref{E} (b) for 
the case of the \emph{Runge--Kutta--Fehlberg} algorithm.
\begin{figure*}[htbp]
\centering
\includegraphics[width=0.48\textwidth]{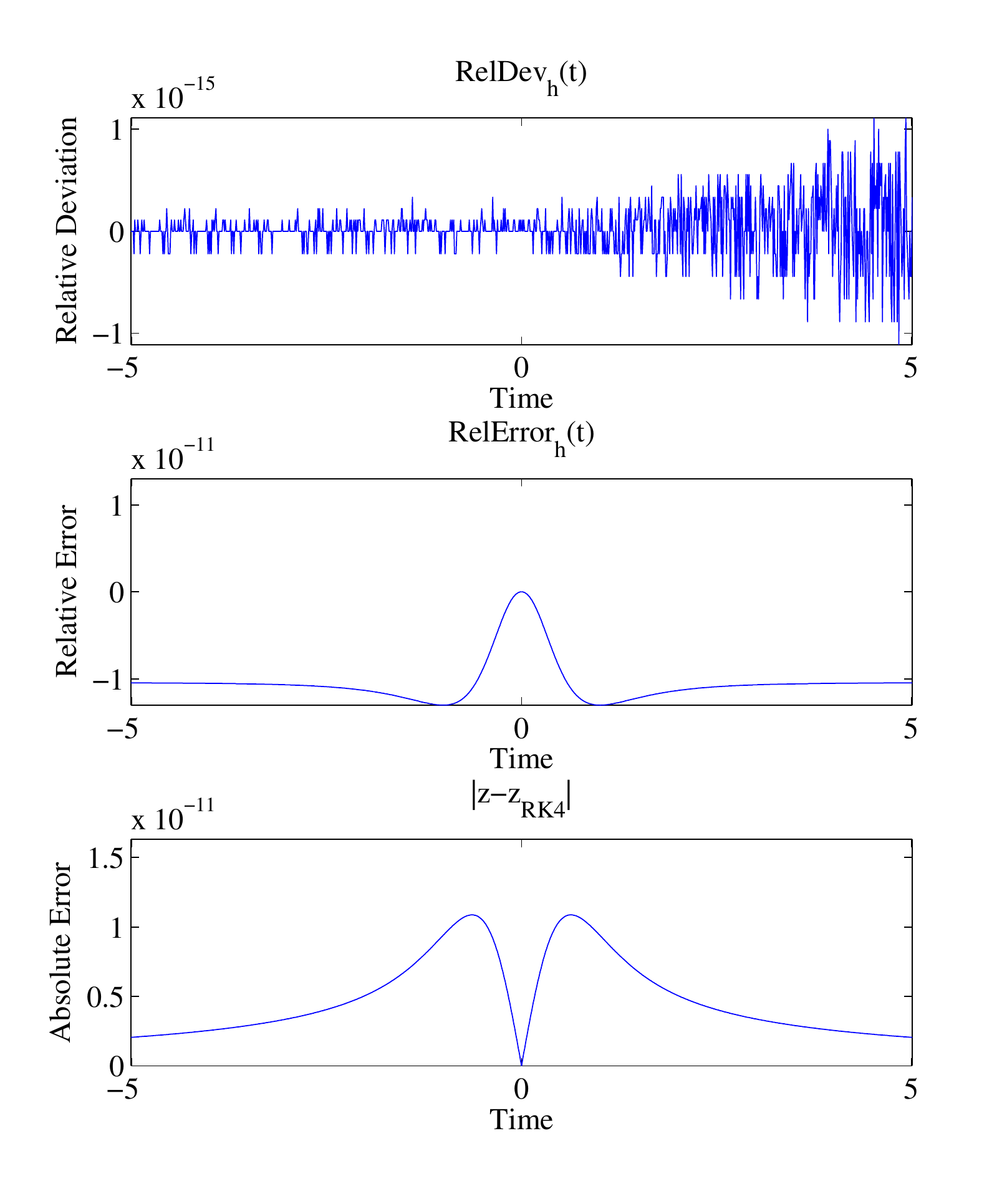}
\includegraphics[width=0.48\textwidth]{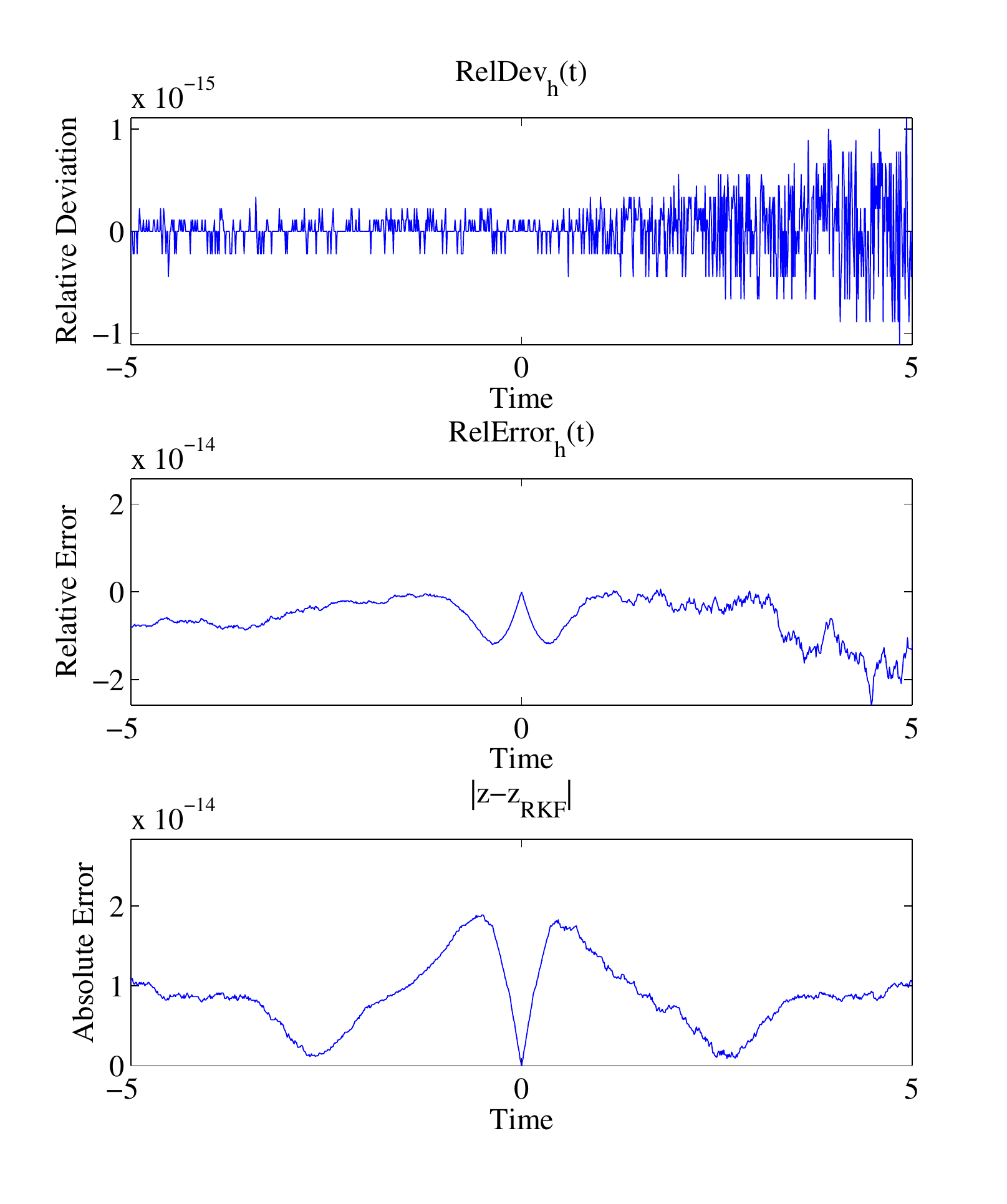}

(a) \hskip 190pt (b)
\caption{Comparison of the errors in a vicinity of a regular point obtained (a) when using the $4^{th}$ order 
Runge--Kutta algorithm, and (b) when using the Runge--Kutta--Fehlberg algorithm. 
The vector field $\exp(z)\del{}{z}$ was used with initial condition $z_{0}=i\frac{\pi}{2}$.}
\label{R}
\end{figure*}
\begin{figure*}[htbp]
\centering
\includegraphics[width=0.48\textwidth]{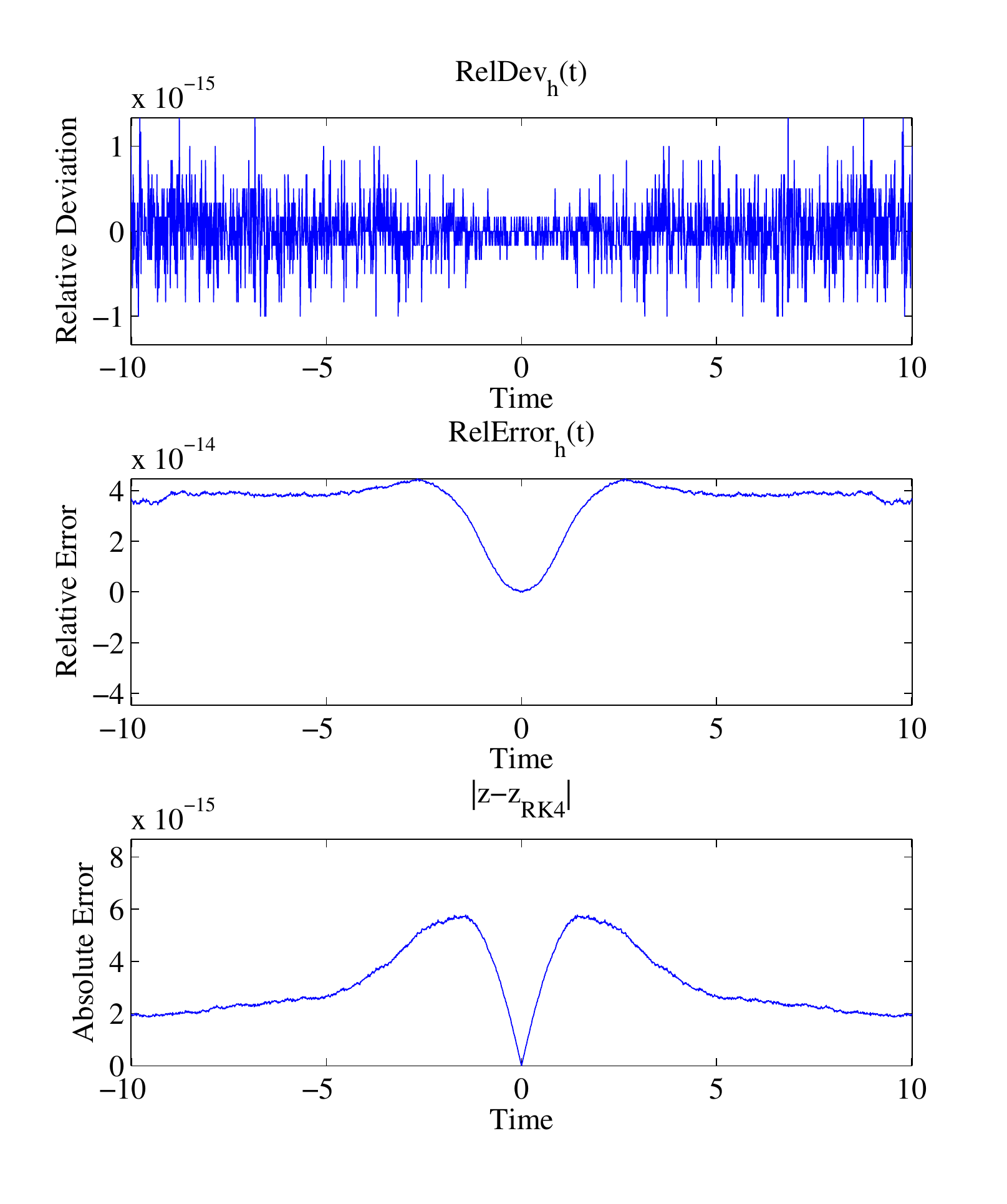}
\includegraphics[width=0.48\textwidth]{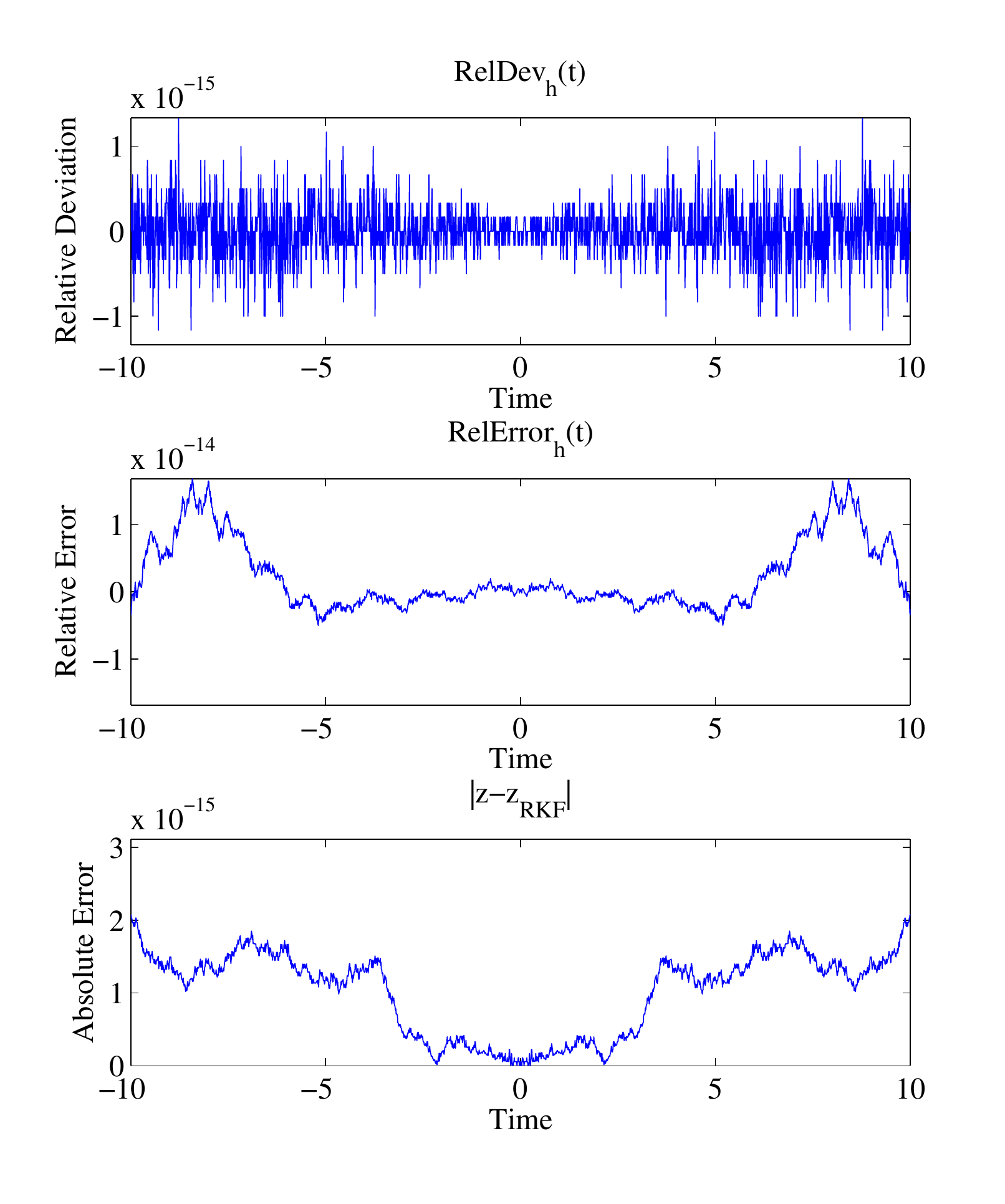}

(a) \hskip 190pt (b)
\caption{Comparison of the errors in a vicinity of a zero obtained (a) when using the $4^{th}$ order Runge--Kutta 
algorithm, and (b) when using the Runge--Kutta--Fehlberg algorithm. The vector field $z^{4}\del{}{z}$ was used 
with initial condition $z_{0}=\frac{i}{2}$.}
	\label{Z}
\end{figure*}
\begin{figure*}[htbp]
\centering
\includegraphics[width=0.48\textwidth]{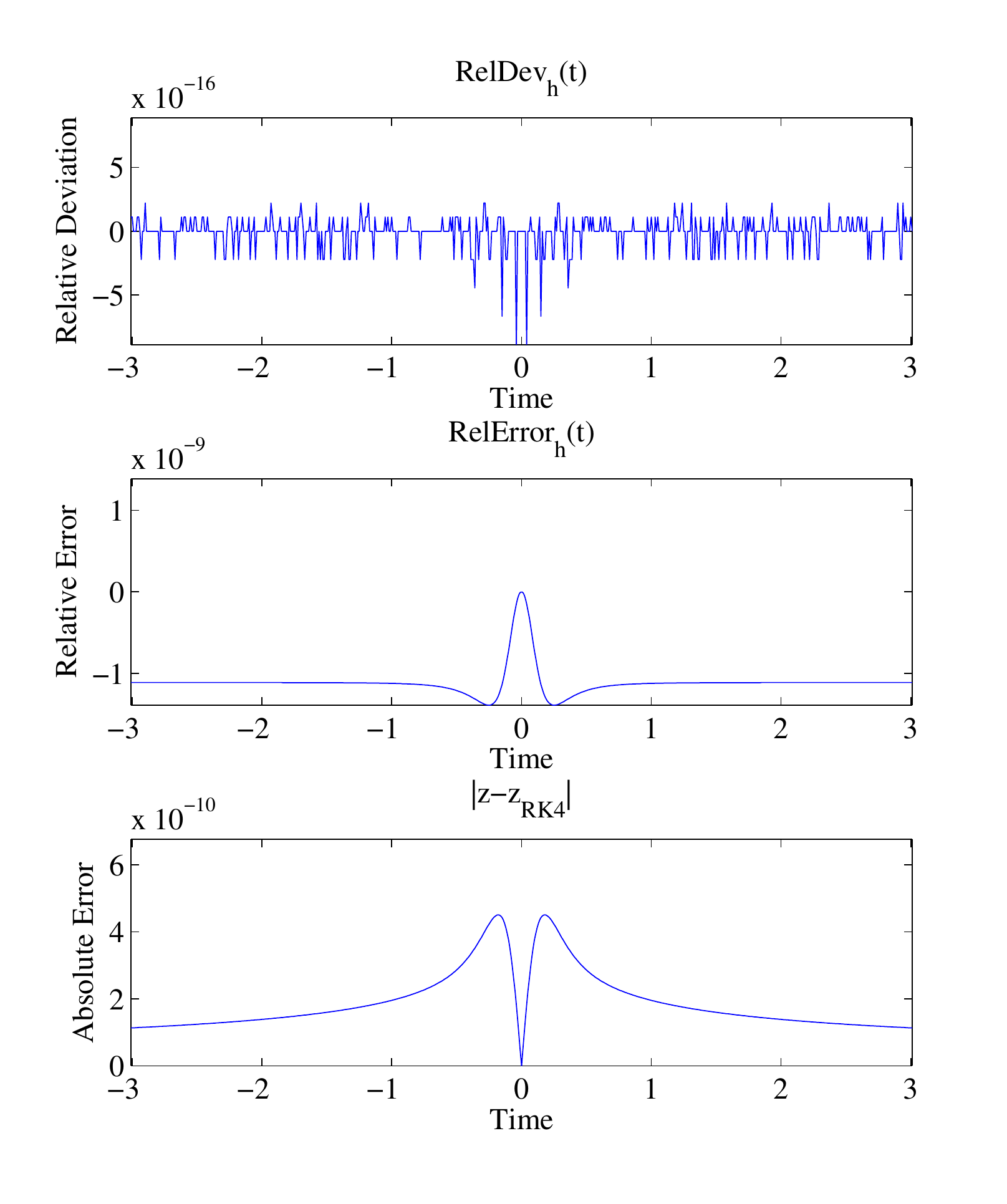}
\includegraphics[width=0.48\textwidth]{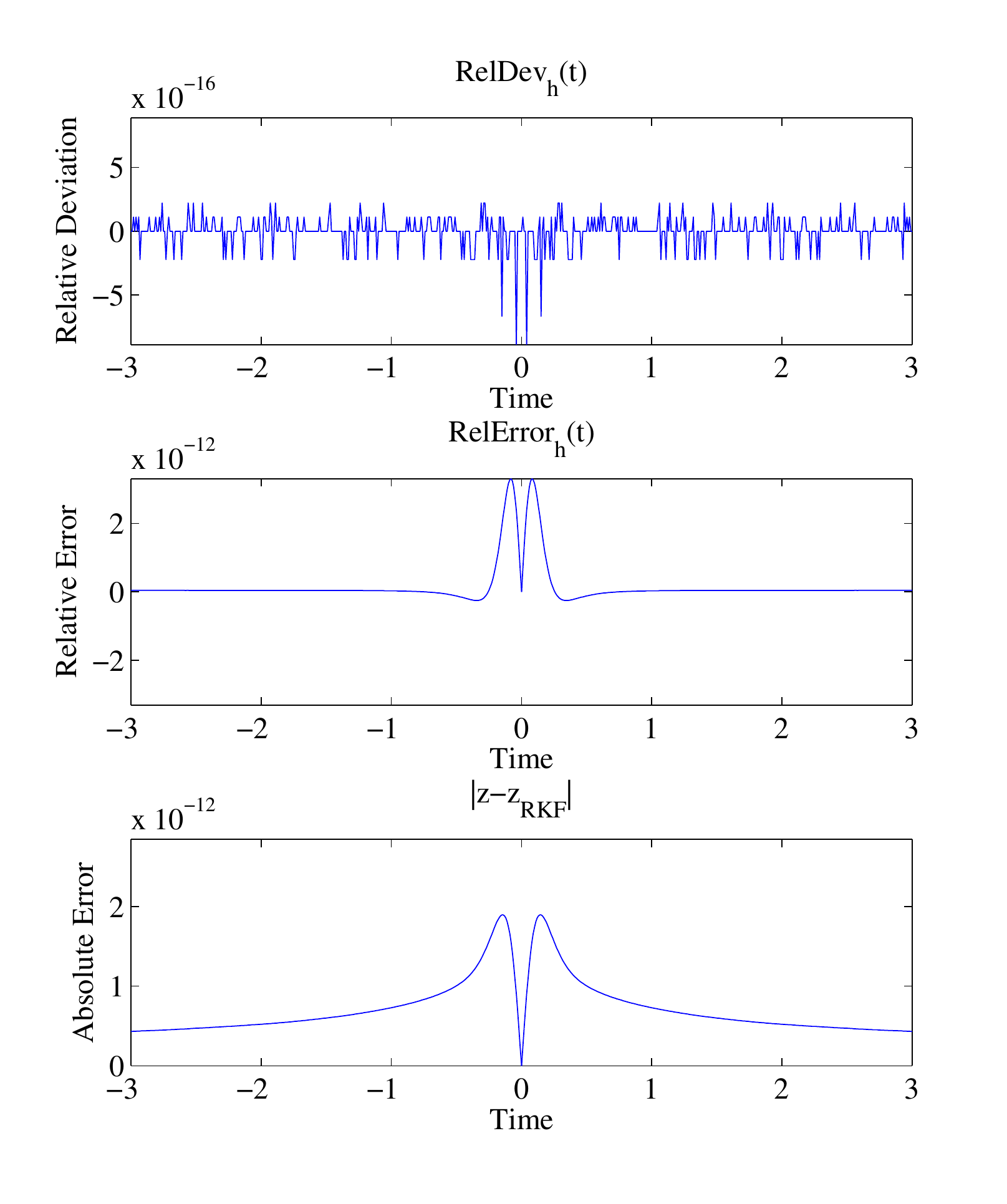}

(a) \hskip 190pt (b)
\caption{Comparison of the errors in a vicinity of a Pole obtained (a) when using the $4^{th}$ order Runge--Kutta 
algorithm, and (b) when using the Runge--Kutta--Fehlberg algorithm. The vector field $\frac{1}{z}\del{}{z}$ was 
used with initial condition $z_{0}=-\frac{1}{2}-i\frac{1}{2}$.}
	\label{P}
\end{figure*}
\begin{figure*}[htbp]
\centering
\includegraphics[width=0.48\textwidth]{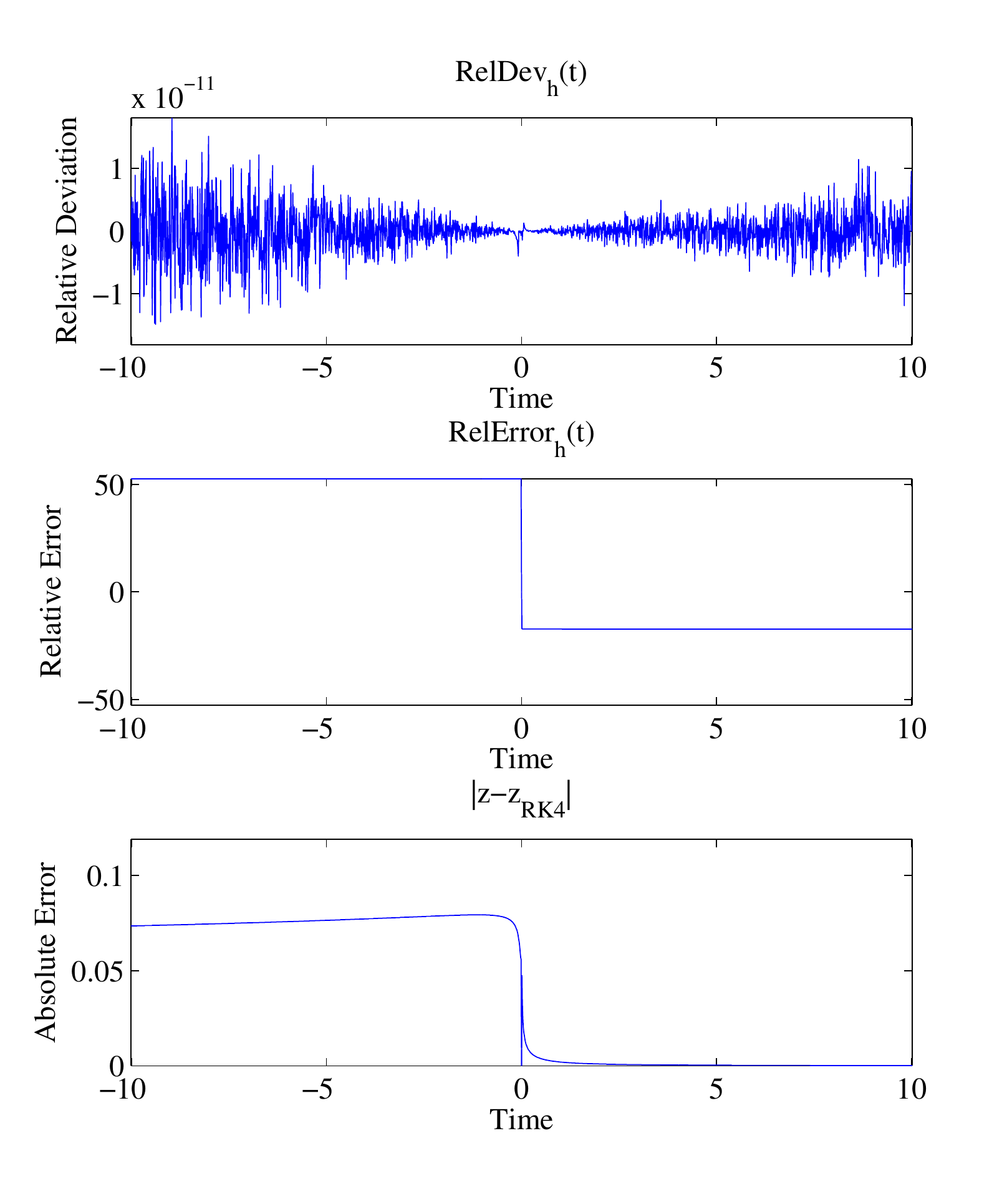}
\includegraphics[width=0.48\textwidth]{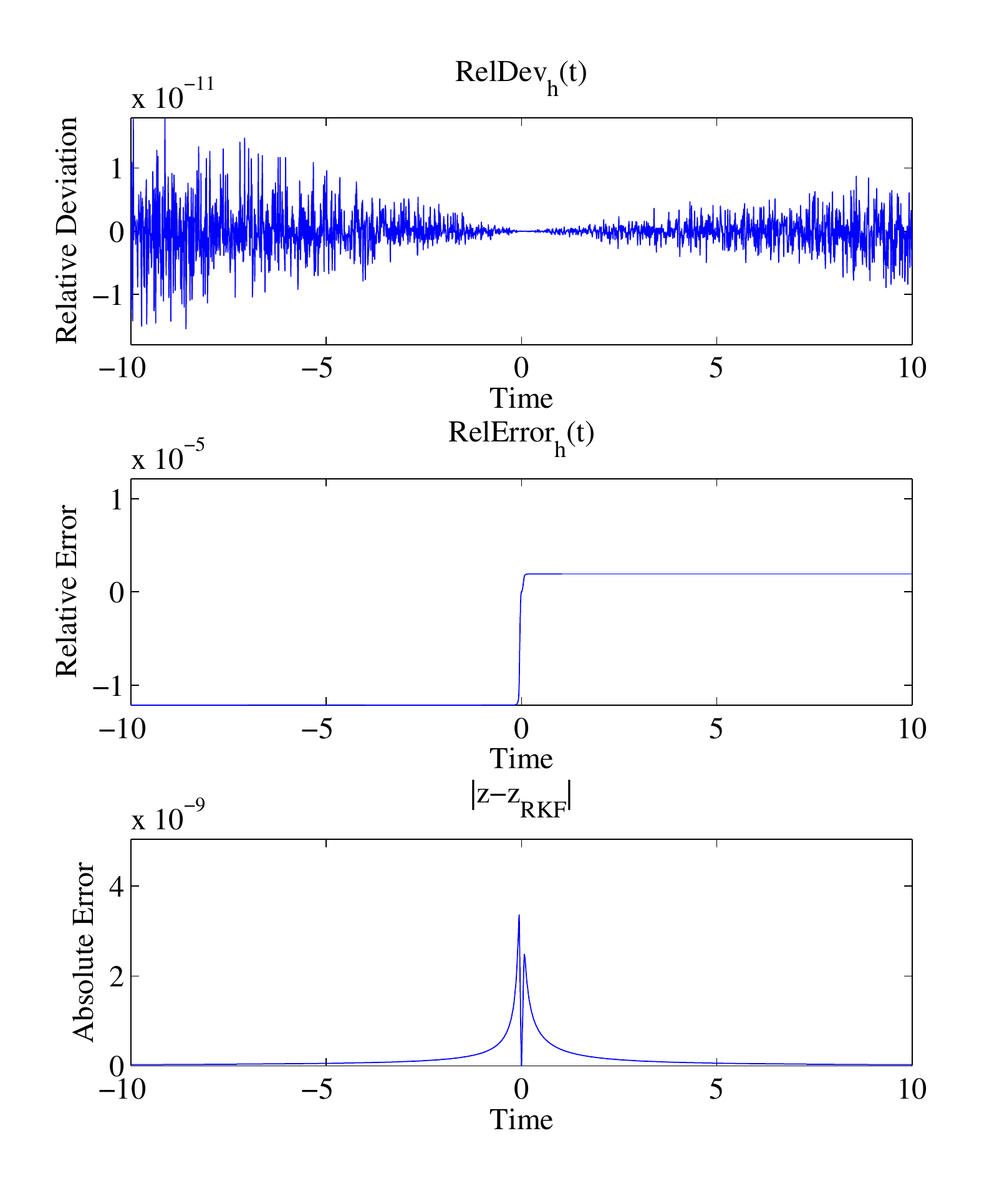}

(a) \hskip 190pt (b)
\caption{Comparison of the errors in a vicinity of an isolated essential singularity obtained (a) when using the 
$4^{th}$ order Runge--Kutta algorithm, and (b) when using the Runge--Kutta--Fehlberg algorithm. 
The vector field $-z^{2}\exp(\frac{1}{z})\del{}{z}$ was used with initial condition $z_{0}=0.1-i0.066$.}
\label{E}
\end{figure*}
As for a measure of the computational resources, we report the CPU time used by the algorithms in calculating 
the trajectories. 
In order to have a more realistic scenario, we measured the time it took to calculate the trajectories starting at 5 
different initial conditions $z_{0}$ in each vicinity and report the average times obtained.
 We report these times for each of the different generic neighborhoods in Table \ref{tablatiempos}.
 
\begin{table}
\renewcommand{\arraystretch}{1.3}
\caption{Average CPU time (measured in seconds) used for the calculation of 5 selected trajectories with the 
distinct algorithms, and for the calculation of the complete (global) field. The integration-time interval chosen was 
500 time units.}
{\tiny
\centering
\begin{tabular}{|c|c|c|c|c|}
\hline
& Time in seconds & Time in seconds & Time in seconds & Time in seconds\\
Type of  & using Newton & using $4^{th}$ order  & using RKF & for the calculation\\
neighborhood & method & RK4 algorithm & algorithm & of the global field \\ 
\hline
regular point (R) & 0.026 & 0.042 & 0.252 & 0.844 \\
\hline
singular point (Z) & 0.591 & 0.892 & 0.312 & 0.884 \\
\hline
singular point (P) & 0.492 & 0.627 & 0.312 & 0.887 \\
\hline	
singular point (E) & 0.420 & 0.652 & 0.358 & 0.830\\
\hline
\end{tabular}
\label{tablatiempos}
}
\end{table}%

\subsection{Discussion of the results of the comparison}
\subsubsection{Error comparison}
As a result of examining Figures \ref{R} thru \ref{E}, one first notices that the \emph{Newton} method proposed 
in this note has a very small error as can be observed on the graphs of the \emph{relative deviation from the 
exact solution} which in all cases but one remains below $10^{-15}$. Even in the case of the vicinity of an 
essential singularity (E) the \emph{relative deviation} remains below $10^{-11}$ which is at least 6 orders of 
magnitude better than the same case with the use of \emph{integration--based} techniques. The reason for this 
difference can be attributed to the fact that in a vicinity of an essential singularity the vector field has a mixture of 
behaviours as is further explained in (4) below.

As for the errors incurred by the \emph{integration--based} algorithms, there is a marked difference in the 
different generic cases:
\begin{enumerate}[label=(\arabic*)]
\item In the neighborhood (R) of a \emph{regular} point of the flow, the errors are very small: the relative 
deviation (that is the relative difference in the calculated value of the constant of motion $\rho(z)$ versus the 
value $\rho(z_{0})$) is less than $10^{-15}$, while the relative error and the absolute error in the RK4 case is of 
the order of $10^{-11}$, and $10^{-14}$ in the RKF case. Hence, even when there is a difference of 3 orders of 
magnitude between the RK4 and the RKF case, there is only one order of magnitude difference between using 
the RKF algorithm and the \emph{Newton} algorithm, in fact in this case the errors observed are mainly due to 
the numerical precision employed in the calculations. 
\item In the neighborhood (Z) of a \emph{zero} of $X(z)$, the errors are in fact smaller than those encountered in 
a vicinity of a regular point: the absolute error is of the order of $10^{-15}$ in both the RK4 and the RKF case, 
while the relative error and deviation differ only by one order of magnitude (even though there is a factor of 2 
between the relative error of the RK4 and the RKF case). This small difference is due to the fact that the 
trajectories are approaching a zero, hence the trajectories tend to converge. \\
In fact, in this case, the \emph{integration--based} algorithms need longer integration times (hence larger 
computational requirements) in order to visualize the trajectories as they approach the zero. The higher the 
order of the zero, the longer the integration times needed to obtain the same ``quality'' in the visualization.
\item In the neighborhood (P) of a \emph{pole} of $X(z)$, the errors behave pretty much as in the case of a 
regular point, except in the vicinity of $\tau=0$, where they increase quite noticeably. This is due to the fact that 
since the initial point $z_{0}$ of the trajectory is the closest point (on the trajectory) to the pole, this is where the 
values of the vector field are largest, and hence near this point is where the \emph{integration--based} 
algorithms may fail.
\item In the neighborhood (E) of an \emph{essential singularity} of $X(z)$, the errors respond in a more 
complicated manner, since we have a mixture of behaviours. This is expected on the following grounds: from an 
analytical viewpoint, one has by Picard's Theorem that the vector field takes on all but possibly one value in 
$\CC$ infinitively often in any neighborhood of the essential singularity, hence one expects to observe regions of 
behaviour similar to a pole, regions with the behaviour of a zero, and regions with behaviour similar to a regular 
value, all intermingled in a continuous (in fact analytical) way. \\
In this case the observed errors are quite big in the case of the RK4 algorithm, mainly because almost 
immediately the calculated trajectory ``jumps'' to another trajectory that is far from the original one. This can be 
seen in the relative error, where one observes that the relative error is constant for most of the backward and 
forward trajectory \emph{but the calculated value of $\rho(z)$ is very different from the original one 
$\rho(z_{0})$}. This same phenomena occurs in the case of the RKF algorithm, but on a much smaller scale. 
On the other hand, as was already observed, the \emph{relative deviation from the exact solution} remains 
below $10^{-11}$, that is the \emph{Newton} technique is quite accurate.
\end{enumerate}

\subsubsection{CPU time}
Due to the very different nature of the \emph{integration--based} algorithms and the \emph{Newton} method 
proposed, it is rather cumbersome to actually compare the computational resources that each algorithm utilizes 
to visualize a given trajectory. This phenomena is due to the fact that the CPU time used when visualizing with 
the \emph{integration--based} algorithms is directly dependent on the integration time, which in turn depends on 
the parametrization (speed) of the trajectory. 
For instance, for a trajectory that approaches a zero (Z) of the vector field, the usual methods need a very long 
\emph{time} interval in order to visualize the phase portrait (because as time advances, the trajectories are 
slower), in the case of a trajectory that approaches a pole (P), the opposite is the case.

On the other hand the \emph{Newton} method does not have this limitation: one of the advantages of the 
\emph{Newton} method, lies in the fact that one visualizes the complete trajectories corresponding to the value 
$\rho(z_{0})$ that lie in the chosen region. When visualizing using the \emph{Newton} method the \emph{time 
interval} does not matter.

In any case, as can be observed on Table \ref{tablatiempos}, the CPU time employed to calculate the trajectories 
behaves differently in the different generic cases: Apparently the usual \emph{integration--based} techniques are 
faster when there are no critical values of the flow, yet as soon as one approaches a critical value of the flow the 
\emph{Newton} method is faster than the \emph{integration--based} methods. Moreover
we can see an increase in CPU time in the case of the RKF method when compared with the RK4 method. It 
should be noticed that even though the CPU time required for the \emph{Newton} method is basically the same 
in all scenarios, this is not the case for the RK4 and or the RKF algorithms.

It should be noticed that the CPU time taken to calculate and visualize the complete field is just 3 to 4 times 
longer than the CPU time required for visualizing a single trajectory.

Also we would like to point out that another disadvantage that the usual \emph{integration--based} methods 
have and that the proposed method does not, is that in a small enough vicinity of an essential singularity the 
usual methods stop working, but our proposed method provides a clear visualization of the phase portrait (see 
Figure \ref{singEsen}). This is mainly because convergence fails, even with the RKF algorithm, near an essential 
singularity.
\begin{figure*}[htbp]
\centering
(a) \includegraphics[width=0.5\textwidth]{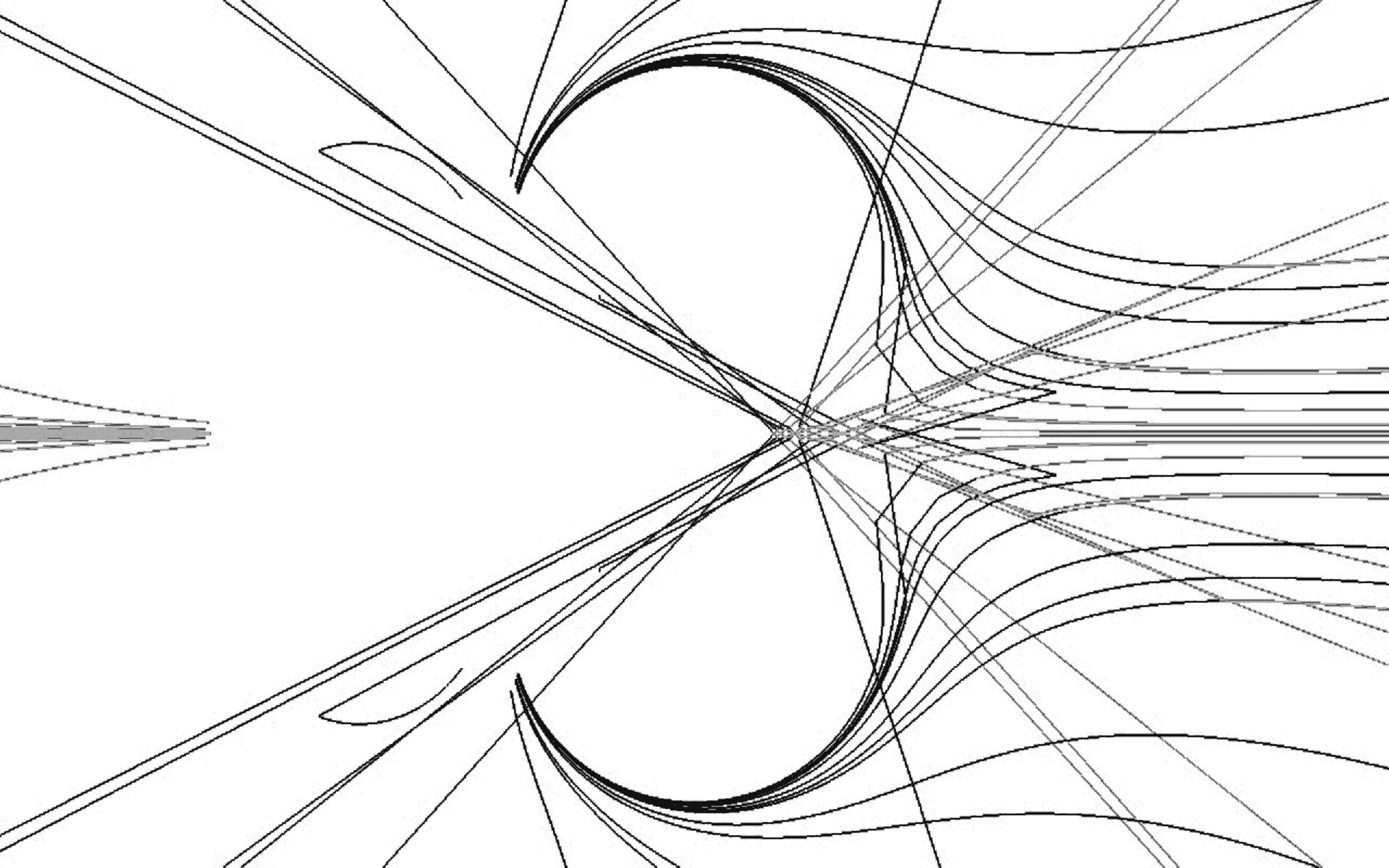}
\\
(b) \includegraphics[width=0.5\textwidth]{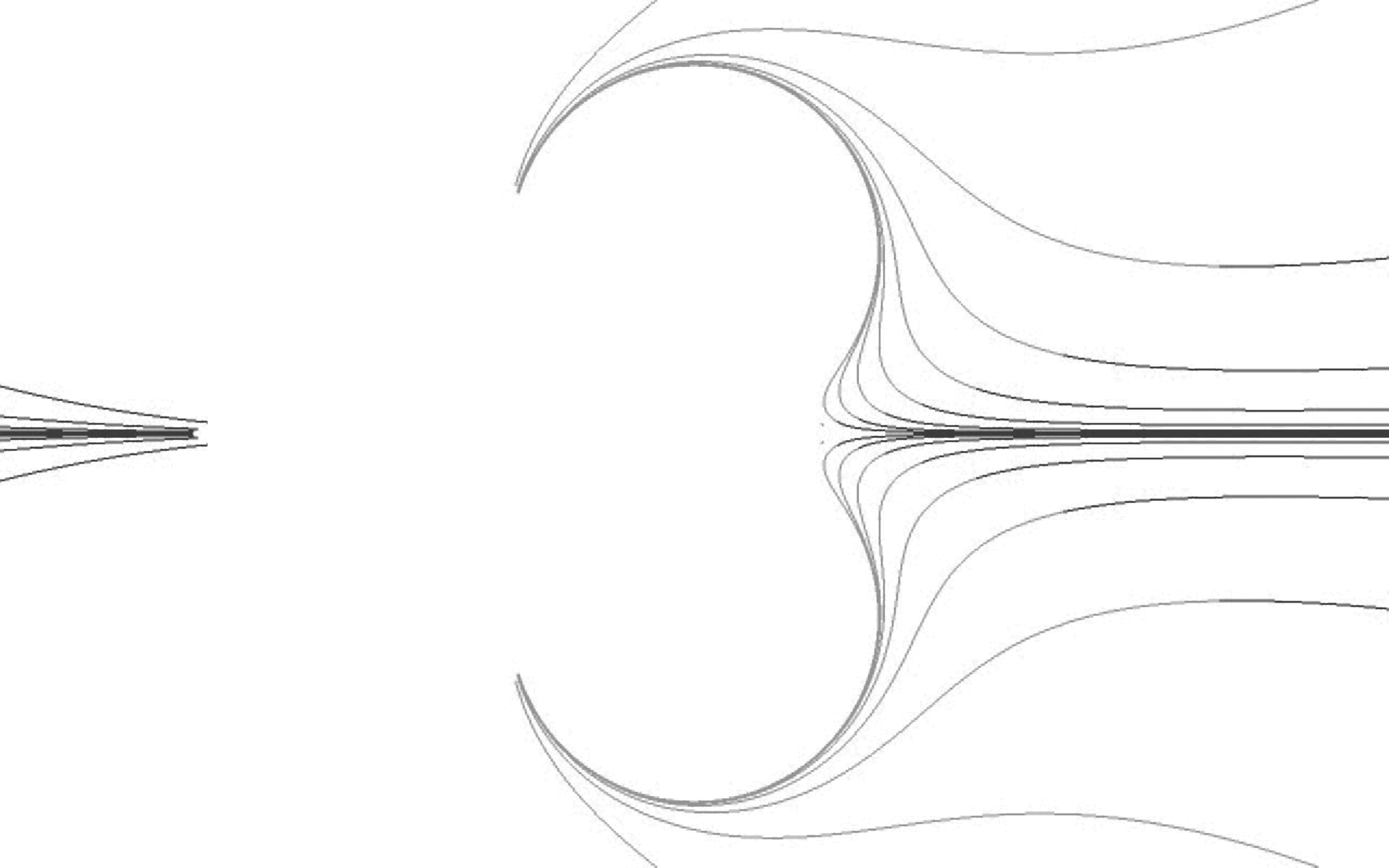}
\\
(c) \includegraphics[width=0.5\textwidth]{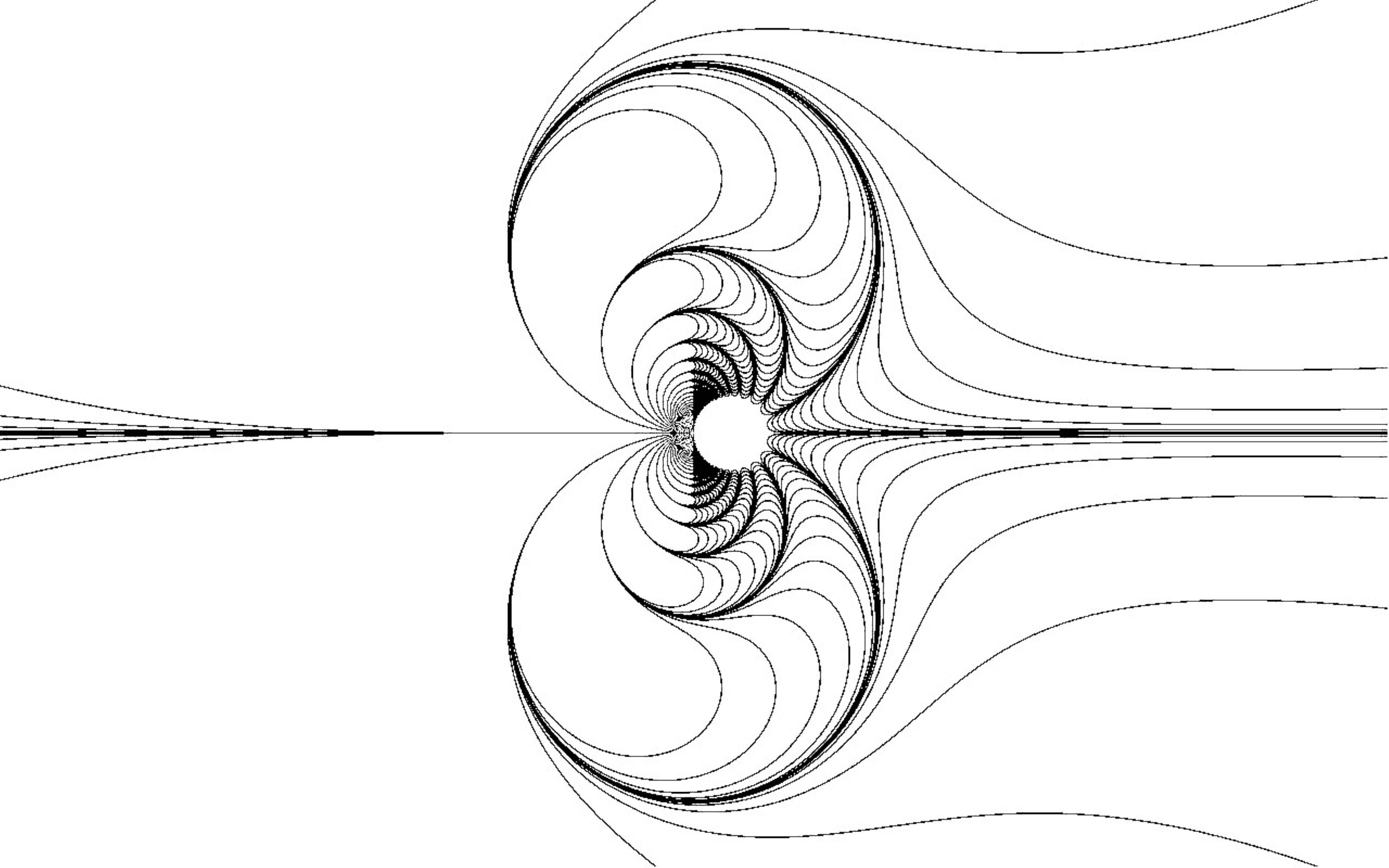}
\caption{In this figure we have plotted the trajectories corresponding to 18 initial conditions for the vector field 
$\-z^{2}\exp(\frac{1}{z})\del{}{z}$ in a vicinity of the essential singularity at 0. In (a) the trajectories were obtained 
using the RK4 algorithm, in (b) using the RKF algorithm, and in (c) with our proposed algorithm. 
In (a) and (b) the trajectories where plotted for forward and backward time.
The initial conditions can be identified in (b) as the place where the trajectories change 
from gray to dark grey.}
\label{singEsen}
\end{figure*}

\subsection{Implementation}
All of the algorithms described (including the visualizations in Figures \ref{campoPolos} thru \ref{coshPlus2} and 
Figure \ref{singEsen}) were implemented using C++ and OPENGL on Mac OS X 10.5 running on a 2.16 Ghz 
Intel Core 2 Duo processor with 2 GB of DDR2 SDRAM. 

\begin{remark}
The algorithms were implemented in C++ on the Mac OS X 10.5 in order to do an accurate comparison between 
the usual integration methods.
However, we have developed a version of the software readily available in JAVA with a function parser that can 
accept either 
\begin{enumerate}[label=(\alph*)]
\item rational vector fields $X(z)=f(z)\del{}{z}$ with $f(z)=\frac{p(z)}{q(z)}$, $p,q\in\CC[z]$, or 
\item a distinguished parameter $\Psi_{X}(z)=\int^{z} \frac{1}{f(\zeta)} d\zeta$, with $f(z)$ 
a singular complex analytic function.
\end{enumerate}
This way we can visualize the singular complex analytic vector field $X(z)=f(z)\del{}{z}$ both on the Riemann 
sphere $\CW$ and on a user specified rectangular region of the plane $\CC$. This was done so that the 
software could be accessible to a wider audience. The latest version can be found at

\noindent
{\small
\url{https://www.dropbox.com/sh/xfuor27nf820mwo/AAACyp0EsGx6Ain49VgUWRFMa?dl=0}
}
\end{remark}

\section{Generalizations and open problems}
\label{generalizaciones}
In this section we present several directions that are natural to follow. We start by generalizing the techniques 
and method to closed Riemann surfaces. 
In particular, we present an example of a vector field on the torus visualized with the 
\emph{Newton} method presented in this note.

We further present some ideas and an overview of how this can work for vector fields in $\RR^{n}$, generalizing 
to differential manifolds of dimension $n$, and finally presenting some open (and natural) questions.
Some of these problems are currently being explored by the authors and will be presented elsewhere.

\subsection{Generalizing to analytic vector fields on Riemann surfaces}
In order to actually implement the visualization of singular complex analytic vector fields on Riemann surfaces 
with the \emph{Newton} method there are two obvious alternatives: using charts of the Riemann surface; and 
using invariant vector fields.

\subsubsection{Visualization using charts}\label{variedadescomplejas}
An immediate generalization of what has been presented is to visualize singular complex analytic vector fields 
on Riemann surfaces by using the charts associated to the Riemann surface. Note that in fact we have already 
done so specifically for the case of the Riemann sphere using Stereographic Projection.

The idea is simple: using charts one can do the visualization on the image in $\CC$ of the chart and then 
return it to the Riemann surface with the inverse map of the chart (as was explained, in the case of the 
Riemann sphere, in \S\ref{implementation}).

Thus to generalize the method described above to Riemann surfaces, it is enough to take the singular complex 
analytic vector field $\widetilde{X}$ defined locally in $M$ and by pushforward with the coordinate of the 
chart $f_{\alpha}$, find the singular complex analytic vector field $X={f_{\alpha}}_{*}\widetilde{X}$ in $\CC$. 
Since this field is a Newton vector field (by Theorem \ref{everythingisnewton}) then we can proceed to 
calculate its trajectories in $\CC$ and finally take them to $M$ using $f_{\alpha}^{-1}$, thereby solving the case 
of Riemann surfaces. 

Note that this approach also works with the usual \emph{integration--based techniques}, and with other vector 
field visualization techniques as well. In fact in \cite{LiEtal} they use this approach to visualize vector fields on 
``arbitrary surfaces'', using a \emph{texture-based} approach. One of the problems that they encounter, and 
deal with effectively, is that visual artifacts appear when passing from one parametrization to another (that is 
when visualizing the vector field in $U_{\alpha}\cap U_{\beta}$ via $f_{\alpha}^{-1}$ or when visualizing the 
vector field using $f_{\beta}^{-1}$). This problem is expected to be present if using \emph{direct--flow 
visualization} and \emph{geometric flow visualization} techniques as well.

\subsubsection{Visualization using invariant vector fields}\label{InvariantVectFields}
In this case the idea is to use the fact that closed Riemann surfaces $M$ can be modeled as the quotient of 
the universal cover $\widetilde{M}$ of $M$ with a subgroup $\Gamma$ of the automorphism group, 
$\text{Aut}(\widetilde{M})$, of $\widetilde{M}$, i.e.
$$
M=\widetilde{M}/\Gamma.
$$

Then a complex vector field on $\widetilde{M}$ of the form 
\begin{equation}\label{invariantVF}
\widetilde{X}(z)=\widetilde{f}(z)\del{}{z},
\end{equation}
with $\widetilde{f}(z)$ invariant under the group $\Gamma^{*}$, where $\Gamma^{*}$ denotes the group action 
of $\Gamma$ on the tangent vector bundle as in Lemma \ref{pullbackformula}, 
descends to a complex vector field $X([z])$ on $M$.
Thus visualizing $\widetilde{X}(z)$ on $\widetilde{M}$ is equivalent to visualizing $X([z])$ on $M$.

By the uniformization theorem of closed Riemann surfaces, $\widetilde{M}$ is either the Riemann sphere 
$\CW$, the complex plane $\CC$, or the hyperbolic plane $\HH=\{z\in\CC : \Im{z}>0\}$. Furthermore, the first 
two cases produce only one more family of orientable Riemann surfaces: \emph{tori} or Riemann surfaces of 
genus $g=1$ (which arise from considering the group $\Gamma$ generated by two non--collinear translations 
on $\CC$). The third case produces a plethora of orientable Riemann surfaces, all of whom have negative 
curvature (Riemann surfaces of genus $g>1$).

In the case of tori, the class of functions which are invariant under the action of the automorphism group of a 
torus are called \emph{elliptic functions} and have been extensively studied. In particular recall 
(see \S\ref{meromorphic}) that the first author has previously shown that all elliptic vector fields are in fact 
Newton vector fields \cite{AP-2}, by showing that an elliptic function $f(z)$ can be expressed as a quotient 
$-\frac{\Phi(z)}{\Phi'(z)}$, for an explicitly constructed $\Phi$ in terms of 
Weierstrass $\sigma$ and $\zeta$ 
functions (and their derivatives). Thus one can apply the techniques introduced in this paper and reduce the 
problem of visualizing the vector field to that of visualizing the level curves of $\rho(z)=\argp{ \Phi(z) }$.

As an example, in Figure \ref{torus} the elliptic vector field
\begin{equation}\label{weirstrass}
\widetilde{X}(z)=-\frac{\wp(z)}{\wp'(z)}\del{}{z},
\end{equation}
is visualized using the techniques described in this note: the strip flows of $\rho(z)=\argp{\wp(z)}$ are plotted 
in Figure \ref{torus} (a),
moreover the corresponding vector field on the torus is visualized in Figure \ref{torus} (b). 
It should be noted that the vector field given by \eqref{weirstrass} was previously studied by G.\,F.\,Helminck
{\it et al.}, 
see \cite{HelminckEtAl} for further details. In particular they showed that up to conjugation the family of 
vector fields of the form \eqref{weirstrass} consist of three classes characterized\footnote{
The three types of behaviour are related to structural stability and the underlying lattice: if the underlying lattice 
is non--rectangular the Newton vector field is structurally stable, otherwise the Newton vector field is not 
structurally stable and there are two options for the underlying lattice, square and rectangular but not square. 
Up to conjugacy these are all the options available.
}
by the form of the parallelogram 
spanned by the parameters $\omega_{1}$, $\omega_{2}$ defining\footnote{
The Weierstrass $\wp$--function is in fact characterized by the lattice $\Omega(\omega_{1},\omega_{2})$ with 
basis $\{\omega_{1},\omega_{2}\}\subset\CC$ such that $\omega_{1}/\omega_{2}\not\in\RR$. It can be proved 
that $\wp$ is doubly periodic with periods precisely $\omega_{1}$ and $\omega_{2}$. See \cite{HelminckEtAl} 
and references therein for further details.
} 
$\wp(z)$. They visualized a couple of trajectories using a $4^{th}$ order Runge-Kutta integration algorithm for 
the case corresponding to $\omega_{1}=1$, $\omega_{2}=\frac{1}{4}+ \frac{5}{4} i$, compare 
Figure \ref{torus} (a) with \cite{HelminckEtAl} figure 8. 
\begin{figure*}[htbp]
\centering
\includegraphics[height=0.45\textwidth]{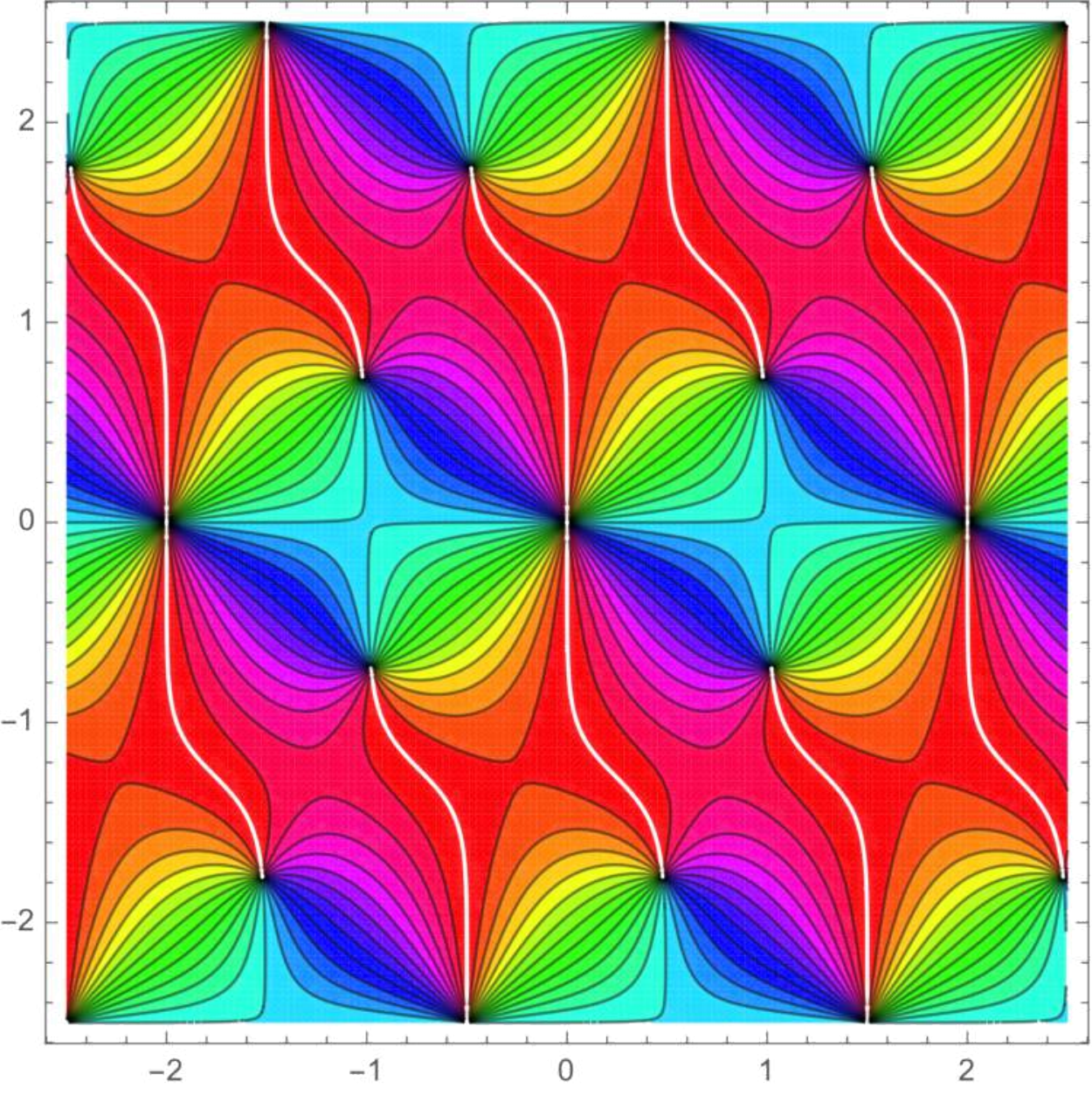}
\includegraphics[height=0.45\textwidth]{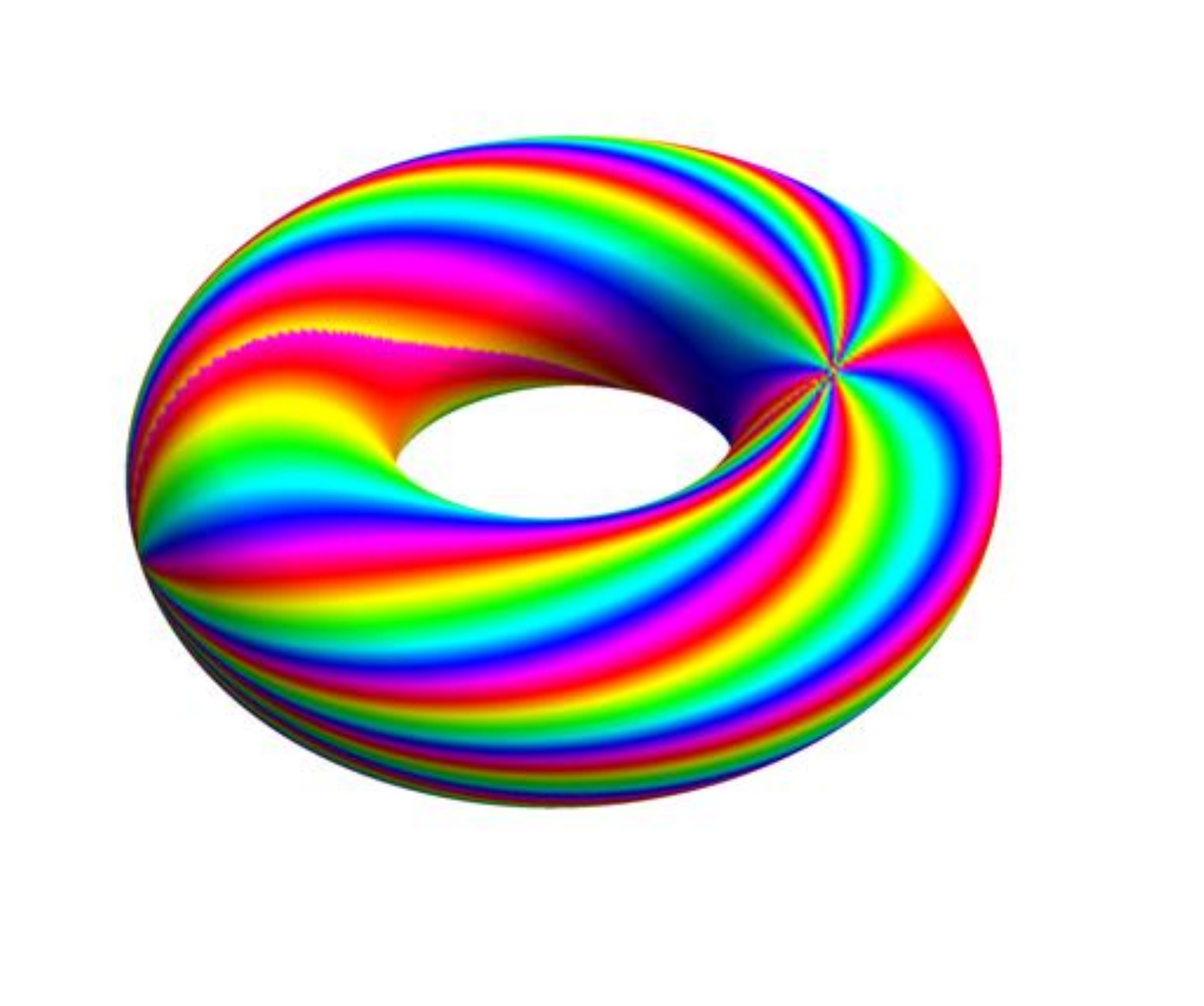}

(a) \hskip 150pt (b)
\caption{Visualization of the elliptic vector field $\widetilde{X}(z)=-\frac{\wp(z)}{\wp'(z)}\del{}{z}$. 
(a) on the plane and (b) the corresponding vector field on the torus. 
The basis $\{\omega_{1},\omega_{2}\}\subset\CC$ for the underlying lattice 
$\Omega(\omega_{1},\omega_{2})$ is $\{\omega_{1}=1, \omega_{2}=\frac{1}{4}+ \frac{5}{4} i\}$.}
\label{torus}
\end{figure*}
In the case of Riemann surfaces of genus $g>1$, a similar scheme will work: the $\Gamma$--invariant vector 
field $\widetilde{X}(z)$ given by \eqref{invariantVF} (which is characterized by the $\Gamma$--invariant 
function $\widetilde{f}(z)$ which is analytic on $\HH\backslash \mathcal{S}$, where $\mathcal{S}$ is a discrete 
set) will be meromorphic in $\CC$, and hence one will be able to explicitly represent them as Newton vector 
fields. Since the characterization of the $\Gamma$--invariant meromorphic functions, for the case of genus 
$g\geq3$, is a whole subject on its own\footnote{The $\Gamma$--invariant functions in this case are known as 
automorphic functions and have very interesting applications to number theory, amongst other things.}, the 
visualization of vector fields on Riemann surfaces of genus $g\geq3$ is left for future work.

\subsection{Generalizing to vector fields in $\RR^{n}$}
In this direction there is already some work reported in the literature: 
in 1988 S.\,A.\,Burns and J.\,I.\,Palmore \cite{BurnsPalmore} presented a generalization of some of the 
techniques they developed in \cite{PalmoreBurnsBenzinger}, to some vector fields in $\RR^{n}$.
In \S\ref{GenRn} we present an overview of these techniques and mention some of its limitations. 

An alternative generalization that is based upon the present work, but uses commutative algebras with unity, 
instead of the complex number field, is outlined in \S\ref{AlgRn}. 

\subsubsection{Vector--valued Newton method in $\RR^{n}$}\label{GenRn}
Since the results presented in this section have been published elsewhere 
(see \cite{PalmoreBurnsBenzinger}, \cite{BurnsPalmore}), we only give a rough sketch of the steps that need 
to be followed, without going into the details.

Given a differentiable function $F:\RR^{n}\to \RR^{n}$, locally one--to--one, and with a differentiable inverse, then $F$ 
defines the following vector field
$$H(x)=-\left[ DF(x)\right]^{-1} F(x),$$
where $DF(x)$ is the Jacobian matrix of partial derivatives. 
This vector field is known as the Newton vector field associated to $F$. 

On the other hand, given a vector field $H(x)$, it is possible to show that in a neighborhood of a regular 
point $x_{0}$ of $H$ with $H(x_{0})\neq0$, there exists $F:\RR^{n}\to \RR^{n}$ such that $H$ is the Newton 
vector field associated to $F$.
Also, $H$ and $F$ satisfy 
\begin{equation}\label{def0}
DF(x) H(x) = -F(x).
\end{equation}

This last equation, as in the complex case, is intimately linked to the solutions of
\begin{equation}\label{def1}
\frac{d x}{d\tau}=H(x),
\end{equation}
since if $x(\tau)$ satisfies
$$F(x(\tau))=\e^{-(\tau-\tau_{0})}F(x(\tau_{0})),$$
then $x(\tau)$ is a trajectory solution of \eqref{def1}. 
Hence, by considering $F=(F_{1}, F_{2},\cdots, F_{n})$ and $H=(H_{1},H_{2},\cdots,H_{n})$
the relation \eqref{def0} can be re-written as
$$H(x)\cdot \nabla F_{i}(x)=-F_{i}(x),\ i=1,2,\cdots,n.$$ 
so by proposing that the $F_{i}$ be of the form $F_{i}=\exp\left[G_{i}\right]$
one has
$$H(x)\cdot\nabla G_{i}(x)=-1,$$
and each difference $G_{ij}=G_{i}-G_{j}$ satisfies 
$$H(x)\cdot\nabla G_{ij}(x)=0,$$
so that the $G_{ij}$ are constant on the trajectories of $\frac{d x}{d\tau}=H(x)$.

Note that in this case, as opposed to the complex case, one still has to solve a system of differential equations 
to find the auxiliary functions $G_{ij}$, but these equations are usually much simpler to solve than the original 
ones that define the trajectories and that are solutions to \eqref{def1}.
\emph{In a couple of cases examined in \cite{PalmoreBurnsBenzinger}, \cite{BurnsPalmore} these auxiliary 
functions $G_{ij}$ can be found, but there is no known technique that works in general.} 

\subsubsection{Generalized analytic functions and the Newton method}\label{AlgRn}
An alternative to the previous method that will work for a large class of vector fields has been presented by 
some of the authors in \cite{AP-FA-LG-YR}. 
It is based upon the notions of \emph{generalized analytic functions}: the basic idea is that under certain 
conditions one may define a commutative algebra with unity of dimension $n$. 
Moreover, the \emph{differentiable} functions of the algebra satisfy certain ``Cauchy--Riemann equations'', in 
an analogous manner as the case of the usual analytic functions over $\CC$ (hence these functions are also 
called \emph{generalized analytic functions}). 

Once this commutative algebra  with unity is defined, one can develop an analog of analytic function theory for 
these \emph{generalized analytic functions}, so that finally the scheme presented in this note can be carried 
over to the algebra, providing a framework where the \emph{generalized analytic vector fields} can be 
visualized by the same techniques.

For instance in the case of dimension 2, the class of vector fields for which this generalization will work, 
includes the real valued vector fields 
$$F(x,y)=u(x,y)\del{}{x}+v(x,y)\del{}{y},$$
whose Jacobian matrix of partial derivatives is of the form 
$A J A^{-1}$, where $A$ is an invertible matrix and $J$ is a matrix in one of the following normal forms 
$$\begin{pmatrix}a&b\\-b&a\end{pmatrix}, \qquad \begin{pmatrix}a&b\\0&a\end{pmatrix}, \qquad 
\begin{pmatrix}a&0\\0&b\end{pmatrix}.$$
Further work related to vector fields in this direction can 
be found in \cite{FriasYLopez}.

\subsection{Generalizing to vector fields on differentiable $n$--dimensional manifolds}
Once the issue of visualizing vector fields in $\RR^{n}$ is solved, an immediate option is to generalize 
this to differentiable $n$-dimensional manifolds, once again using the charts associated to the differentiable 
manifolds, or by using invariant vector fields on the universal cover. The description is completely analogous to 
the one given in \S\ref{variedadescomplejas} and \S\ref{InvariantVectFields}, so we omit it.

\section{Application: Visualizing complex functions}\label{visualizingcomplexfunctions}

As an application of the proposed geometrical method for visualizing singular complex analytic vector fields, we 
look into the problem of visualizing complex functions.

The naive approach of visualization of complex functions by their graph fails on account of a simple dimension 
count: the domain and range of complex functions require two real dimensions each, thus the graph is simply a 
(real) two dimensional surface embedded in (real) four dimensional space. This is not easy to visualize because 
of our natural limitation to visualize (real) three dimensional space.
Thus it is not surprising that complex analysis has made huge advances with a symbolic/algebraic approach.

\noindent
However, with the advent of computers and the ease of use of them as tools for visualizing mathematical 
objects, there has been a dramatic increase in the geometrical aspects of complex analysis. 
In particular, a related topic with beautiful images is the iteration of complex functions.
As far as we know, it was B. Mandelbrot who first considered images of iterations, introducing what is now 
known as the \emph{Mandelbrot set} \cite{Mandelbrot}; later on the work of R.\,L.\,Devaney 
(see \cite{Alexander-Devaney} \S4, and \cite{Devaney}), P.\,Blanchard \cite{Blanchard}, J.\,Milnor \cite{Milnor} 
and many others, made iteration of complex functions known to a much wider audience.
Nowadays this is an area of intense and very productive research.

\noindent
For the visualization of complex functions, a starting point which includes diverse articles, course materials and 
applets is ``Websites related to Visual Complex Analysis" \cite{websites}. 

In this section we do a quick review of some of the more common visualization techniques, emphasizing what 
exactly each technique can or can not do.
One wishes to be able to distinguish zeros and poles (with their respective order/multiplicity), critical points and 
other singularities, in particular essential singularities.

\noindent
Understanding essential singularities is much more involved than the case of zeros and poles; a reasonable first 
step can be found in \cite{AP-MR}, 
where the particular case of isolated essential singularities arising from logarithmic branch points over a finite 
asymptotic value $a\in\CC$ is studied. 
In the cited work, angular sectors associated to the germ of a singular analytic vector field about an isolated 
essential singularity are introduced; very roughly speaking each of these new \emph{entire} angular sectors 
consists of an infinite collection of hyperbolic and elliptic angular sectors, see Figures 7 and 8 of this work and 
figures 1, 2, 3 and 5 of \cite{AP-MR}. 

\noindent
\textbf{Warning:} Our discussion related to essential singularities in this section is restricted to isolated 
essential singularities arising from logarithmic branch points over 
finite asymptotic values $\{ a \}\subset \CC$. Many 
other families of essential singularities exist, for example $\infty\in\CW$ for trigonometric functions, functions 
that have non conformal punctures 
etc., see \cite{Guillot1} and \cite{Guillot3}.

\subsection{Image of regions under $f(z)$}
A classical approach investigates images of specific curves (and regions) under the mapping $f(z)$. 
Even though this is part of any introductory course in complex analysis and is essential for gaining an intuitive 
grasp of elementary functions, it is hard to implement with more general complex functions. For some readily 
available free software see \cite{Akers}, \cite{MathewsFreeSoftware} and for some commercial software see for 
instance \cite{LascauxGraphics}.

As an example of this technique, in Figure \ref{imageMap} we visualize the rational functions 
$f_{1}(z)=\frac{z^{3}-1}{z^{2}}$ and $f_{2}(z)=\frac{z^2}{z^3-1}$ using the image of a square region of $\CC$.
\begin{figure}[htbp]
\begin{center}
\includegraphics[width=0.75\textwidth]{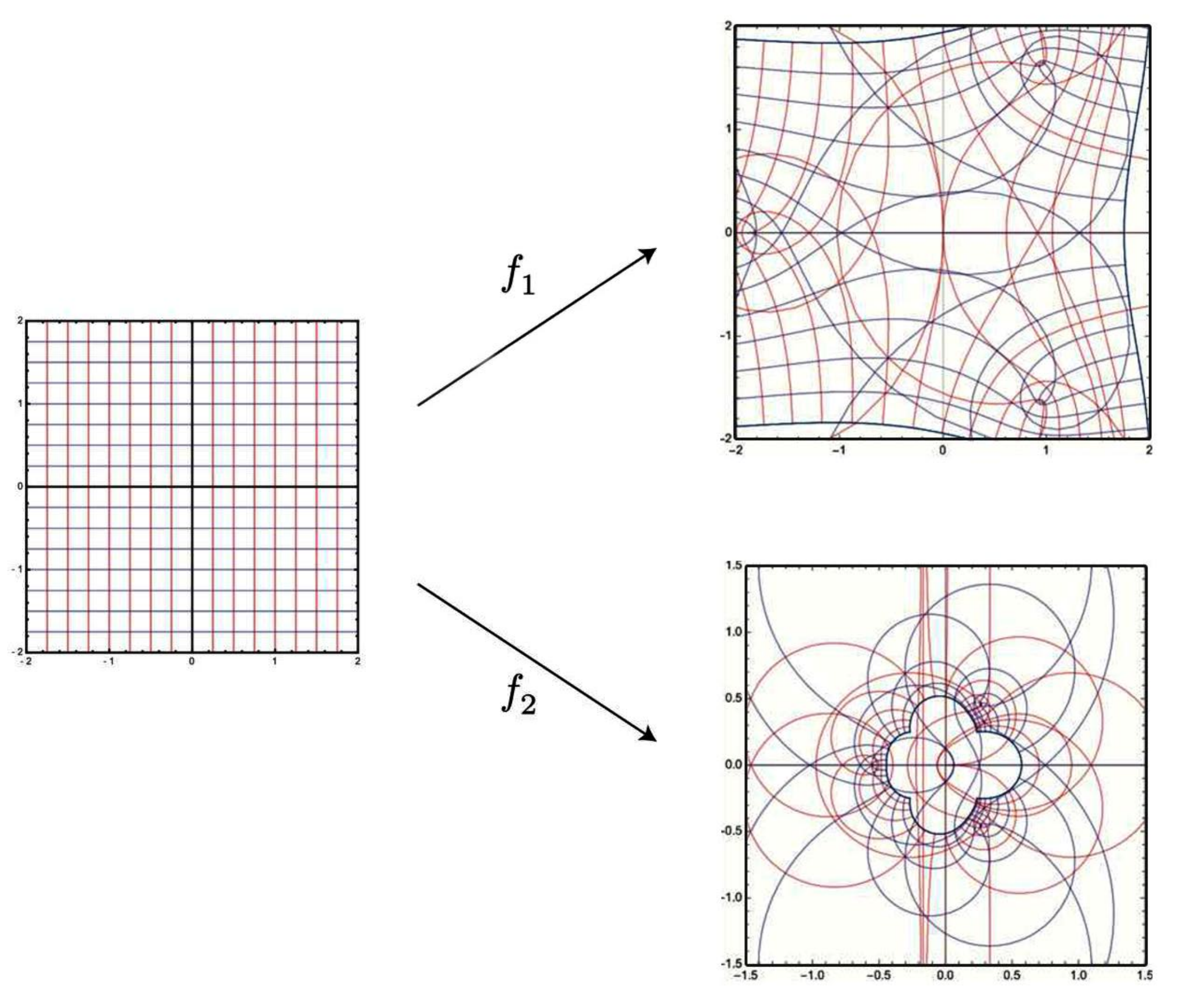}
\caption{Example of the rational functions $f_{1}(z)=\frac{z^{3}-1}{z^{2}}$ and $f_{2}(z)=\frac{z^2}{z^3-1}$ 
visualized using the image of the square region $[-2,2]\times[-2,2]\subset\CC$.}
\label{imageMap}
\end{center}
\end{figure}
As it is clear from this example, by pushing forward the image, the result is a multivalued map which makes it 
difficult to ``read--off'' the information regarding the map. This technique works best with small domains.

\smallskip
Advantages: easy to understand and present, even in elementary courses.

Disadvantages: hard to implement for general complex functions, moreover it is not easy to ``read'' the 
information regarding the function.

\subsection{Tilings {\it a la Klein}}
To avoid the previously encountered difficulty, it is possible to use the pullback.
In this direction, a technique pioneered, as far as we know, by F.\,Klein in his ``Protokolle'', see \cite{KleinClay} 
and \cite{Chislenko-KleinProtocols}, is the following. 

\noindent
Let $f:\CW_{z}\longrightarrow\CW_{w}$ be a complex analytic function, and let 

\centerline{
$R_{f}=\{w_{0}, w_{1}, \ldots, w_{s} \}\subset\CW_{w}$}

\noindent
be its set of ramification values, and 

\centerline{
$C_{f}=\{c_{0}, c_{1}, \ldots, c_{\ell} \}\subset\CW_{z}$}

\noindent
its critical set.
The simplest case is when $f$ is a rational function 
(moreover, the method applies for more general classes of functions).

\noindent
Secondly, let $\gamma\subset\CW_{w}$ be an oriented Jordan path running through $R_{f}$. 
Then $\CW_{w}\backslash \gamma$ is the union of two open simply connected domains. 

\smallskip
\noindent
Recognizing $\CW_{w}\backslash \gamma$ as a two color tiling of the sphere, 
\emph{the pullback $f^{*}\gamma$ determines a second tiling $\CW_{z}\backslash f^{*}\gamma$.}
That is 
$$
(f,\gamma)\longrightarrow f^{*}\gamma.
$$

\smallskip
\noindent
For example if $f$ is a rational function of degree $d$, the number of tiles in $\CW_{z}\backslash f^{*}\gamma$ 
is $2d$.
In Figure \ref{figMosaicos} we visualize as an example the rational functions $f_{1}(z)=\frac{z^{3}-1}{z^{2}}$ 
and $f_{2}(z)=\frac{z^2}{z^3-1}$.
\begin{figure}[htbp]
\begin{center}
\includegraphics[width=0.75\textwidth]{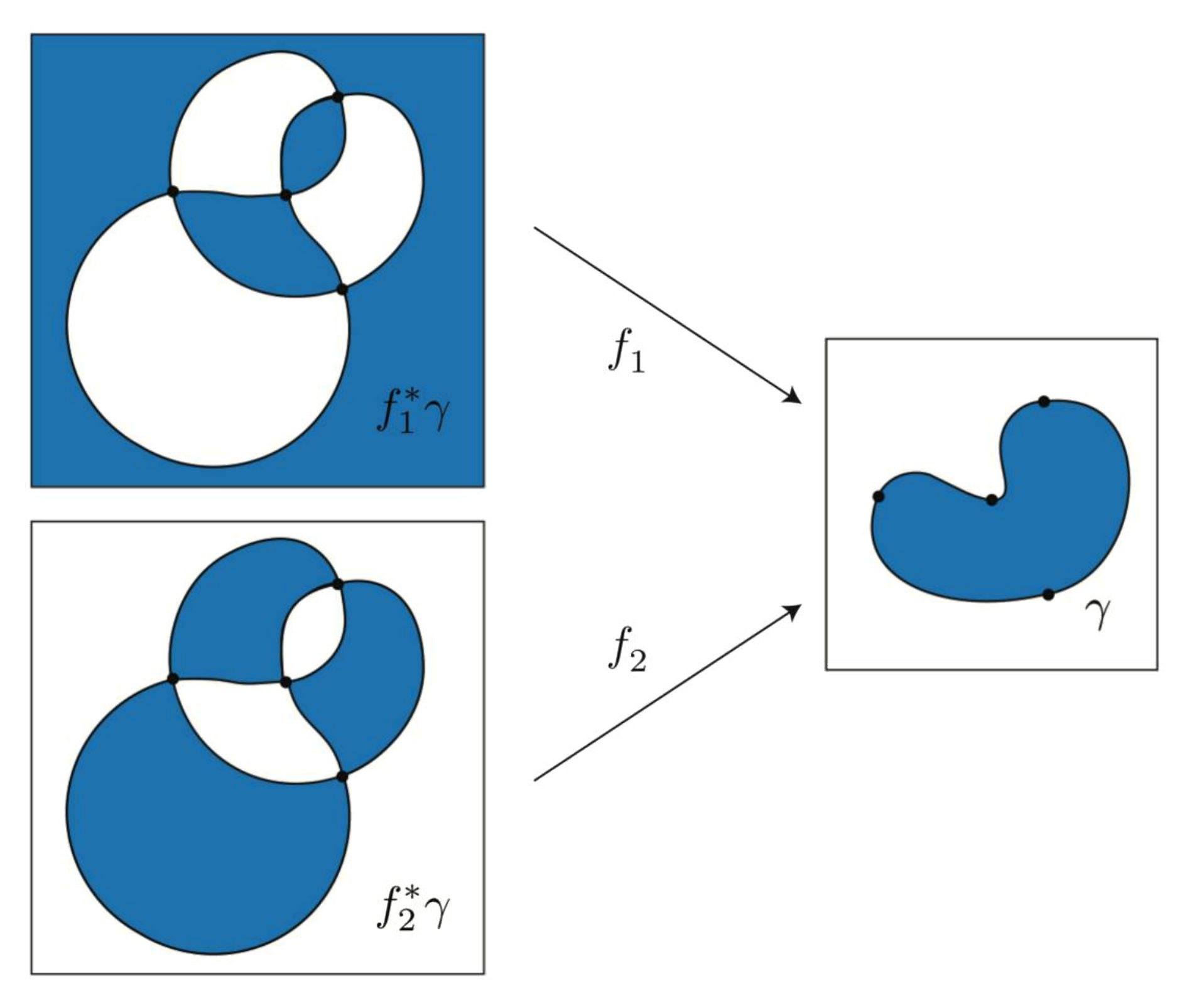}
\caption{Example of the rational functions $f_{1}(z)=\frac{z^{3}-1}{z^{2}}$ and $f_{2}(z)=\frac{z^2}{z^3-1}$ 
visualized with tilings a la Klein. In both cases there are four critical points and four critical values.}
\label{figMosaicos}
\end{center}
\end{figure}

It is to be noted that unless the poles and zeros of the function $f$ are critical points, the poles and zeros will be 
undistinguishable by this technique.

\begin{table}[htp]
\caption{Interpretation of tilings of $f^{*}\gamma$}
\begin{center}
\begin{tabular}{|c|c|}
\hline
Function $f$ & Tiling of $f^{*}\gamma$ \\
\hline
\hline
$z_{0}$ is a simple & \\
zero of $f$ & \\
\hline
$z_{0}$ is a simple & \\
pole of $f$ & \\
\hline
$z_{0}$ is an isolated essential & $z_{0}$ is a vertex of the tiling \\
singularity of $f$ & with infinite tiles bordering it \\
\hline
$z_{0}$ is a critical point & $z_{0}$ is a vertex of the tiling \\
of order $k$ & with $2k$ tiles bordering it \\
\hline
\end{tabular}
\end{center}
\label{tilingsalaKlein}
\end{table}%

Advantages: the procedure can be done by hand: for rational functions it is simple. However, it also works for 
infinite ramified coverings $f:M\longrightarrow\CW$.

Disadvantages: considers the critical points 
(that can be regular points) of $f$, not the zeros or poles. 
Depends strongly on the choice of $\gamma$, in fact the method ``visualizes'' pairs $(f,\gamma)$: 
changing $\gamma$ for fixed $f$ determines very different tilings.

\subsection{Analytical landscapes}
Another traditional concept for visualizing complex functions is the so called \emph{analytical landscape}, 
apparently introduced by E.\,Maillet \cite{Maillet} in 1903, which basically is a graph of the absolute 
value $\abs{f(z)}$ of a complex function $f(z)$,
$$f(z)\longrightarrow \{(z,\abs{f(z)})\ \vert\ z\in\CC\}.$$
Of course, even though useful, not all the information of the function $f(z)$ could be conveyed: 
the argument of $f(z)$ was lost. However by drawing also the lines of constant argument this could be solved. 
Even better, by use of colormaps it is possible to draw \emph{colored analytic landscapes} where isochromatic 
lines correspond to lines of constant argument. See Figure \ref{landscape}.
\begin{figure}[htbp]
\begin{center}
\includegraphics[width=0.65\textwidth]{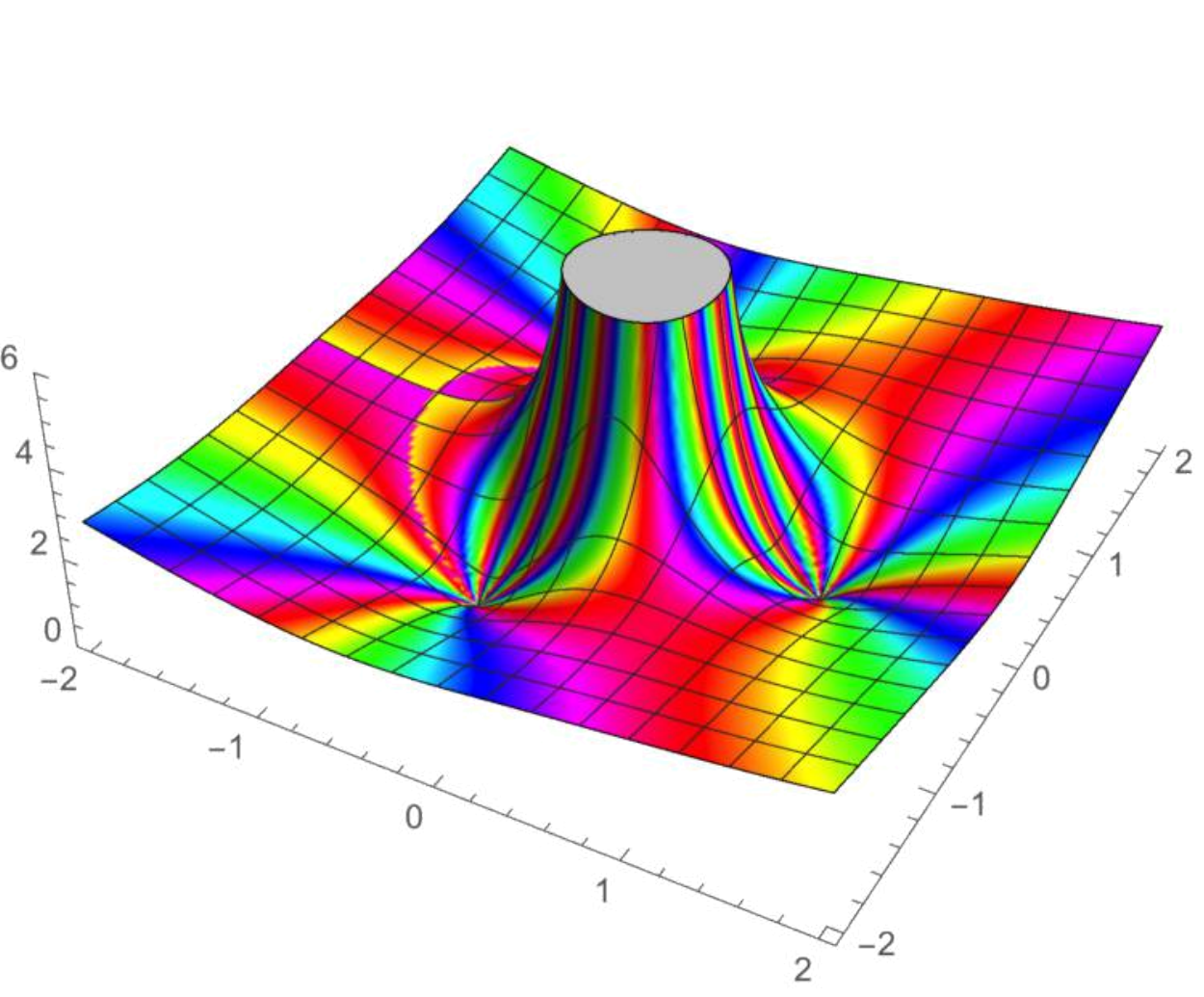}
\includegraphics[width=0.65\textwidth]{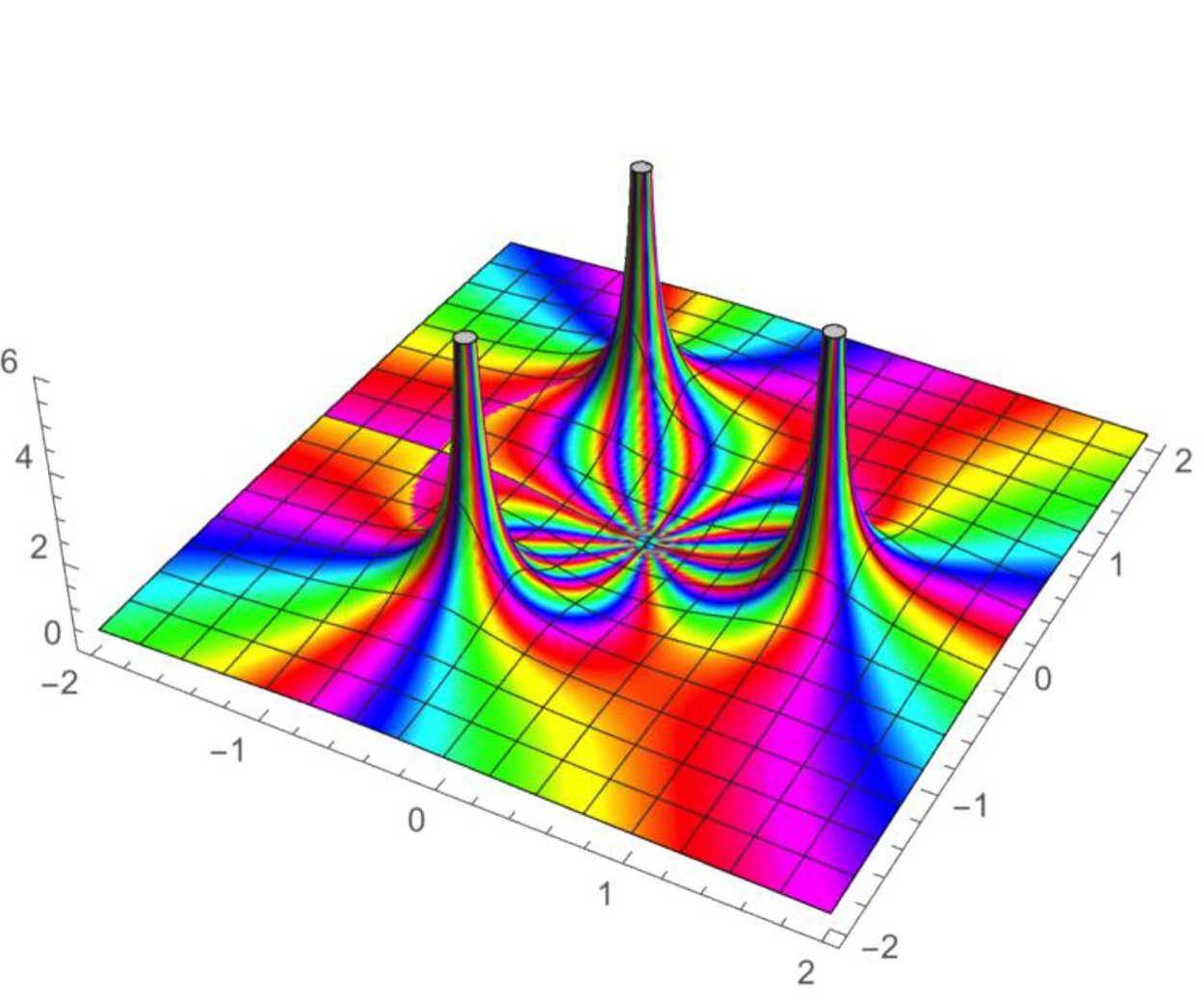}
\caption{Example of the rational functions $f_{1}(z)=\frac{z^{3}-1}{z^{2}}$ and $f_{2}(z)=\frac{z^2}{z^3-1}$ 
visualized with the colored analytic landscape technique.}
\label{landscape}
\end{center}
\end{figure}

However, there is an inherent difficulty of reading off the information related to the function $f(z)$ when some of 
the essential parts might be hidden in the valleys behind the mountain rims or covered by towers of poles.

Advantages: useful, particularly the colored analytical landscapes since it contains complete information of the 
function $f(z)$.

Disadvantages: inherent difficulty reading off the information 
related to the function, since it is a three dimensional image 
projected onto two dimensions.

In these situations, the best one can do is to view the colored analytic landscape straight from the top. 
The result is a flat color image sometimes called the \emph{phase portrait} 
of $f(z)$, representing the color coded phase or argument of $f(z)$,
$$f(z)\longrightarrow \argp{f(z)}.$$
According to E.\,Wegert \cite{Wegert}, the phase is better suited than the modulus to understand a function and 
to reconstruct the properties of $f(z)$. 
This alternative approach using colormaps for coloring the domain of $f(z)$ is called \emph{domain colorings}.

\subsubsection{Domain colorings}
The main idea is to pullback the color, {\it i.e.} assign a color to each point of the range of $f(z)$ and then color 
each point $z$ of the domain of the function according to the color of its image $f(z)$. 
Apparently these were first used in the WWW by F. Farris \cite{Farris}.

\noindent
A very common colormap is to use the polar form of complex numbers to assign a color to each point on the 
range $\CW$; the magnitude is assigned an intensity or brightness (0 is assigned black, $\infty$ is assigned 
white, or vice versa), while the argument is mapped to a ``rainbow'' color wheel
using the polar form of complex numbers. 
\emph{It is clear that phase portraits are a special form of domain colorings where the modulus of $f(z)$ is 
ignored.}

Let us look a little more in detail on how phase portraits of functions convey the information related to the 
function $f(z)$. 
An order $s$ zero of $f(z)$ is seen as a rainbow color wheel repeated exactly $s$ times around the placement 
of the zero, with the color wheel following the same arrangement as in the range about the origin. On the other 
hand, an order $-\kappa\leq-1$ pole of $f(z)$ is seen as a rainbow color wheel repeated exactly $\kappa$ times 
around the placement of the pole, with the color wheel following the \emph{opposite} arrangement as in the 
range about the origin. 
In other words, if we do not know the color map associated to the range it is impossible to distinguish between a 
pole and a zero; this same idiosyncrasy appears also when we have a gray colored image. 
See Figure \ref{domaincoloring} for an example of a phase portrait for the two functions 
$f_{1}(z)=\frac{z^{3}-1}{z^{2}}$ and $f_{2}(z)=\frac{z^2}{z^3-1}$.
\begin{figure}[htbp]
\begin{center}
\includegraphics[width=0.45\textwidth]{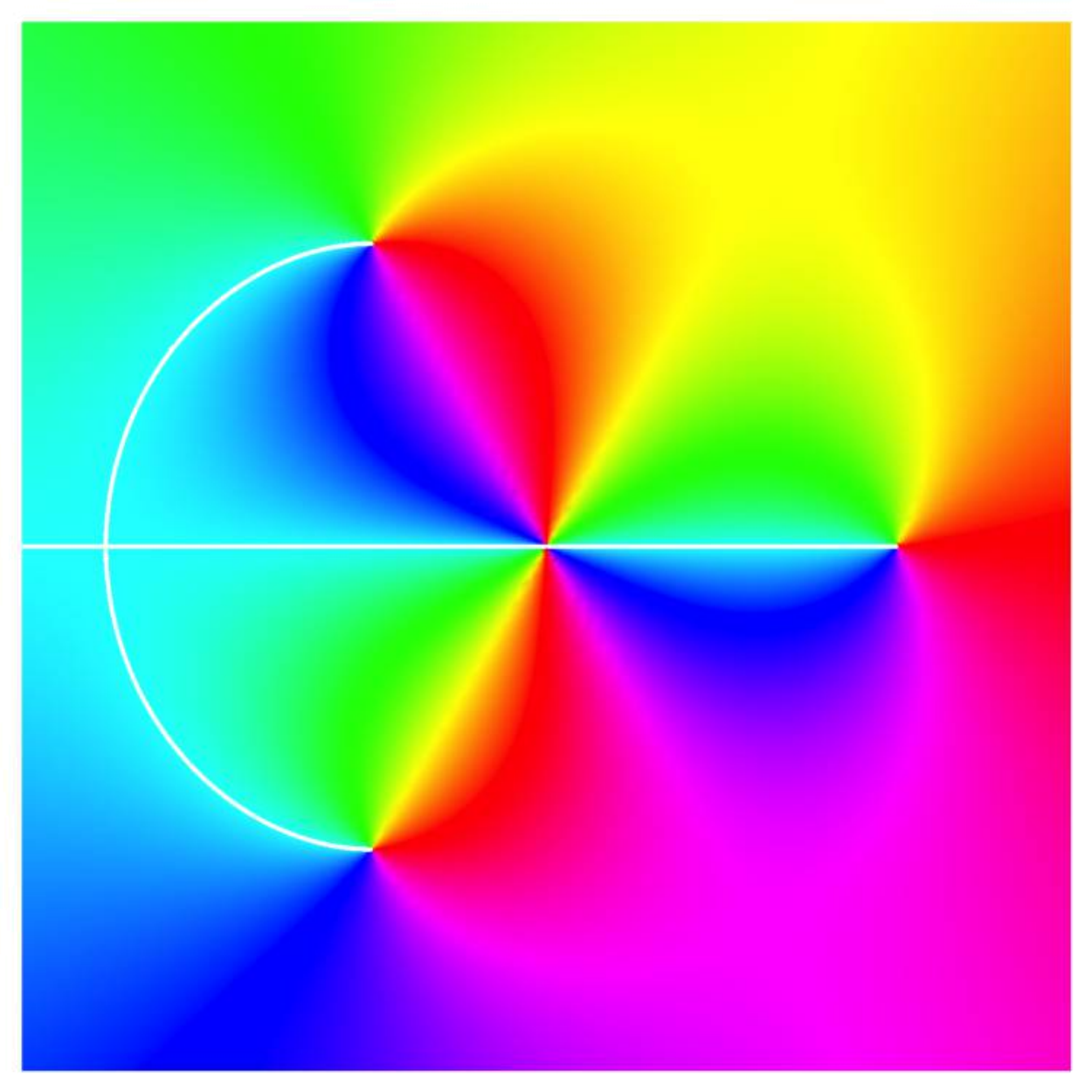}
\hskip 15pt
\includegraphics[width=0.45\textwidth]{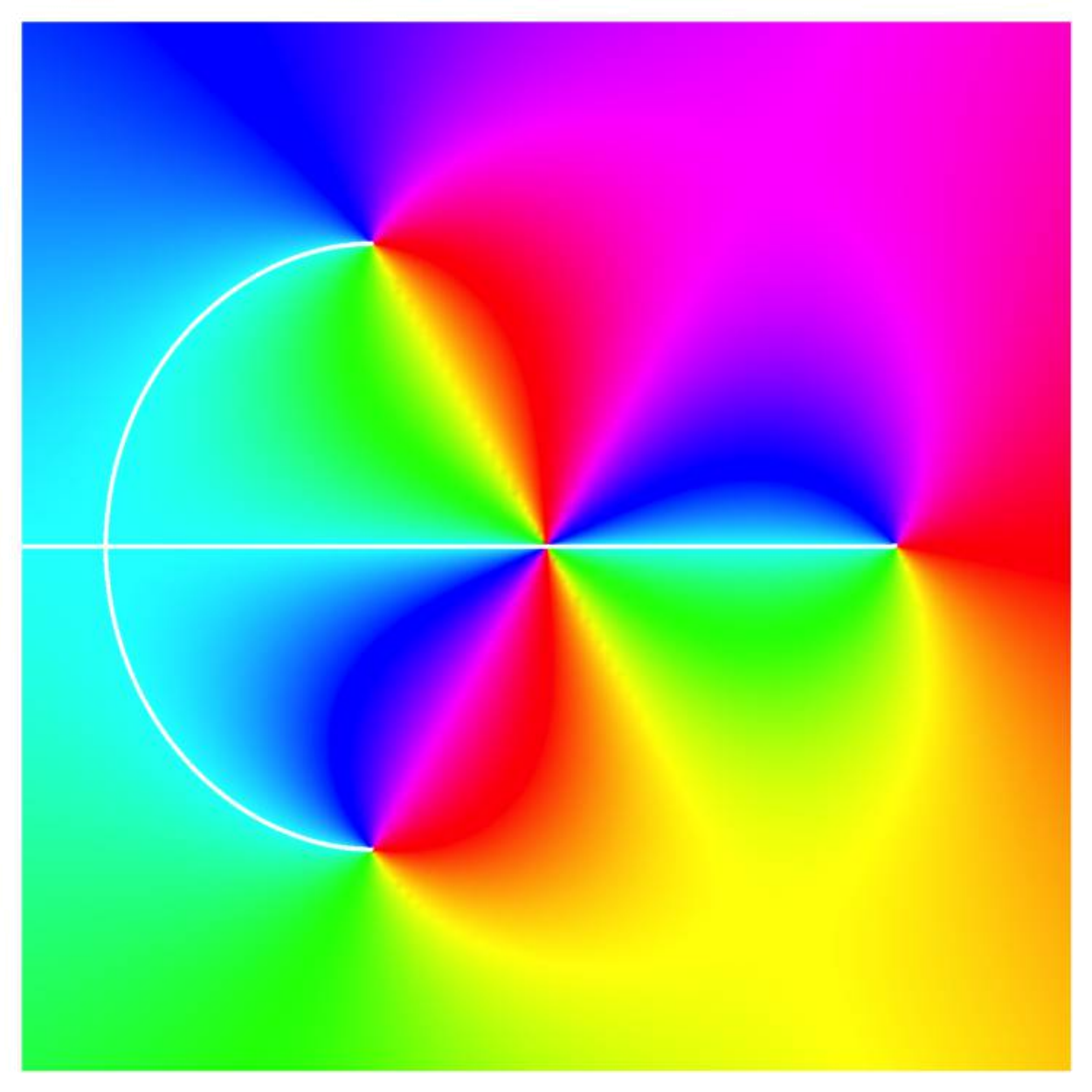}
\caption{Example of the rational functions $f_{1}(z)=\frac{z^{3}-1}{z^{2}}$ and $f_{2}(z)=\frac{z^2}{z^3-1}$ 
visualized by their phase portrait. 
It is impossible to distinguish a pole from a zero unless we know beforehand the colormap used. Note that the 
order of the pole or zero is distinguished by the number of repeated color wheels around the pole or zero.}
\label{domaincoloring}
\end{center}
\end{figure}
              
\noindent
Identifying isolated essential singularities and/or accumulation points of poles, zeros or essential singularities is 
easy: they are characterized by the fact that any neighborhood of the singularity intersects infinitely many 
isochromatic lines of the same color, see \cite{Wegert} p.\,181.

\noindent
On the other hand, since critical points are precisely where conformality is lost, critical points of $f$ are located 
where isochromatic lines form saddle points.

\noindent
In Table \ref{domaincoloringtable} we summarize the above information.

\begin{table}[htp]
\caption{Domain coloring: Interpretation of phase portrait of $f$}
\begin{center}
\begin{tabular}{|c|c|}
\hline
Function $f$ & Phase portrait of $f$ \\
\hline
\hline
$z_{0}$ is an order $s\geq1$ & rainbow color wheel \\
zero of $f$ & repeated exactly $s$ times around $z_{0}$ \\
\hline
$z_{0}$ is an order $-\kappa\leq-1$ & opposite rainbow color wheel \\
pole of $f$ & repeated exactly $\kappa$ times around $z_{0}$ \\
\hline
$z_{0}$ is an isolated essential & any neighborhood of the singularity \\
singularity of $f$ & intersects infinitely many isochromatic lines \\
&  of the same color \\
\hline
$z_{0}$ is a critical point & the isochromatic lines form a saddle\\
\hline
\end{tabular}
\end{center}
\label{domaincoloringtable}
\end{table}%

Advantages: easy to implement, essential singularities are easily distinguished.

Disadvantages: impossible to distinguish between a pole and a zero if the color map associated to the range is 
unknown.

\subsection{Visualizing complex functions via the phase portrait of vector fields}
The following methods use the phase portraits of certain vector fields to visualize complex functions. 
In order to plot the phase portrait we use the methodology developed in the preceding sections, 
particularly \S\ref{metodo}. A useful interpretation is expressed by the following diagram between complex 
analytic sections of trivial, cotangent and tangent holomorphic line bundles over $M$
\begin{equation}\label{lo-de-siempre-con-F}
F(z) \longrightarrow d F = F'(z)\,dz \longrightarrow \frac{1}{F'(z)}\del{}{z}.
\end{equation}
Thus \emph{the real trajectories of $\frac{1}{F'(z)}\del{}{z}$ are paths
 whose image under $F(z)$ are horizontal lines; in other words 
the horizontal trajectories of the quadratic differential $(F'(z))^2\,dz^{2}$}.

Note that because of \eqref{dosmaneras}, when using the 
techniques developed in this work for visualizing singular complex 
analytic vector fields $X(z)=f(z)\del{}{z}$, we have the option of 
using $\Psi(z)$ or $\Phi(z)$.

\subsubsection{Visualizing $f(z)$ on $\CC$ via $X_{f}(z)=f(z)\del{}{z}$ or 
$\widetilde{P_f}(z)=\frac{1}{f(z)}\del{}{z}$}

In their classic work ``Complex Variables'' G.\,Polya and G.\,Latta propose the use of vector fields to 
visualize complex functions $f(z)$ on the complex plane $\CC$, see \cite{Polya} p.\,61. 
Specifically, they propose the use of what is now known as the Polya vector field 

\centerline{
$f(z)\longrightarrow P_{f}(z)=\bar{f(z)}\del{}{z}$.}

\noindent 
They use the Polya vector field as opposed to the more immediate vector field 

\centerline{
$f(z)\longrightarrow X_{f}(z)=f(z)\del{}{z}$, } 

\noindent
because Polya's vector field has the following physical interpretation: \emph{a complex function $f(z)$ is analytic 
in a region $D\subset\CC$ if its Polya vector field is differentiable, divergence free and curl free throughout the 
region $D$}. 
A further advantage is that Polya's approach can also be used to visualize and estimate complex integrals. 
In this direction recently, B.\,Braden contributes to the Polya vector field interpretation of complex integrals, 
see \cite{Braden2}.

\noindent
However, it is to be noted that the Polya vector field $P_{f}(z)$ is not a holomorphic vector field 
on $D\subset\CC$; it is anti--holomorphic. This can be circumvented by recalling first that 
$\bar{f(z)}=\frac{\abs{f(z)}^2}{f(z)}$ and secondly that multiplying a vector field by a non--vanishing scalar 
factor does not alter the phase portrait, it just changes the parametrization. Hence we introduce the following.
\begin{definition}
Given a complex valued function $f(z)$ on $D\subset\CC$, a priori $C^0$, 
considering the operator 
\begin{equation}
f(z)\longrightarrow\widetilde{P_{f}}(z)=\frac{1}{f(z)}\del{}{z}, 
\end{equation}
the image is the \emph{normalized\footnote{This 
parametrization is precisely the one that makes complex 
time have norm one in the metric provided by the 
pullback of the flat metric $(\CC,\delta)$ under $\Psi(z)=\int^{z} f(\zeta) \, d\zeta$.
} 
Polya vector field of $f$}.
\end{definition}
\noindent
Note that, 
if $f(z)$ is a complex singular analytic function 
then $\widetilde{P_f}(z)$ is a complex singular analytic vector field with the same phase portrait (but a different 
parametrization) as the usual Polya vector field $P_{f}(z)$.

In 1996, T.\,Newton and T.\,Lofaro \cite{NewtonLofaro}, use $X_{f}(z)=f(z)\del{}{z}$ to visualize functions of a 
complex variable $f(z)$; including $f(z)=\e^{z}$ on the Riemann sphere.
They use Runge--Kutta--Fehlberg RKF4(5) integration based techniques to plot the flow or phase portrait of 
$X_{f}(z)$.
Also, T.\,Needham \cite{Needham}, champions the use of vector fields for visualizing complex functions; 
however he does not propose any particular implementation and/or discuss visualization techniques for vector 
fields.

\smallskip
Visualizing complex functions $f(z)$ on $\CC$ by plotting the phase portrait of $X_{f}(z)$ or 
$\widetilde{P_f}(z)$ provides the following mayor advantage.
Because of the normal forms for meromorphic vector fields, 
see Proposition \ref{prop:normalforms} and Figure \ref{forma-normal}, 
\emph{poles and zeros of vector fields are unequivocally recognized (including their order) 
via the topology\footnote{
Associated to each pole of order $-\kappa\leq-1$ of a meromorphic vector field $X$, there are exactly 
$2(\kappa+1)$ \emph{hyperbolic} angular sectors, similarly zeros of order $s\geq 2$ have exactly 
$2(s-1)$ \emph{elliptic} angular sectors; simple zeros have sectors that depend on the residue of the 
associated 1--form $\omega_{X}$. See Figure \ref{forma-normal}.
} 
of the phase portrait}, even without color plots. 
For examples see Figures \ref{ratfunc}, \ref{campoPolos} and \ref{vfplots}. 
\begin{figure}[htbp]
\begin{center}
\includegraphics[width=0.45\textwidth]{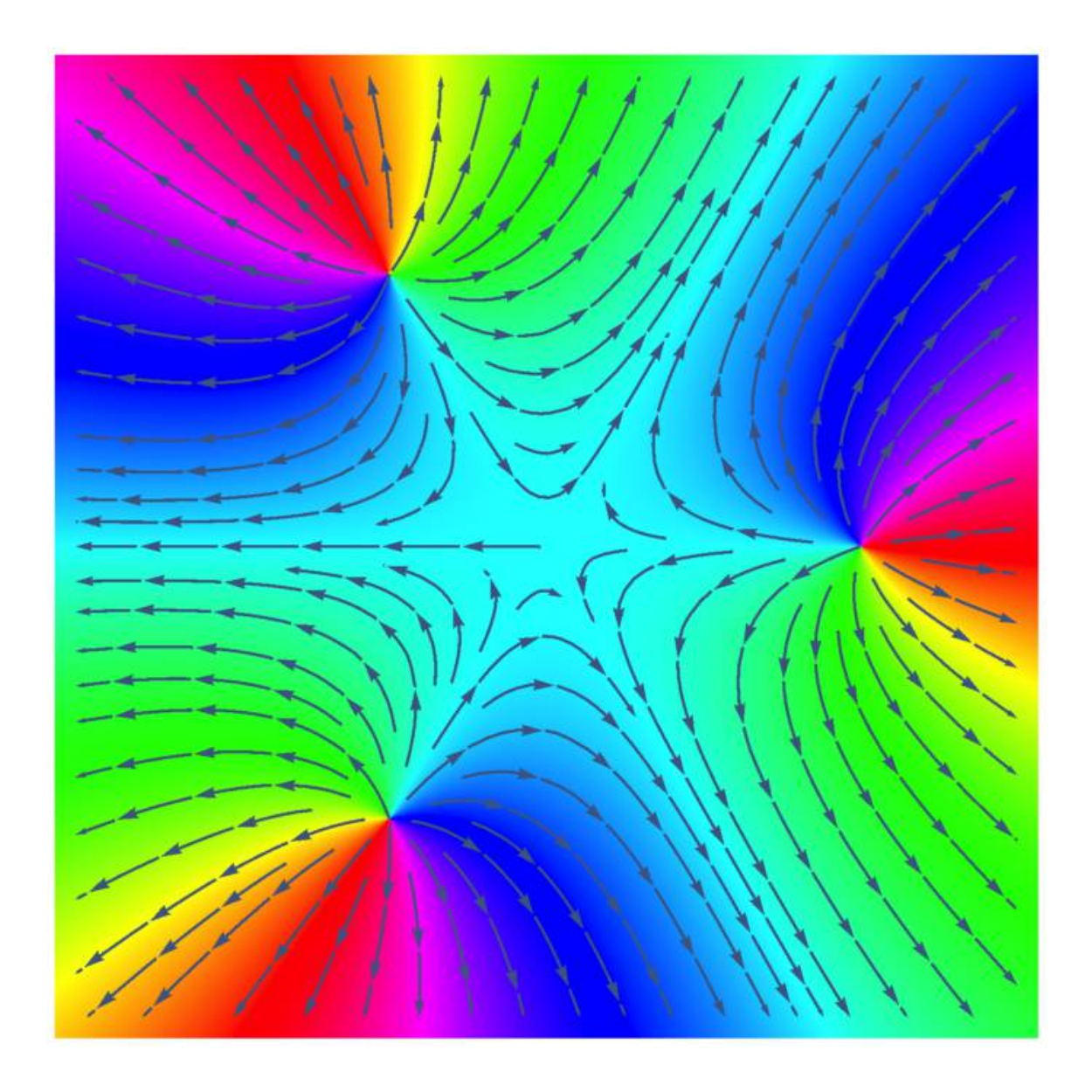}
\hskip 5pt
\includegraphics[width=0.45\textwidth]{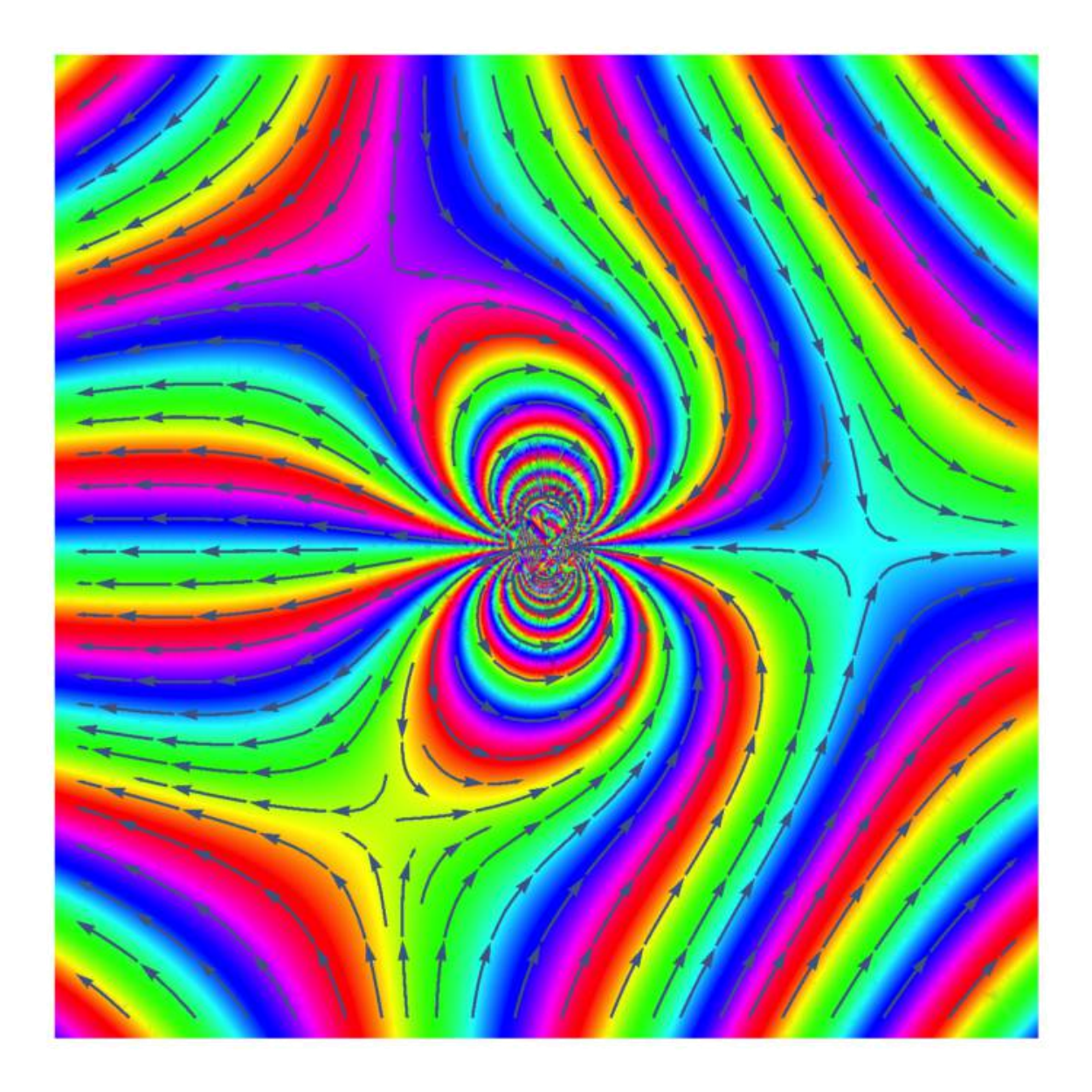}
\caption{Example of the rational functions $f_{1}(z)=\frac{z^{3}-1}{z^{2}}$ and $f_{2}(z)=\frac{z^2}{z^3-1}$ 
visualized as $X_{f_{1}}(z)=f_{1}(z)\del{}{z}$ and $X_{f_{2}}(z)=f_{2}(z)\del{}{z}$ respectively. 
In this case there is no problem distinguishing between the two functions and all the information can be read 
from the plots.
Note that the phase portrait of $X_{f_{1}}(z)$ and $P_{f_{2}}(z)$ are the same (with a different parametrization), 
similarly $X_{f_{2}}(z)$ and $P_{f_{1}}(z)$ have the same phase portrait.
}
\label{vfplots}
\end{center}
\end{figure}

\noindent
The phase portraits of singular analytic vector fields with isolated essential singularities 
arising from logarithmic branch points over finite asymptotic values 
$\{ a\}\subset \CC$, can with distinguished by the presence of \emph{entire} angular sectors, 
see \cite{AP-MR} \S5.2 for the appropriate definitions. Examples can be found in 
Figures \ref{campoExp}, \ref{CampoExpZ3} and more examples in \cite{AP-MR} and \cite{AP-MR-2}.

\smallskip
However, a mayor flaw related to visualizing a function $f(z)$ on an arbitrary Riemann surface $M$ via the 
normalized 
Polya vector field $\widetilde{P_f}(z)=\bar{f(z)}\del{}{z}$, or the vector field $X_{f}(z)=f(z)\del{}{z}$, is that 
considered as 
tensors $f(z)$ and $\widetilde{P_f}(z)$ (or $X_{f}(z)$) are quite different: {\it i.e.} when acted upon by $Aut(M)$ 
they do not transform in the same way. 
This is made evidently clear in the following example.
\begin{example}
Consider $f(z)=\frac{z^2}{(z-1)^2}$ on the Riemann sphere $\CW$. 
Of course $\infty\in\CW$ is a regular point. 
However when considering 

\centerline{
$X_{f}(z)=\frac{z^2}{(z-1)^2}\del{}{z}$ \quad or \quad
$\widetilde{P_f}(z)=\frac{(z-1)^2}{z^2}\del{}{z}$}

\noindent
on $\CW$, each has a double zero at $\infty\in\CW$.
Considering  rational functions $f$
and vector fields $X$ on any 
compact Riemann surface $M$ of genus $g$, 
topological invariants are
the Chern class of the trivial and tangent holomorphic 
line bundles as follow
$$
\# \hbox{zeros}(f) - \# \hbox{poles}(f) = 0, 
\ \ \ 
\#  \hbox{zeros}(X) - \# \hbox{poles}(X) = 2-2g.
$$

\end{example}

\noindent
A summary of the information that can be observed with this visualization technique is presented in 
Table \ref{phaseportraitXtable}.
{\small
\begin{table}[htp]
\caption{Visualization of $f$ via phase portrait of $X_{f}(z)=f(z)\del{}{z}$ and 
$\widetilde{P_f}(z)=\frac{1}{f(z)}\del{}{z}$}
\begin{center}
\begin{tabular}{|c|c|c|}
\hline
Function $f$ & Phase portrait of & Phase portrait of \\
& $X_{f}(z)=f(z)\del{}{z}$ & $\widetilde{P_f}(z)=\frac{1}{f(z)}\del{}{z}$ \\[4pt]
\hline
\hline
$z_{0}$ is a simple zero of $f$ & angular sector corresponding & 4 hyperbolic angular sectors \\
& to a center & around $z_{0}$ \\
$z_{0}$ is an order $s\geq2$ & $2(s-1)$ elliptic angular & $2(s+1)$ hyperbolic angular \\
zero of $f$ & sectors around $z_{0}$ & sectors around $z_{0}$ \\
\hline
$z_{0}$ is a simple pole of $f$ & 4 hyperbolic angular sectors & angular sector corresponding \\
& around $z_{0}$ & to a center \\
$z_{0}$ is an order $-\kappa\leq-2$ & $2(\kappa+1)$ hyperbolic angular & $2(\kappa-1)$ elliptic angular \\
pole of $f$ & sectors around $z_{0}$ & sectors around $z_{0}$ \\
\hline
$z_{0}$ is an isolated essential & infinitely many elliptic & infinitely many elliptic \\
singularity of $f$ & and hyperbolic sectors & and hyperbolic sectors \\
& around $z_{0}$ & around $z_{0}$ \\
\hline
$z_{0}$ is a critical point & & \\
\hline
& & \\[-8pt]
& trajectory of $X_{f}$ & trajectory of $\widetilde{P_f}$ \\[2pt]
\hline
\end{tabular}
\end{center}
\label{phaseportraitXtable}
\end{table}%
}

Advantages: easily distinguishes key features of the function (zeros, poles, isolated essential singularities and 
accumulation points of any of the above), even without color plots.

Disadvantages: it is not global in nature: because of the different tensor type between $f$ and $X_{f}$ or 
$\widetilde{P_f}$, this technique will work on charts, but not necessarily on the whole Riemann surface $M$ 
where $f$ is defined. Moreover, the critical points of $f(z)$ are not immediately appreciated.

\subsubsection{Visualizing the functions $\Psi_X(z)$ on $M$ via the vector fields 
$X(z)=\frac{1}{\Psi_X^{\prime}(z)} \del{}{z}$ or $\widetilde{P}_{X}(z)=\Psi^{\prime}_X (z) \del{}{z}$ (via the 
singular analytic dictionary)}
An alternative that avoids the above mentioned problems is to use (1) of Theorem \ref{visualizationPsiPhi}.

\noindent
From the correspondence given by the singular complex analytic dictionary 
(Proposition \ref{basic-correspondence}) and 
Diagram \ref{diagrama-basico}, use 

\centerline{
\emph{$X(z)=\frac{1}{\Psi'_{X}(z)}\del{}{z}$ as a way of visualizing the complex functions $\Psi_{X}(z)$}.}

\noindent
It should be noted that this is in fact the visualization of the complex integral of the 1--form $\omega_{X}$ 
associated to the vector field $X(z)$.

\noindent
More precisely considering \eqref{lo-de-siempre-con-F} 
with $\Psi_{X}(z)$,  we observe that
\begin{equation}
\Psi_{X}(z) \longrightarrow X(z)
=\frac{1}{\Psi'_{X}(z)}\del{}{z}.
\end{equation}
Of course we can also consider the normalized Polya vector field 
\begin{equation}
\Psi_{X}(z) \longrightarrow \widetilde{P_X}(z)
=\Psi'_{X}(z)\del{}{z}. 
\end{equation}

Note that simple zeros of $\Psi_{X}(z)$ can not be distinguished since they correspond to regular points for the 
vector fields, however the critical points of $\Psi_{X}$ appear as poles of $X$ and zeros of $\widetilde{P_X}$, 
hence critical points of $\Psi_{X}$ can not be distinguished from multiple zeros of $\Psi_{X}$. 
See Figure \ref{Phiplots} and Table \ref{phaseportraitPsitable}.
\begin{figure}[htbp]
\begin{center}
\includegraphics[width=0.45\textwidth]{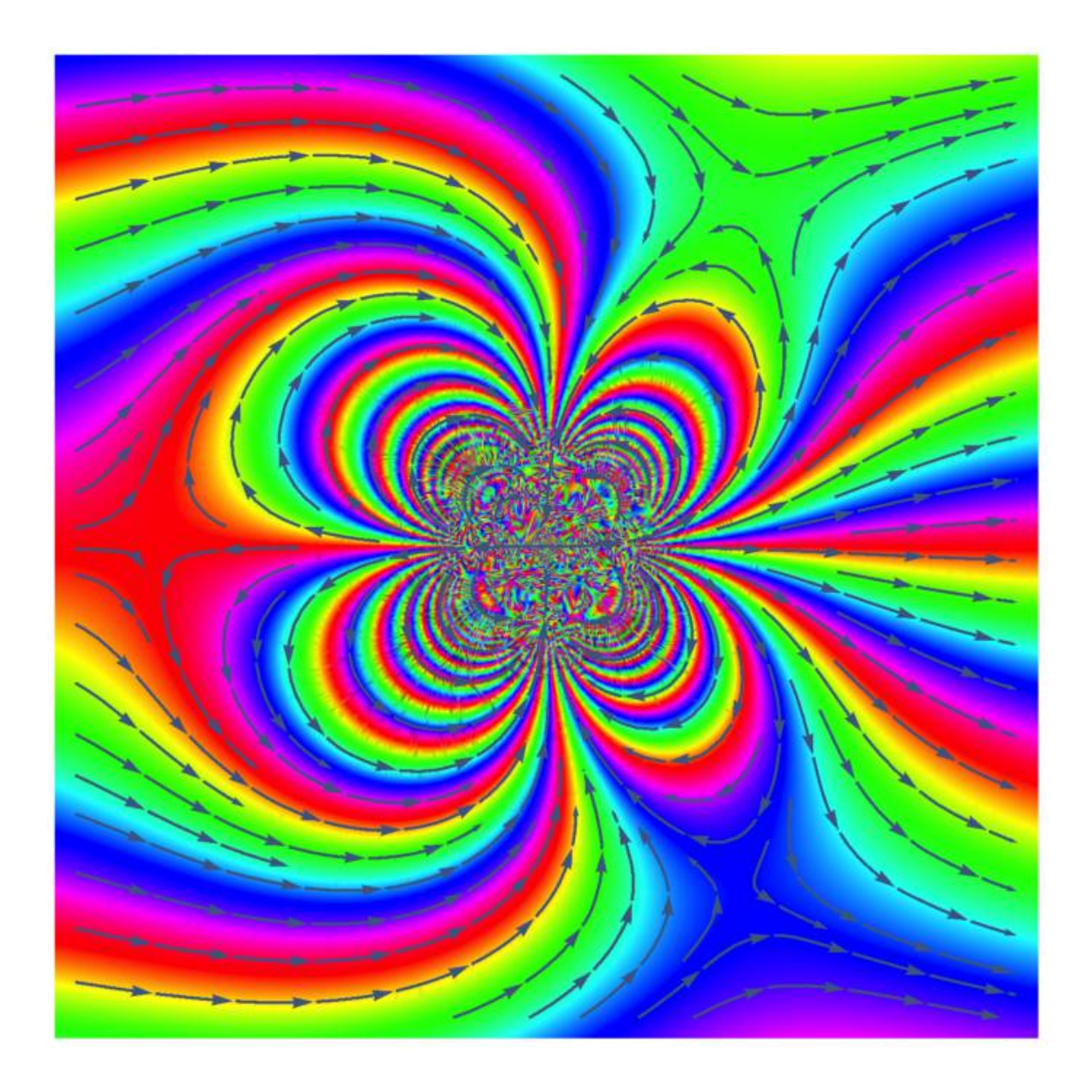}
\hskip 5pt
\includegraphics[width=0.45\textwidth]{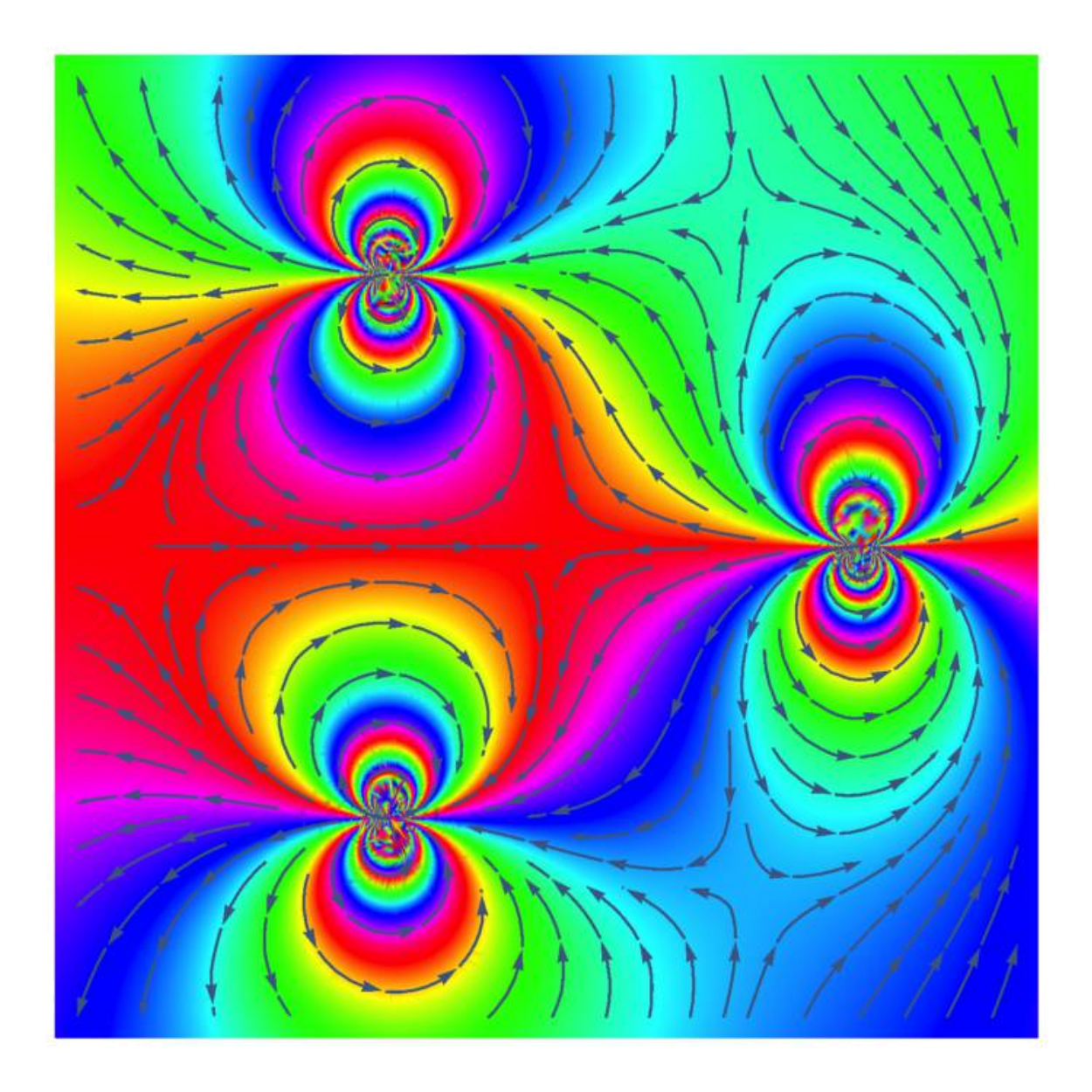}
$X_{1} (z)=\frac{1}{\Psi'_{X_1} (z)} \del{}{z}$
\qquad\qquad\qquad\qquad 
$X_{2} (z)=\frac{1}{\Psi'_{X_2} (z)} \del{}{z}$

\smallskip
\includegraphics[width=0.45\textwidth]{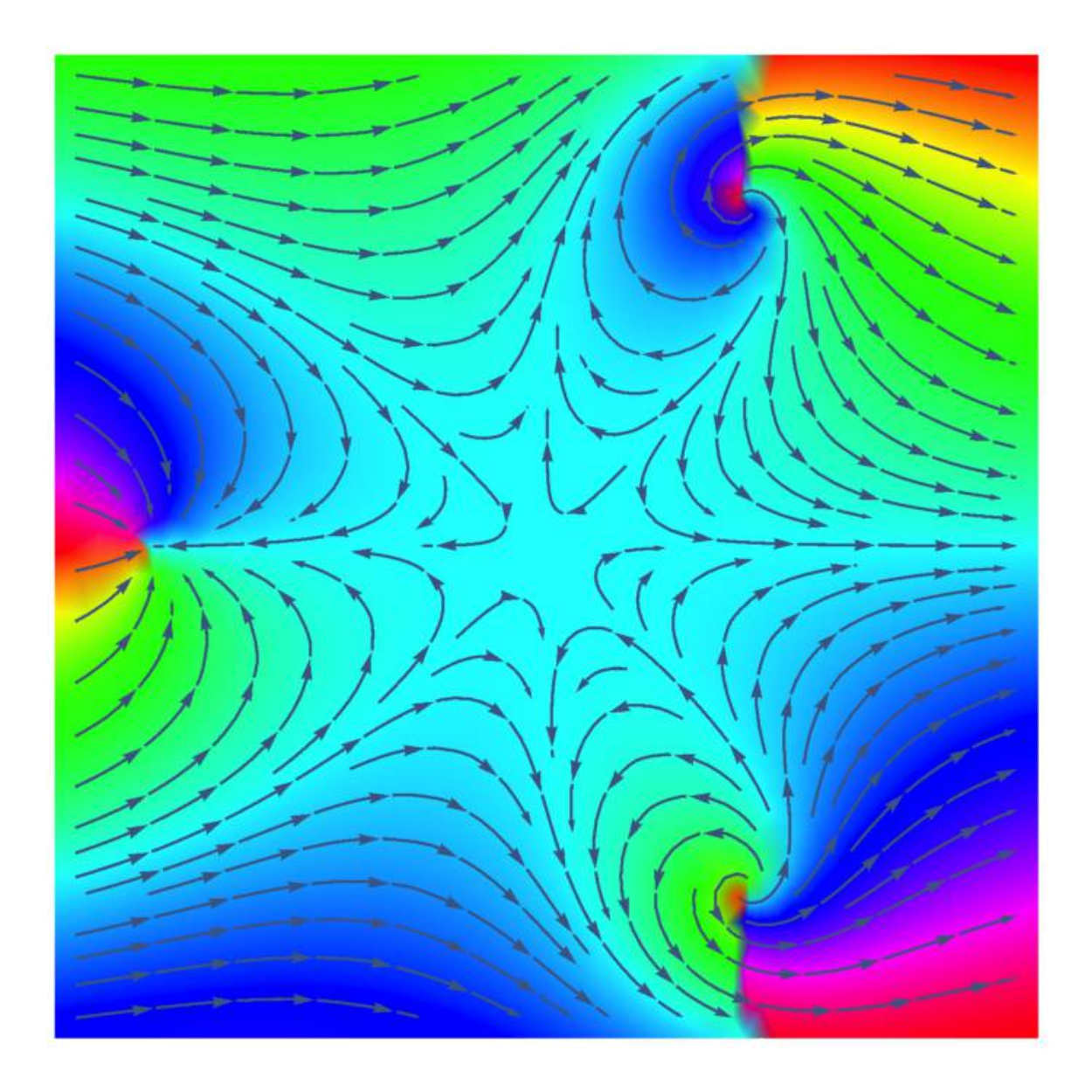}
\hskip 5pt
\includegraphics[width=0.45\textwidth]{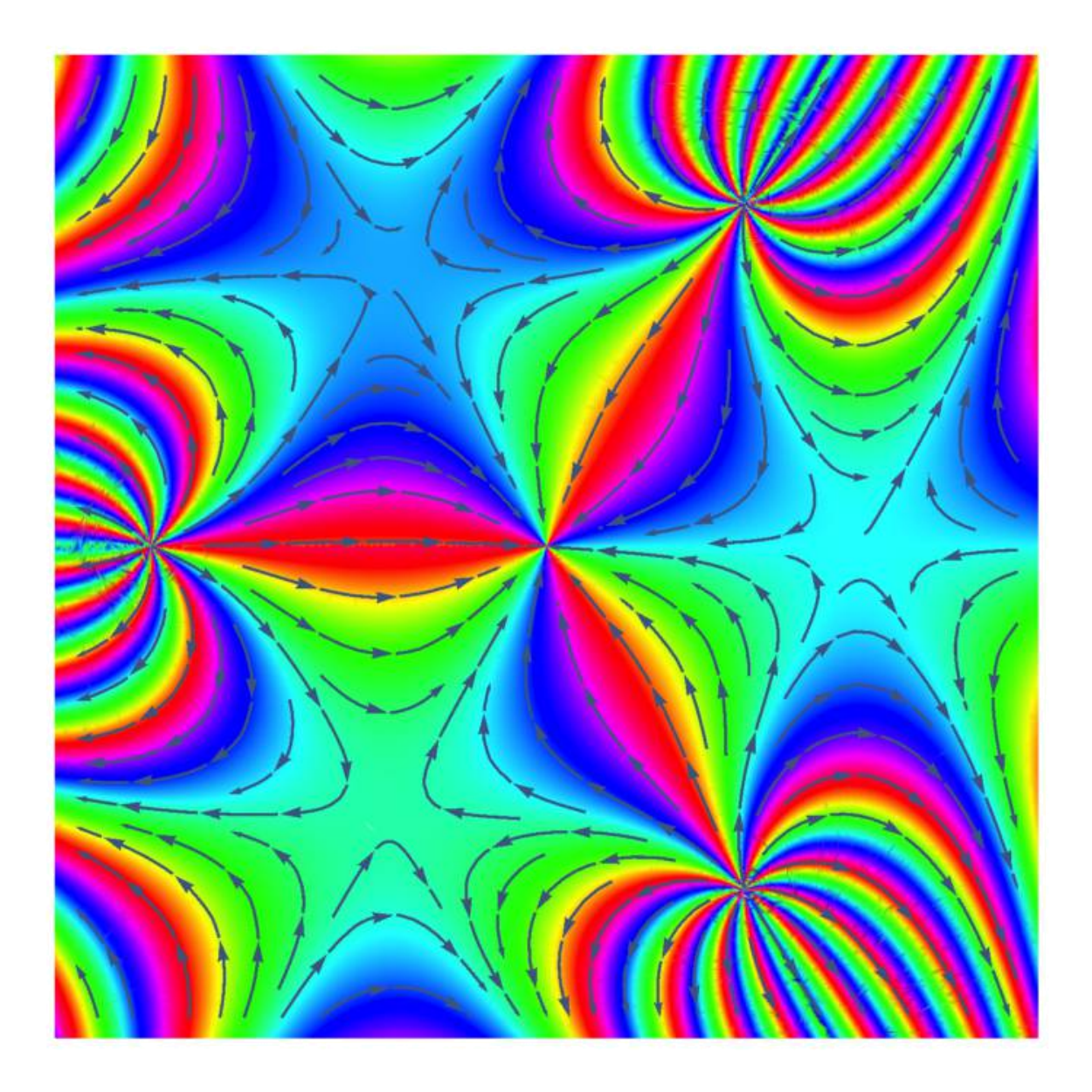}
$\widetilde{P_{X_1}} (z)=\Psi'_{X_1}(z) \del{}{z}$
 \qquad\qquad\qquad\qquad 
$\widetilde{P_{X_2}} (z)=\Psi'_{X_2}(z) \del{}{z}$
\caption{Visualization of the rational functions 
$\Psi_{X_1} (z)=\frac{z^3 -1}{z^2}$ and $\Psi_{X_2}(z)=\frac{z^2}{z^3 -1}$.
The phase portrait of the singular complex analytic vector fields
$X_{1} (z)=\frac{1}{\Psi'_{X_1} (z)} \del{}{z}$ and 
$X_{2} (z)=\frac{1}{\Psi'_{X_2} (z)} \del{}{z}$ is presented on the top row, and 
on the bottom row the phase portrait of the singular complex analytic vector fields
$\widetilde{P_{X_1}} (z)=\Psi'_{X_1}(z) \del{}{z}$ and 
$\widetilde{P_{X_2}} (z)=\Psi'_{X_2}(z) \del{}{z}$. 
In this case the information of the simple zeros of $\Psi_{X_1} (z)$ is lost and instead information on the location 
of the critical points of both $\Psi_{X_1} (z)$ and $\Psi_{X_2} (z)$ appears.}
\label{Phiplots}
\end{center}
\end{figure}
{\small
\begin{table}[htp]
\caption{Visualization of $\Psi_{X}$ via phase portrait of $X(z)=\frac{1}{\Psi'_{X}(z)}\del{}{z}$ and 
$\widetilde{P_X}(z)=\Psi'_{X}(z)\del{}{z}$.}
\begin{center}
\begin{tabular}{|c|c|c|}
\hline
& & \\[-6pt]
Function $\Psi_{X}$ & Phase portrait of & Phase portrait of \\
& $X(z)=\frac{1}{\Psi'_{X}(z)}\del{}{z}$ & $\widetilde{P_X}(z)=\Psi'_{X}(z)\del{}{z}$ \\
\hline
\hline
& & \\[-8pt]
$z_{0}$ is a simple zero of $\Psi_{X}$ & $z_{0}$ is a regular point of $X$ & $z_{0}$ is a regular point of 
$\widetilde{P_X}$ \\[2pt]
\hline
& & \\[-8pt]
$z_{0}$ is an order 2 zero of $\Psi_{X}$ & $z_{0}$ is a simple pole of $X$ & $z_{0}$ is a simple zero of 
$\widetilde{P_X}$ \\
& 4 hyperbolic angular & angular sector \\
& sectors around $z_{0}$ & corresponding to a center \\
$z_{0}$ is an order $s\geq3$ & $z_{0}$ is an order $-(s-1)\leq-2$ & $z_{0}$ is an order $s-1\geq2$ \\
zero of $\Psi_{X}$ & pole of $X$ & zero of $\widetilde{P_X}$ \\
& $2s$ hyperbolic angular & $2(s-2)$ elliptic angular \\
& sectors around $z_{0}$ & sectors around $z_{0}$ \\
\hline
$z_{0}$ is an order $-\kappa\leq-2$ & $z_{0}$ is an order $\kappa+1\geq2$ & $z_{0}$ is an order 
$-(\kappa+1)\leq-2$ \\
pole of $\Psi_{X}$ & zero of $X$ & pole of $\widetilde{P_X}$ \\
& $2\kappa$ elliptic angular & $2(\kappa-2)$ hyperbolic angular\\
& sectors around $z_{0}$ & sectors around $z_{0}$ \\
\hline
$z_{0}$ is an isolated essential & infinitely many elliptic and & infinitely many elliptic and \\
singularity of $\Psi$ & hyperbolic sectors & hyperbolic sectors \\
& around $z_{0}$ & around $z_{0}$ \\
\hline
& & \\[-8pt]
$z_{0}$ is a critical point of $\Psi_{X}$ of & $z_{0}$ is a pole of $X$ & $z_{0}$ is a zero of 
$\widetilde{P_X}$ \\[2pt]
ramification index $\mu+1\geq2$ & or order $-\mu\leq-1$ & f order $\mu\geq1$ \\
\hline
& & \\[-8pt]
& trajectory of $X(z)$ & trajectory of $\widetilde{P_X}(z)$ \\[2pt]
\hline
$z_{0}$ is a logarithmic & $z_{0}$ is a simple zero & $z_{0}$ is a simple pole \\
singularity of $\Psi_{X}$ & of $\Psi_{X}$ & of $\Psi_{X}$ \\
$\lambda \log(z-z_{0})$ & $\frac{1}{\lambda}(z-z_{0})\del{}{z}$ & $\lambda \frac{1}{z-z_{0}} \del{}{z}$ \\[4pt]
& angular sector & 4 hyperbolic angular \\
& corresponding to a center & sectors around $z_{0}$ \\
\hline
\end{tabular}
\end{center}
\label{phaseportraitPsitable}
\end{table}%
}

Advantages: has the correct tensor type behaviour, hence can be used globally on Riemann surfaces $M$.
Critical points of the function $\Psi$ appear as zeros of $X_{\Psi}$ and poles of $\widetilde{P_\Psi}$. 
Logarithmic singularities of $\Psi$ appear as simple poles of $X_{\Psi}$ and simple zeros of $\widetilde{P_\Psi}$

Disadvantages: simple zeros of $\Psi$ are invisible, furthermore there is no way to distinguish critical 
points of $\Psi$ from zeros of $\Psi$ (necessarily of order $\geq2$).

\section{Complex flows}\label{subsec:Flows}

Recall the following elementary fact. 

\begin{example}[Complete vector fields]\label{camposcompletos}
A singular complex analytic vector field on a Riemann surface 
$M$ is \emph{complete} if its flow is defined for all time
and for all initial condition. 
A pair $(M,X)$ determine a complete vector field on a Riemann 
surface if and only if they appear in the following list:
\begin{enumerate}[label=\arabic*)]
\item $\big(\CC,X=P(z)\del{}{z}\big)$ with $P$ a polynomial of degree at most 1,
\item $\big(\CW,X=P(z)\del{}{z}\big)$ with $P$ a polynomial of degree at most 2, 
\item $\big(\CC^{*}, X=\lambda z \del{}{z}\big)$, with $\lambda\in\CC^{*}$,
\item $\big(M=\CC/\Gamma, X=\lambda\del{}{z}\big)$, where $M$ is a complex torus.
\end{enumerate}
For a proof see \cite{LopezMucino}.
\end{example}

The singular complex analytic dictionary 
(Proposition \ref{basic-correspondence})
allows the description of 
the maximal domain for the flow even
in presence of poles, as we show in the following.   

\begin{example}[Example \ref{ejemplo-con-psi-racional}
revisited]

Let $\big(\CW,X=\frac{1}{R^\prime(z)}\del{}{z}\big)$ be a 
rational vector field, for $R(z)$ a rational function of degree at least two.
Then 

\centerline{$\Psi_X (z) =R(z) : \CW_z \longrightarrow \CC_t $}

\noindent 
is single valued and in accordance with Diagram \eqref{diagramaRX}.
In this case

\centerline{
$
\R_X = \Omega_X = \big\{ \big(z, R(z) \big) \ \vert \ z \in \CW \big\}. 
$}

\noindent 
The assertion, the global flow
$ \Psi^{-1}(t): \Omega_X \longrightarrow \CC_z$ of $X$
starting at an initial condition $z_0$, 
makes precise sense, since
$$
\Psi^{-1}:= \pi_{X, 1}: \R_X \longrightarrow \CC_z.
$$ 
\end{example}

More generally,
considering the maximal analytic continuation of the local flows,
a structure for the maximal analytic continuation can be  
recognized as follows.

\noindent 
Given $z_0 \in M^* $ (not a pole, essential singularity, or an 
accumulation point of poles, zeros or isolated
essential singularities of $X$),
the \emph{local flow} of $X$ at $z_0$ is a holomorphic map
$$
\begin{array}{rcl}
\varphi_{\tt j} (z,t) : \Omega(z_0,{\tt j}) \subset  M^* \times \CC_{t}
& \longrightarrow & M^* \\
 (z,t)& \longmapsto &
\left\{
\begin{array}{ll}
z_0 & \text{if } z = z_0 \text{ is a zero of } X, \\
& \vspace{-.3cm}  \\
\Psi_{\tt j }^{-1}(t) \ \ \ & \text{if }z \text{ is a non singular point of } X, 
\end{array}
\right.
\end{array}
$$
the second case is 
described as $z_{1} = \Psi_{\tt j } ^{-1} (t)$, where 
$$
\Psi_{\tt j}:
z \longmapsto \int_{z }^{z_1} \frac{d \zeta }{f_{\tt j} (\zeta) }  = t .
$$

\noindent
The integral 
is computed in a simply connected neighborhood
$V_{\tt j} \subset M^*$ of $z_0$,
thus $1/f_{\tt j}(\zeta)$ is holomorphic on $V_{\tt j}$,
hence the map 
$\{  \Psi_{\tt j} : z \longmapsto t \} $ is single valued.
However, the value of the integral
a priori depends on the choice of $V_{\tt j}$. Hence, we must fix $V_{\tt j}$ in order to
construct the local flow $\varphi _{\tt j} (z,t)$.
Making all these precisions, 
$\Omega(z_0,{\tt j}) \subset V_{\tt j} \times \CC_{t} $ is
an open set.

\theoremstyle{plain}
\newtheorem*{theo3}{Theorem \ref{FlujoMaximal}}

\begin{theo3}[Maximal domain for the flow]
Let $X$ be a singular complex analytic vector field on a 
Riemann surface $M$, and let $z_0 \in M \backslash Sing(X)$
be an initial condition. 
\begin{enumerate}[label=\arabic*)]
\item 
The maximal analytic continuation 
of the local flow 

\centerline{$\varphi_{\tt j} (z_0,t) 
: \{ z_0 \} \times ( \CC_t, 0)  \longrightarrow  M^*$}

\noindent 
is univalued on the 
Riemann surface $\R_{X} \subset M \times \CC_t$, which is the 
graph of 
$$
\Psi_{X} (z) = \int_{z_0}^z \omega_X : M^* \longrightarrow \CC_t .
$$ 

\item 
The Riemann surface $\R_X$ is a leaf of the  foliation 
$\mathcal{F}$ defined by the
complex analytic vector field
$$ 
f_{\tt j} (z) \del{}{z} + \del{}{t} \ \ \
 \hbox{ on }  M^* \times \CC_{t}
$$ 
and the changes of the initial conditions $z_0$ 
determine $t$--translations of $\R_X$.
\end{enumerate}
\end{theo3}

\begin{proof}
From the point of view of complex differential equations,
$\varphi_{\tt j} (z_0,t)= \Psi_{\tt j}^{-1}(t) 
:  \{ z_0 \} \times ( \mathbb{C}_t, 0)  \longrightarrow  M^*$ 
for $z_0 \in M^0$, is a local complex solution.
Hence
 the map 
$$\pi_{X,1}: \R_X \longrightarrow M,$$
together with an initial condition, $z_0 \in M^{0}$,
provides the maximal domain for the complex trajectory solution
of $X$ as an ordinary differential equation determined on $M$.
The Riemann surface $\R_X$ is a leaf of the 
complex analytic vector field
$$ 
f_{\tt j} (z) \del{}{z} + \del{}{t} \ \ \
 \hbox{ on }  M^* \times \CC_{t}.
$$ 
The singular complex
analytic foliation in the whole two dimensional
complex
manifold has as leaves copies
of $\R_X$ under translations
in the $\CC_t$ factor.
The maximal domain of the maximal analytic continuations
of  the flow for all $z_0 \in M^0$ is 
\begin{equation}\label{foliacion-de-t-translaciones}
\Omega_{X}= \bigcup_{t \ \hbox{translations}} \R_X
\end{equation}

\noindent 
that is the domain for the assertion (3). 
It follows that, 
the complex flow $\varphi (z,t) : \Omega_{X} \subset M^* \times \CC_{t} \to M^*$ is completely determined 
by copies of one trajectory by $\pi_{X,1}$, on each connected component of $M$.
\end{proof}

\begin{corollary}[From Riemann surfaces to vector fields]\label{DeRaX}
Consider $\CC_{t}$ provided with the vector field $\del{}{t}$.
The following statements are equivalent.
\begin{enumerate}[label=\arabic*)]
\item An arbitrary Riemann surface $\R \subset M \times \CC_{t}$
determines a singular complex analytic vector field $X$ 
on $M$, following Diagram \eqref{diagramaRX}.
\item $\R$ is the graph of an additively automorphic map $\Psi$ as in Diagram \eqref{diagramaRX}. 

\item 
The decomposition

\centerline{
$\Omega_{X}=\bigcup\limits_{t \ \hbox{translations}} \R \subset M^{*} \times \CC_{t}$
}

\smallskip
\noindent
determines a holomorphic foliation $\mathcal{F}$ of $\Omega_{X}$, where the 
leaves of $\mathcal{F}$
do not contain plaques of the form 
$D(z_0,r)\times \{t_0\}$, for $r>0$.
\end{enumerate}
\end{corollary}

\begin{proof}
Note that, $\Psi$ is 
an additively automorphic map if and only if
its graph $\R$ satisfies
Equation \eqref{foliacion-de-t-translaciones}.
Hence it satisfies Diagram \ref{diagramaRX} for an appropriate $X$.

\noindent
Allowing plaques of the form 
$\{z_0\} \times D(t_0,r)$, for $r>0$, in the leaves of $\mathcal{F}$
gives rise to $X(z_0)= 0$ on $M$.
Not allowing plaques of the form $D(z_0,r) \times \{t_0\}$ 
ensures that $X$ is not identically $\infty$.
A plaque different to any of the above two cases in a leaf 
of $\mathcal{F}$
determines locally a non identically constant $X$ on $M$.
\end{proof}

\bibliographystyle{spmpsci}      


\end{document}